\documentclass[11pt]{amsart}

\usepackage{graphicx}
\usepackage{mathptmx}

\usepackage{amscd}
\usepackage{amsxtra}
\usepackage{a4wide}
\usepackage{latexsym}
\usepackage{amssymb}
\usepackage{amsfonts}
\usepackage{amsmath}
\usepackage{amsrefs}
\usepackage{mathrsfs}
\usepackage{upref}
\usepackage{txfonts}
\usepackage{color}
\usepackage{tcolorbox}
\usepackage[utf8]{inputenc}

\usepackage[bookmarksnumbered, colorlinks, plainpages]{hyperref}
\hypersetup{colorlinks=true,linkcolor=red, anchorcolor=green,
citecolor=cyan, urlcolor=red, filecolor=magenta, pdftoolbar=true}

\usepackage[left=2cm, right=2cm, top=2cm, bottom=2cm]{geometry}

\allowdisplaybreaks

\newtheorem{theorem}{Theorem}[section]
\newtheorem{lemma}[theorem]{Lemma}

\theoremstyle{definition}

\theoremstyle{remark}

\numberwithin{equation}{section}

\def\GG{\operatorname{G\Gamma}}

\def\esup{\operatornamewithlimits{ess\,sup}}

\def\R{\mathbb R}

\def\ap{\approx}

\def\mf{\mathfrak M}
\def\rad{\operatorname{rad}}

\def\qq{\qquad}
\def\rn{\R^n}

\def\la{\lambda}

\def\vp{\varphi}

\def\rw{\rightarrow}

\def\dn{\downarrow}

\def\R{\mathbb R}

\def\M{\mathfrak M}

\def\mp{{\mathfrak M}}
\def\W{{\mathcal W}}

\begin{document}

\setcounter{page}{1}

\title[On some restricted inequalities for $T_{u,b}$ and their applications]{On some restricted inequalities for the iterated Hardy-type operator involving suprema and their applications}

\author[R.Ch. Mustafayev]{RZA MUSTAFAYEV}
\address{RZA MUSTAFAYEV, Department of Mathematics, Kamil \"{O}zda\u{g} Faculty of Science, Karamano\u{g}lu Mehmetbey University, 70200, Karaman, Turkey}
\email{rzamustafayev@gmail.com}

\author[N. B\.{I}LG\.{I}\c{C}L\.{I}]{NEV\.{I}N B\.{I}LG\.{I}\c{C}L\.{I}}
\address{NEV\.{I}N B\.{I}LG\.{I}\c{C}L\.{I}, Republic of Turkey Ministry of National Education, Kirikkale High School, 71100, Kirikkale, Turkey}
\email{nevinbilgicli@gmail.com}

\author[M. Yilmaz]{MERVE YILMAZ}
\address{MERVE YILMAZ, Department of Mathematics, Kamil \"{O}zda\u{g} Faculty of Science, Karamano\u{g}lu Mehmetbey University, 70200, Karaman, Turkey}
\email{mervegorgulu@kmu.edu.tr}



\begin{abstract}
	In this paper  we characterize the inequality
	\begin{equation*}
	\bigg( \int_0^{\infty} \bigg( \int_0^x \big[ T_{u,b}f^* (t)\big]^r\,dt\bigg)^{\frac{q}{r}} w(x)\,dx\bigg)^{\frac{1}{q}} \le C \, \bigg( \int_0^{\infty} \bigg( \int_0^x [f^* (\tau)]^p\,d\tau \bigg)^{\frac{m}{p}} v(x)\,dx \bigg)^{\frac{1}{m}}
	\end{equation*}
	for $1 < m < p \le r < q < \infty$ or $1 < m \le r < \min\{p,q\} < \infty$, where $w$ and $v$ are weight functions on  $(0,\infty)$. The inequality is required to hold with some positive constant $C$ for all measurable functions defined on measure space $({\mathbb R}^n,dx)$. Here $f^*$ is the non-increasing rearrangement of a measurable function $f$ defined on $\rn$ and $T_{u,b}$ is the iterated Hardy-type operator involving suprema, whish is defined for a measurable non-negative function $f$ on $(0,\infty)$ by 
	$$
	(T_{u,b} g)(t) : = \sup_{t \le \tau < \infty} \frac{u(\tau)}{B(\tau)} \int_0^{\tau} g(s)b(s)\,ds,\qquad t \in (0,\infty),
	$$ 
	where $u$ and $b$ are two weight functions on $(0,\infty)$ such that $u$ is continuous on $(0,\infty)$ and the function 
	$B(t) : = \int_0^t b(s)\,ds$ satisfies  $0 < B(t) < \infty$ for every $t \in (0,\infty)$. 
	
	At the end of the paper, as an application of obtained results, we calculate the norm of the generalized maximal operator $M_{\phi,\Lambda^{\alpha}(b)}$,
	defined with $0 < \alpha < \infty$ and functions $b,\,\phi: (0,\infty) \rightarrow (0,\infty)$ for all measurable functions $f$ on ${\mathbb R}^n$ by
	\begin{equation*}
	M_{\phi,\Lambda^{\alpha}(b)}f(x) : = \sup_{Q \ni x} \frac{\|f \chi_Q\|_{\Lambda^{\alpha}(b)}}{\phi (|Q|)}, \qquad x \in \rn,
	\end{equation*}
	from $\GG(p_1,m_1,v)$ into $\GG(p_2,m_2,w)$. Here $\Lambda^{\alpha}(b)$ and $\GG(p,m,w)$ are the classical and generalized	Lorentz spaces, defined as a set of all measurable functions $f$ defined on $\rn$ for which 
	$$
	\|f\|_{\Lambda^{\alpha}(b)} = \bigg( \int_0^{\infty} [f^*(s)]^{\alpha} b(s)\,ds \bigg)^{\frac{1}{\alpha}} < \infty
	\quad \mbox{and} \quad \|f\|_{\GG(p,m,w)} = \bigg( \int_0^{\infty} \bigg( \int_0^x [f^* (\tau)]^p\,d\tau \bigg)^{\frac{m}{p}} v(x)\,dx \bigg)^{\frac{1}{m}} < \infty,
	$$
	respectively.
\end{abstract}

\maketitle


\section{Introduction}\label{in}

Let $({\mathcal R}, \mu)$ be a $\sigma$-finite non-atomic measure space. Denote by
$\mf({\mathcal R})$ the set of all $\mu$-measurable functions on ${\mathcal R}$
and $\mf_0 ({\mathcal R})$ the class of functions in $\mf ({\mathcal R})$ that
are finite $\mu$ - a.e. on ${\mathcal R}$. The symbol $\mp^+ ({\mathcal R})$ stands for the
collection of all $f\in\mp ({\mathcal R})$ which are non-negative on
${\mathcal R}$. 

The non-increasing rearrangement $f^*$ of $f \in \mp_0({\mathcal R})$ is given by
$$
f^* (t) = \inf \big\{ \lambda \ge 0:\, \mu (\{x \in {\mathcal R}:\,|f(x)| > \lambda\}) \le t \big\}, \qquad t \in [0,\mu({\mathcal R})).
$$
The maximal non-increasing rearrangement of $f$ is defined as follows 
\begin{equation}\label{max.func.decr.rear.}
f^{**}(t) : = \frac{1}{t} \int_0^t f^* (\tau)\,d\tau, \qquad t \in (0,\mu({\mathcal R})).
\end{equation}

Most of the functions which we shall deal with will be defined on ${\mathbb R}^n$ or $(0,\infty)$. In this case, $({\mathcal R}, \mu)$ is ${\mathbb R}^n$ or $(0,\infty)$ endowed with the $n$-dimensional Lebesgue measure or the one-dimensional Lebesgue measure, respectively. We shall write just $\mp^+$ instead of $\mp^+ (0,\infty)$.

Let $\Omega$ be any measurable subset of $\rn$, $n\geq 1$. The family of all weight functions (also called just weights) on $\Omega$, that is, locally integrable non-negative functions on $\Omega$, denoted by $\W(\Omega)$.

For $p\in (0,\infty]$ and $w\in \mp^+(\Omega)$, we define the functional
$\|\cdot\|_{p,w,\Omega}$ on $\mp (\Omega)$ by
\begin{equation*}
\|f\|_{p,w,\Omega} : = \left\{\begin{array}{cl}
\left(\int_{\Omega} |f(x)|^p w(x)\,dx \right)^{1/p} & \qq\mbox{if}\qq p<\infty, \\
\esup_{\Omega} |f(x)|w(x) & \qq\mbox{if}\qq p=\infty.
\end{array}
\right.
\end{equation*}

If, in addition, $w\in \W(\Omega)$, then the weighted Lebesgue space
$L^p(w,\Omega)$ is given by
\begin{equation*}
L^p(w,\Omega) = \{f\in \mp (\Omega):\,\, \|f\|_{p,w,\Omega} <
\infty\}
\end{equation*}
and it is equipped with the quasi-norm $\|\cdot\|_{p,w,\Omega}$.

When $w\equiv 1$ on $\Omega$, we write simply $L^p(\Omega)$ and
$\|\cdot\|_{p,\Omega}$ instead of $L^p(w,\Omega)$ and
$\|\cdot\|_{p,w,\Omega}$, respectively.

Quite many familiar function spaces can be defined using the
non-increasing rearrangement of a function. One of the most
important classes of such spaces are the so-called classical Lorentz
spaces.

Let $p \in (0,\infty)$ and $w \in {\mathcal W}(0,\mu({\mathcal R}))$. Then the
classical Lorentz spaces $\Lambda^p (w)$ and $\Gamma^p (w)$ consist
of all functions $f \in {\mathfrak M}({\mathcal R})$ for which
$$
\|f\|_{\Lambda^p(w)} : = \bigg( \int_0^{\mu({\mathcal R})} [f^*(s)]^p w(s)\,ds \bigg)^{\frac{1}{p}} < \infty
\quad \mbox{and} \quad \|f\|_{\Gamma^p(w)} : = \bigg( \int_0^{\mu({\mathcal R})} [f^{**}(s)]^p w(s)\,ds \bigg)^{\frac{1}{p}} < \infty,
$$
respectively. For more information about the Lorentz $\Lambda$ and $\Gamma$ see
e.g. \cite{cpss} and the references therein.

The study of particular problems in the regularity theory of PDE's led to the definition of spaces involving inner integral means involving powers of the non-increasing rearrangements of functions.
The generalized Lorentz $\GG(p,m,v)({\mathcal R,\mu})$ space (denoted simply by $\GG(p,m,v)$), introduced and studied in \cite{FR2008} and \cite{FRZ2009}, is defined as the collection of all $g \in \mp (\mathcal R)$ such that
$$
\|g\|_{\GG(p,m,v)} = \bigg( \int_0^{\mu({\mathcal R})} \bigg( \int_0^x [g^* (\tau)]^p\,d\tau \bigg)^{\frac{m}{p}} v(x)\,dx \bigg)^{\frac{1}{m}} < \infty,
$$
where $m,\,p \in (0,\infty)$, $v \in \W(0,\mu({\mathcal R}))$.

The spaces $\GG(p,m,v)$ cover several types of important function spaces and have plenty of applications. For example, when $\mu({\mathcal R}) = \infty$, $p = 1$, $m > 1$ and $v(t) = t^{-m}w(t)$, $t \in (0,\infty)$, where $w$ is another weight on $(0,\infty)$, then $\GG(p,m,v)$ reduces the spaces $\Gamma^m(w)$. Another important example is obtained when $\mu({\mathcal R}) = 1$, $m = 1$, $p \in (1,\infty)$ and $v(t) = t^{-1}\big( \log ({2} / {t})\big)^{-1/p}$ for $t \in (0,1)$. In this case the space $\GG(p,m,v)$ coincides with the small Lebesgue space, which was originally studied by Fiorenza in \cite{F2000}. In the same paper it was proved that this space is the associate space of the grand Lebesgue space introduced in \cite{IS} in connection with integrability properties of Jacobians. Subsequently, Fiorenza and Karadzhov in \cite{FK} derived an equivalent form of the norm in the small Lebesgue space written in the form of the norm in the $\GG(p,m,v)$ space with the above mentioned parameters and weight. Note that the condition $\int_{(0,1)} t^{m / p} v(t)\,dt < \infty$ ensures $\GG(p,m,v)$ to be a quasi-Banach function space when ${\mathcal R} \subset \rn$ with $\mu ({\mathcal R}) = 1$ (cf. \cite{FFGKR}). Recently, the connection of the $\GG(p,m,v)$ space with some well-known function spaces have been studied in \cite{AFFGR}. In the present paper we take $(\rn,dx)$ as underlying measure space and use the notation $\GG(p,m,v)$ for $\GG(p,m,v)(\rn,dx)$.  

The study on maximal operators occupies an important place in harmonic analysis. These significant non-linear operators, whose
behavior are very informative in particular in differentiation
theory, provided the understanding and the inspiration for the
development of the general class of singular and potential operators
(see, for instance, \cites{stein1970,guz1975,GR,tor1986,stein1993,graf2008,graf}).

The main example is the Hardy-Littlewood maximal function which is
defined  for locally integrable functions $f$ on $\rn$ by
$$
(Mf)(x) : = \sup_{Q \ni x}\frac{1}{|Q|} \int_Q |f(y)|\,dy, \qquad x \in \rn,
$$
where the supremum is taken over all cubes $Q$ containing $x$. By a cube, we mean an open cube with sides parallel to the coordinate axes.

Another important example is the fractional maximal operator, $M_{\gamma}$, $\gamma \in (0,n)$,
defined  for locally integrable functions $f$ on $\rn$ by
$$
(M_{\gamma} f) (x) := \sup_{Q \ni x} |Q|^{ \gamma / n - 1} \int_{Q}
|f(y)|\,dy,\qquad x \in \rn.
$$

One more example is the fractional maximal operator $M_{s,\gamma,\mathbb A}$ defined in \cite{edop} for all measurable functions $f$ on $\rn$ by
$$
(M_{s,\gamma,\mathbb A} f) (x) : = \sup_{Q \ni x}
\frac{\|f\chi_Q\|_s}{ \|\chi_Q\|_{sn / (n-\gamma),{\mathbb A}}},
\qquad x \in \rn.
$$
Here $s \in (0,\infty)$, $\gamma \in [0,n)$, $\mathbb A = (A_0,A_{\infty}) \in {\mathbb R}^2$ and
$$
\ell^{\mathbb A} (t) : = (1 + |\log t|)^{A_0} \chi_{[0,1]}(t) + (1 +
|\log t|)^{A_{\infty}} \chi_{[1,\infty)}(t),  \qquad t \in (0,\infty).
$$
Recall that the following equivalency holds:
$$
(M_{s,\gamma,\mathbb A} f) (x) \ap \sup_{Q \ni x}
\frac{\|f\chi_Q\|_s}{ |Q|^{(n-\gamma)/(sn)}\ell^{\mathbb A}(|Q|)},
\quad x \in \rn.
$$

Hence, if $s = 1$, $\gamma = 0$ and ${\mathbb A} = (0,0)$, then
$M_{s,\gamma,\mathbb A}$ is equivalent to $M$. If $s = 1$, $\gamma \in
(0,n)$ and ${\mathbb A} = (0,0)$, then $M_{s,\gamma,\mathbb A}$ is
equivalent to $M_{\gamma}$. Moreover, if $s = 1$, $\gamma \in [0,n)$ and ${\mathbb A} \in \R^2$,
then $M_{s,\gamma,\mathbb A} $ is the fractional maximal operator
which corresponds to potentials with logarithmic smoothness treated
in \cites{OT1,OT2}. In particular, if $\gamma = 0$, then
$M_{1,\gamma,\mathbb A}$ is the maximal operator of purely
logarithmic order.

Given $p$ and $q$, $0 < p,\,q < \infty$, let $M_{p,q}$ denote the
maximal operator associated to the Lorentz $L^{p,q}$ spaces defined for all measurable function $f$ on $\rn$ by
$$
M_{p,q} f (x) : = \sup_{Q \ni x} \frac{\|f \chi_Q\|_{p,q}}{\|\chi_Q\|_{p,q}},
$$
where $\|\cdot\|_{p,q}$ is the usual Lorentz norm
$$
\|f\|_{p,q} : = \bigg( \int_0^{\infty} \big[ \tau^{1 / p} f^*
(\tau)\big]^q \frac{d\tau}{\tau}\bigg)^{1 / q}.
$$
This operator was introduced by Stein in \cite{stein1981} in order
to obtain certain endpoint results in differentiation theory. The
operator $M_{p,q}$ have been also considered by other authors, for
instance see  \cite{neug1987,leckneug,basmilruiz,perez1995,ler2005}.

Let $0 < \alpha < \infty$, $b \in \W(0,\infty)$ and $\phi: (0,\infty)
\rightarrow (0,\infty)$. Recall the definition of the generalized maximal function introduced in \cite{musbil} and denoted for all measurable function $f$ on $\rn$ by
\begin{equation}
M_{\phi,\Lambda^{\alpha}(b)}f(x) : = \sup_{Q \ni x} \frac{\|f \chi_Q\|_{\Lambda^{\alpha}(b)}}{\phi (|Q|)}, \qquad x \in \rn.
\end{equation}

Obviously, $M_{\phi,\Lambda^{\alpha}(b)} = M$, where $M$ is the Hardy-Littlewood maximal operator, when $\alpha = 1$,
$b \equiv 1$ and $\phi (t) = t$ $(t >0)$.
Note that $M_{\phi,\Lambda^{\alpha}(b)} = M_{\gamma}$, where
$M_{\gamma}$ is  the fractional maximal operator, when $\alpha = 1$,
$b \equiv 1$ and $\phi (t) = t^{1 - \gamma / n}$ $(t >0)$ with $0 <
\gamma < n$. Moreover, $M_{\phi,\Lambda^{\alpha}(b)} \ap
M_{s,\gamma,\mathbb A}$, when $\alpha = s$, $b \equiv 1$ and $\phi
(t) = t^{(n - \gamma) / (sn)} \ell^{\mathbb A} (t)$ $(t >0)$ with
$0 < \gamma < n$ and $\mathbb A = (A_0,A_{\infty}) \in {\mathbb
	R}^2$. It is worth also to mention that
$M_{\phi,\Lambda^{\alpha}(b)} = M_{p,q}$, when $\alpha = q$, $b(t) =
t^{q/ p - 1}$ and $\phi (t) = t^{1 / p}$ $(t >0)$.

The boundedness of $M_{\phi,\Lambda^{\alpha}(b)}$ between classical Lorentz spaces $\Lambda$ was completely characterized in \cite{musbil}. In view of importance of generalized Lorentz spaces in our opinion it will be interesting to obtain a characterization of the boundedness of this maximal function between Lorentz $\GG$ spaces.

The iterated Hardy-type operator involving suprema $T_{u,b}$ is defined for non-negative measurable function $g$ on the interval $(0,\infty)$ by 
$$
(T_{u,b} g)(t) : = \sup_{t \le \tau < \infty}
\frac{u(\tau)}{B(\tau)} \int_0^{\tau} g(y)b(y)\,dy,\qquad t \in (0,\infty),
$$
where $u$ and $b$ are two weight functions on $(0,\infty)$ such that $u$ is continuous on $(0,\infty)$ and the function $B(t) : = \int_0^t b(s)\,ds$ satisfies  $0 < B(t) < \infty$ for every $t \in (0,\infty)$. 
Such operators have been found indispensable in the search for optimal pairs of rearrangement-invariant norms for which a Sobolev-type inequality holds (cf. \cite{kerp}). They constitute a very useful tool for characterization of the associate norm of an
operator-induced norm, which naturally appears as an optimal domain norm in a Sobolev embedding (cf. \cite{pick2000, pick2002}). Supremum operators are also very useful in limiting interpolation theory as can be seen from their appearance for example in \cite{cwikpys, dok, evop, pys}.

The aim of the paper is to characterize the following restricted inequality for $T_{u,b}$:
\begin{equation}\label{main.ineq.}
\bigg( \int_0^{\infty} \bigg( \int_0^x \big[ T_{u,b}f^* (t)\big]^r\,dt\bigg)^{\frac{q}{r}} w(x)\,dx\bigg)^{\frac{1}{q}} \le C \, \bigg( \int_0^{\infty} \bigg( \int_0^x [f^* (\tau)]^p\,d\tau \bigg)^{\frac{m}{p}} v(x)\,dx \bigg)^{\frac{1}{m}}.
\end{equation}
Here $m,\,p,\,q,\,r$ are positive real numbers and $w,\,v$ are weight functions on  $(0,\infty)$. 

A function $\phi: (0,\infty) \rightarrow (0,\infty)$ is said to satisfy the $\Delta_2$ - condition, denored $\phi \in \Delta_2$, if for some $C > 0$
$$
\phi (2t) \le C \, \phi (t) \quad \mbox{for all} \quad 0 < t < \infty.
$$

A function $\phi: (0,\infty) \rightarrow (0,\infty)$ is said to be
quasi-increasing, if for some $C > 0$
$$
\phi (t_1) \le C \phi (t_2),
$$
whenever $0 < t_1 \le t_2 < \infty$.

A function $\phi: (0,\infty) \rightarrow (0,\infty)$ is said to
satisfy the $Q_r$-condition, $0 < r < \infty$, denoted $\phi \in
Q_r(0,\infty)$, if for some constant $C > 0$
$$
\phi \bigg(\sum_{i=1}^n t_i \bigg) \le C \bigg( \sum_{i=1}^n
\phi(t_i)^r \bigg)^{1/r},
$$
for every finite set of non-negative real numbers
$\{t_1,\ldots,t_n\}$.

Motivation for studying inequality \eqref{main.ineq.} comes directly from the following eqiuvalency statement.
\begin{theorem}\label{main.reduc.thm}
	Let $0 < p_1,\,p_2 < \infty$, $0 < m_1,\,m_2 < \infty$, $0 < \alpha \le r < \infty$ and $v,\,w \in \W (0,\infty)$. Assume that $\phi \in Q_{r}(0,\infty)$ is a quasi-increasing function. Suppose that $b \in \W (0,\infty)$ is such that $0 < B(t) < \infty$ for all $t > 0$, $B \in \Delta_2$, $B(\infty) = \infty$ and $B(t) / t^{\alpha / r}$ is quasi-increasing. Then the inequality 
	\begin{equation}\label{opnorm of M}
	\bigg( \int_0^{\infty} \bigg( \int_0^x \big[ \big(M_{\phi,\Lambda^{\alpha}(b)}f\big)^* (t)\big]^{p_2}\,dt\bigg)^{\frac{m_2}{p_2}} w(x)\,dx\bigg)^{\frac{1}{m_2}} \le C \bigg( \int_0^{\infty} \bigg( \int_0^x [f^* (\tau)]^{p_1}\,d\tau \bigg)^{\frac{m_1}{p_1}} v(x)\,dx \bigg)^{\frac{1}{m_1}}
	\end{equation}
	holds for all $f \in \mp (\rn)$ if and only if the inequality
	$$
	\bigg( \int_0^{\infty} \bigg( \int_0^x \big[ T_{B/\phi^{\alpha},b} h^* (t) \big]^{\frac{p_2}{\alpha}}\,dt\bigg)^{\frac{m_2}{p_2}} w(x)\,dx\bigg)^{\frac{1}{m_2}} \le C \bigg( \int_0^{\infty} \bigg( \int_0^x [h^* (\tau)]^{\frac{p_1}{\alpha}}\,d\tau \bigg)^{\frac{m_1}{p_1}} v(x)\,dx \bigg)^{\frac{1}{m_1}}
	$$
	holds for all $h \in \mp (\rn)$.
\end{theorem}

The method used for solution of inequalty \eqref{main.ineq.} is based on the combination of the duality techniques with the formula 
\begin{equation*}\label{transmon}
\sup_{g:\,\int_0^x g \le \int_0^x f} \int_0^{\infty} g(x)w(x)\,dx = \int_0^{\infty} f(x) \bigg( \sup_{t \ge x} w(t)\bigg)\,dx
\end{equation*}
from \cite{Sinn}, which holds for $f,\,w \in \mp^+(0,\infty)$. On the other hand, it uses estimates of optimal constants in iterated Hardy-type inequalities, "gluing" lemmas, which allows to reduce the problem to using of the boundedness of weighted iterated Hardy-type operators involving suprema from weighted Lebesgue spaces into weighted Ces\`{a}ro function spaces. Detailed information on materials that are used in the proofs of the main results is given in the following section.

It should be noted that the method developed in the present paper allows to solve inequality \eqref{main.ineq.} in the case $0 < m \le 1$, $0 < p < \infty$ or $1 < m < \infty$, $0 < p \le 1$, as well. However, we are able to solve the inequality under the restrictions $1 < m < p \le r < q < \infty$ or $1 < m \le r < \min\{p,q\} < \infty$. For the remaining values of parameters the conditions that characterize the weighted iterated Hardy-type inequalities contains more complicated expressions and the approach used in our paper needs an improvement.

Using duality principle puts restriction $r \le q$ on parameters. But for $r = q$ inequality \eqref{main.ineq.} is a special case of the inequality
\begin{equation}
\bigg( \int_0^{\infty}  \big[ T_{u,b}f^* (x)\big]^q \, w(x)\,dx\bigg)^{\frac{1}{q}} \le C \, \bigg( \int_0^{\infty} \bigg( \int_0^x [f^* (\tau)]^p\,d\tau \bigg)^{\frac{m}{p}} v(x)\,dx \bigg)^{\frac{1}{m}},
\end{equation}
which we are going to investigate in our forthcoming studies.

Throughout the paper, we always denote by  $C$ a positive
constant, which is independent of main parameters but it may vary
from line to line. However a constant with subscript such as $C_1$
does not change in different occurrences. By $a\lesssim b$,
we mean that $a\leq \la b$, where $\la >0$ depends on
inessential parameters. If $a\lesssim b$ and $b\lesssim a$, we write
$a\approx b$ and say that $a$ and $b$ are  equivalent. 

As usual, we put $0 \cdot \infty = 0$, $\infty / \infty =0$ and $0/0 = 0$. If $p\in [1,+\infty]$, we define $p'$ by $1/p + 1/p' = 1$.

The paper is organized as follows. We start with formulations of background material in Section \ref{BM}. In Section \ref{MR} we present solution of the restricted inequality. Finally, in Section \ref{BofMF} we calculate the norm of generalized maximal function between Lorentz $\GG$ spaces.




\section{Background material}\label{BM}

In this section we collect background material that will be used in the proofs of the main theorems.

We begin with the following characterization of the norm of the associate space of $\GG(p,m,v)$ given in \cite{GPS}. Recall that the associate space $\GG(p,m,v)^{\prime}$ of $\GG(p,m,v)$ is defined as the collection of all functions $g \in \mp (\rn)$ such that
$$
\|g\|_{\GG(p,m,v)^{\prime}} = \sup_{\|f\|_{\GG(p,m,v)} \le 1} \int_0^{\infty} f^* (t) g^* (t)\,dt < \infty.
$$
As it is mentioned in \cite{GPS}, it is reasonable to adopt a general assumption that $p,\,m$ and $v$ are such that
\begin{equation}\label{nontriv}
\int_0^t v(s)s^{\frac{m}{p}}\,ds + \int_t^{\infty} v(s)\,ds < \infty,  \qquad t \in (0,\infty), 
\end{equation}
because if this requirement is not satisfied, then $\GG(p,m,v) = \{0\}$. 

Under the assumption \eqref{nontriv}, we denote
\begin{equation}\label{defof_v}
	v_0 (t) : = t^{\frac{m}{p} - 1}\int_0^t v(s)s^{\frac{m}{p}}\,ds \int_t^{\infty} v(s)\,ds, \qquad t \in (0,\infty),
\end{equation}
and
\begin{equation}\label{defof_u}
	v_1(t) : = \int_0^t v(s)s^{\frac{m}{p}}\,ds + t^{\frac{m}{p}} \int_t^{\infty} v(s)\,ds, \qquad t \in (0,\infty).
\end{equation}

Moreover, we assume that a weight $v$ is non-degenerate (with respect to the power function $t^{m / p}$), that is,
\begin{equation}\label{nondegen}
\int_0^1 v(s)\,ds = \int_1^{\infty} v(s)s^{\frac{m}{p}}\,ds = \infty.
\end{equation}

We denote the set of all weight functions satisfying conditions \eqref{nontriv} and \eqref{nondegen}  by $\W_{m,p}(0,\infty)$.

\begin{theorem}\cite[Theorem 1.1]{GPS}\label{assosGG} 
	Assume that $1 < m,\,p < \infty$ and $v \in \W_{m,p}(0,\infty)$. Then
	\begin{align*}
	\|g\|_{\GG(p,m,v)^{\prime}} \ap & \,\bigg( \int_0^{\infty} \bigg( \int_t^{\infty} g^{**}(s)^{p'}\,ds \bigg)^{\frac{m'}{p'}} \frac{t^{\frac{m'}{p'}} v_0(t)}{v_1(t)^{m' + 1}}\,dt \bigg)^{\frac{1}{m'}}.
	\end{align*}	
\end{theorem}

We recall the following well-known duality principle in weighted Lebesgue spaces. 
\begin{theorem}
Let $p > 1$, $f \in \M^+ (0,\infty)$	and $w \in \W (0,\infty)$. Then
$$
\bigg( \int_0^{\infty} f(t)^p w(t)\,dt \bigg)^{\frac{1}{p}} = \sup_{h \in \M^+ }\frac{\int_0^{\infty} f(t)h(t)\,dt}{\bigg(\int_0^{\infty}h(t)^{p'} w(t)^{1-p'}\,dt\bigg)^{\frac{1}{p'}}}.
$$
\end{theorem}

We will use the following statement. 
\begin{theorem}\cite[Theorem 2.1]{Sinn}\label{transfermon}
	Suppose $f,\,w \in \mp^+$. Then
	\begin{equation*}
	\sup_{g:\,\int_0^x g \le \int_0^x f} \int_0^{\infty} g(x)w(x)\,dx = \int_0^{\infty} f(x) \bigg( \sup_{t \ge x} w(t)\bigg)\,dx.
	\end{equation*}
\end{theorem}

We will apply the following "gluing" lemma.
\begin{lemma}\cite[Lemma 2.7]{gogmusunv}\label{gluing.lem.0} 
	Let $\alpha$ and $\beta$ be positive numbers. Suppose that $g,\,h \in \M^+ $ and $a \in \W (0,\infty)$ is non-decreasing. Then
	\begin{align*}
	\esup_{x \in (0,\infty)} \bigg( \int_0^{\infty} \bigg(\frac{a(x)}{a(x) + a(t)}\bigg)^{\beta} g(t)\,dt \bigg)^{\frac{1}{\beta}} \bigg(\int_0^{\infty} \bigg(\frac{a(t)}{a(x) + a(t)}\bigg)^{\alpha} h(t)\,dt\bigg)^{\frac{1}{\alpha}} & \\
	& \hspace{-7cm} \ap \esup_{x \in (0,\infty)} \bigg( \int_0^x g(t)\,dt\bigg)^{\frac{1}{\beta}} \bigg(\int_x^{\infty} h(t)\,dt\bigg)^{\frac{1}{\alpha}} + \esup_{x \in (0,\infty)} \bigg( \int_x^{\infty} a(t)^{-\beta} g(t)\,dt \bigg)^{\frac{1}{\beta}}  \bigg( \int_0^x a(t)^{\alpha}h(t)\,dt\bigg)^{\frac{1}{\alpha}}.
	\end{align*}
\end{lemma}

We recall the following "an integration by parts" formula. 
\begin{theorem}\cite[Theorem 2.1]{musbil_2}\label{thm.IBP.0} 
	Let $\alpha > 0$. Let $g$ be a non-negative function on $(0,\infty)$ such that $0 < \int_0^t g < \infty$ for all $t \in (0,\infty)$ and let $f$ be a non-negative non-increasing right-continuous function on $(0,\infty)$. Then
	\begin{align*}
	A_1 : = \int_0^{\infty} \bigg( \int_0^t g \bigg)^{\alpha} g(t) [f(t) - \lim_{t \rw +\infty} f(t)]\,dt < \infty \quad \Longleftrightarrow \quad A_2 : = \int_{(0,\infty)} \bigg( \int_0^t g \bigg)^{\alpha + 1}\,d[-f(t)] < \infty.
	\end{align*}
	Moreover, $A_1 \approx A_2$.
\end{theorem}

Investigation of weighted iterated Hardy-type inequalities started with studying of the inequality
\begin{equation}\label{mainn0}
\bigg( \int_0^{\infty} \bigg(\int_0^t \bigg( \int_s^{\infty} h(y)\,dy\bigg)^m u(s)\,ds \bigg)^{\frac{q}{m}}w(t)\,dt\bigg)^{\frac{1}{q}} \le C \bigg(\int_0^{\infty} h(t)^pv(t)\,dt\bigg)^{\frac{1}{p}}, \qquad h \in \mp^+. 
\end{equation}
Note that inequality \eqref{mainn0} have been considered in the case $m=1$ in \cite{gop2009} (see also \cite{g1}), where the result was
presented without proof, and in the case $p=1$ in \cite{gjop} and \cite{ss}, where the special
type of weight function $v$ was considered. Recall that the inequality has been completely characterized	in \cite{GMP1} and \cite{GMP2} in the case $0<m<\infty$, $0<q\leq \infty$, $1 \le p < \infty$ by using discretization and anti-discretization methods. Another approach to get the characterization of inequalities \eqref{mainn0} was presented in \cite{PS_Proc_2013}. The characterization of the inequality can be reduced to the characterization of the weighted Hardy	inequality on the cones of non-increasing functions (see, \cite{gog.mus.2017_1} and \cite{gog.mus.2017_2}). Different approach to solve iterated Hardy-type inequalities has been given in \cite{mus.2017}.

As it was mentioned in \cite{gog.mus.2017_1} the characterization of "dual" inequality
\begin{equation}\label{iterH2}
\bigg( \int_0^{\infty} \bigg(\int_t^{\infty} \bigg( \int_0^s h(y)\,dy\bigg)^m u(s)\,ds \bigg)^{\frac{q}{m}}w(t)\,dt\bigg)^{\frac{1}{q}} \le C \bigg(\int_0^{\infty} h(t)^pv(t)\,dt\bigg)^{\frac{1}{p}}, \qquad h \in \mp^+.
\end{equation}
can be easily obtained from the solutions of inequality \eqref{mainn0}, which was presented in \cite{GKPS}. 

\begin{theorem}\cite[Theorem 2.9, (a) and (c)]{GKPS}\label{gks} 
	Let $p,\,q,\,m \in (1,\infty)$ and $u,\,w,\,v$ be weights on $(0,\infty)$. Assume that the following non-degeneracy conditions are satisfied:
	\begin{itemize}
		\item 
		$u$ is strictly positive, $\int_t^{\infty} u(s)\,ds < \infty$ for all $t \in (0,\infty)$, $\int_0^{\infty} u(s)\,ds = \infty$, 
		
		\item 
		$\int_0^t w(s)\,ds < \infty$, $\int_t^{\infty} w(s) \bigg( \int_s^{\infty} u(y)\,dy\bigg)^{\frac{q}{m}}\,ds < \infty$ for all $t \in (0,\infty)$,
		
		\item 
		$\int_0^1 w(s) \bigg( \int_s^{\infty} u(y)\,dy\bigg)^{\frac{q}{m}}\,ds = \infty$, $\int_1^{\infty} w(s)\,ds = \infty$.
	\end{itemize}
	Let
	$$
	C = \sup_{h \in \mp^+ } \frac{\bigg( \int_0^{\infty} \bigg(\int_t^{\infty} \bigg( \int_0^s h(y)\,dy\bigg)^m u(s)\,ds \bigg)^{\frac{q}{m}}w(t)\,dt\bigg)^{\frac{1}{q}}}{\bigg(\int_0^{\infty} h(t)^pv(t)\,dt\bigg)^{\frac{1}{p}}}
	$$
	
	{\rm (a)} If $p \le m$ and $p \le q$, then $C \approx D_1 + D_2$, where
	$$
	D_1 = \sup_{t \in (0,\infty)} \bigg( \int_0^t w(s) \,ds \bigg)^{\frac{1}{q}}  \, \bigg( \int_t^{\infty} u(s)\,ds \bigg)^{\frac{1}{m}} \bigg(\int_0^t v(\tau)^{1-p'}\,d\tau\bigg)^{\frac{1}{p'}}
	$$
	and
	$$
	D_2 = \sup_{t \in (0,\infty)} \bigg( \int_t^{\infty} \bigg( \int_s^{\infty} u(y)\,dy \bigg)^{\frac{q}{m}} w(s) \, ds \bigg)^{\frac{1}{q}} \bigg(\int_0^t v(s)^{1-p'}\,ds\bigg)^{\frac{1}{p'}}.
	$$
	
	{\rm (b)} If $m < p$ and $p \le q$, then $C \approx D_2 + D_3$, where
	$$
	D_3 = \sup_{t \in (0,\infty)} \bigg( \int_0^t w(s) \,ds \bigg)^{\frac{1}{q}} \bigg(\int_t^{\infty} \bigg( \int_s^{\infty} u(y)\,dy \bigg)^{\frac{p}{p - m}} \bigg( \int_0^s v(\tau)^{1 - p'}\,d\tau\bigg)^{\frac{p(m-1)}{p - m}} v(s)^{1 - p'}\,ds\bigg)^{\frac{p - m}{pm}}.
	$$
\end{theorem}

Another pair of "dual" weighted iterated Hardy-type inequalities are
\begin{equation}\label{iterH1}
\bigg( \int_0^{\infty} \bigg(\int_t^{\infty} \bigg( \int_s^{\infty} h(y)\,dy\bigg)^m u(s)\,ds \bigg)^{\frac{q}{m}}w(t)\,dt\bigg)^{\frac{1}{q}} \le C \bigg(\int_0^{\infty} h(t)^pv(t)\,dt\bigg)^{\frac{1}{p}}, \qquad h \in \mp^+ 
\end{equation}
and 
\begin{equation}\label{iterH3}
\bigg( \int_0^{\infty} \bigg(\int_0^t \bigg( \int_0^s h(y)\,dy\bigg)^m u(s)\,ds \bigg)^{\frac{q}{m}}w(t)\,dt\bigg)^{\frac{1}{q}} \le C \bigg(\int_0^{\infty} h(t)^pv(t)\,dt\bigg)^{\frac{1}{p}}, \qquad h \in \mp^+.
\end{equation}
Both of them were characterized in \cite{gog.mus.2017_1} by so-called "flipped" conditions. The "classical" conditions ensuring the validity of \eqref{iterH1} was recently presented in \cite{krepick}.

\begin{theorem}\cite[Theorem 1.1, (a) and (c)]{krepick}\label{krepick} 
Let $p,\,q,\,m \in (1,\infty)$ and $u,\,w,\,v$ be weights such that the pair $(u,w)$ is admissible with respect to $(m.q)$, that is,
$$
0 < \int_0^t \bigg(\int_s^t u(y)\,dy\bigg)^{\frac{q}{m}} w(s)\,ds < \infty, \qquad t \in (0,\infty). 
$$
Let
$$
C = \sup_{h \in \mp^+ } \frac{\bigg( \int_0^{\infty} \bigg(\int_t^{\infty} \bigg( \int_s^{\infty} h(y)\,dy\bigg)^m u(s)\,ds \bigg)^{\frac{q}{m}}w(t)\,dt\bigg)^{\frac{1}{q}}}{\bigg(\int_0^{\infty} h(t)^pv(t)\,dt\bigg)^{\frac{1}{p}}}
$$

{\rm (a)} If $p \le m$ and $p \le q$, then $C \approx E_1$, where
$$
E_1 = \sup_{t \in (0,\infty)} \bigg( \int_0^t w(s) \, \bigg( \int_s^t u(y)\,dy \bigg)^{\frac{q}{m}} ds \bigg)^{\frac{1}{q}} \bigg(\int_t^{\infty} v(s)^{1-p'}\,ds\bigg)^{\frac{1}{p'}}.
$$

{\rm (b)} If $m < p$ and $p \le q$, then $C \approx E_1 + E_2$, where
$$
E_2 = \sup_{t \in (0,\infty)} \bigg( \int_0^t w(s) \,ds \bigg)^{\frac{1}{q}} \bigg(\int_t^{\infty} \bigg( \int_t^s u(y)\,dy \bigg)^{\frac{m}{p - m}} u(s) \bigg( \int_s^{\infty} v(\tau)^{1 - p'}\,d\tau\bigg)^{\frac{m(p-1)}{p - m}}\,ds\bigg)^{\frac{p - m}{pm}}.
$$
\end{theorem}

We need solutions of inequalities 
\begin{align*}
\bigg( \int_t^{\infty} \bigg( \int_t^x \bigg( \sup_{s \le \tau}u(\tau) \bigg(\int_t^{\tau} f(y)\,dy \bigg) \bigg) a(s)\,ds \bigg)^q w(x)\,dx \bigg)^{\frac{1}{q}} \le C \, \bigg( \int_t^{\infty} f(s)^pv(s)\,ds\bigg)^{\frac{1}{p}}, \qquad f \in \mp^+ (t,\infty)
\end{align*}
and
\begin{align*}
\bigg( \int_t^{\infty} \bigg( \int_t^y \bigg( \sup_{s \le \tau}u(\tau)  \bigg(\int_{\tau}^{\infty} f(z)\,dz \bigg) \bigg) a(s)\,ds \bigg)^q w(y)\,dy \bigg)^{\frac{1}{q}} \le C \, \bigg( \int_t^{\infty} f(s)^pv(s)\,ds\bigg)^{\frac{1}{p}}, \qquad f \in \mp^+ (t,\infty)
\end{align*}
where $t$ is a given pont in $[0,\infty)$ and $1 < p,\, q < \infty$ and $u \in \W (t,\infty) \cap C (t,\infty)$ and  $a, v,\,w \in \W (t,\infty)$.

The proofs of the following two statements are straightforward and capacious, but for the convenience of the reader we give the complete proof of both of them.
\begin{theorem}\label{aux.thm.2} 
	Let $1 < p , q < \infty$. Given $t \ge 0$ assume that $u \in \W (t,\infty) \cap C (t,\infty)$ and  $a, v,\,w \in \W (t,\infty)$. Moreover, assume that $0 < \int_0^x v(\tau)^{1-p'}\,d\tau< \infty,\, x > t$.
	\begin{itemize}
		\item If $p\le q$, then	
		\begin{align*}
		\sup_{f \in \mp^+ (t,\infty)} \frac{\bigg( \int_t^{\infty} \bigg( \int_t^x \bigg( \sup_{s \le \tau}u(\tau) \bigg(\int_t^{\tau} f(y)\,dy \bigg) \bigg) a(s)\,ds \bigg)^q w(x)\,dx \bigg)^{\frac{1}{q}}}{\bigg( \int_t^{\infty} f(s)^pv(s)\,ds\bigg)^{\frac{1}{p}}} & \\
		& \hspace{-8cm} \approx \, \sup_{y \in (t,\infty)} \bigg( \int_t^{y} v(x)^{1-p'} \, \bigg( \int_{x}^{y} \bigg( \sup_{s \le \tau}u(\tau)\bigg) a(s)\,ds \bigg)^{p'} \, dx \bigg)^{\frac{1}{p'}} \, \bigg( \int_{y}^{\infty} w (z) \,dz \bigg)^{\frac{1}{q}} \\
		& \hspace{-7.5cm} + \sup_{y \in (t,\infty)} \bigg( \int_t^{y} v(x)^{1-p'}  \, dx \bigg)^{\frac{1}{p'}} \, \bigg( \int_{y}^{\infty} \, \bigg( \int_{y}^{z} \bigg( \sup_{s \le \tau}u(\tau)\bigg) a(s)\,ds \bigg)^{q} w(z) \,dz \bigg)^{\frac{1}{q}} \\
		& \hspace{-7.5cm} + \sup_{x \in (t,\infty)} \bigg( \int_{[x,\infty)} \, d \, \bigg( - \sup_{y \le \tau} u(\tau)^{p'} \bigg( \int_t^{\tau} v(s)^{1-p'}\,ds \bigg) \bigg) \bigg)^{\frac{1}{p'}} \, \bigg( \int_t^{x} \bigg( \int_{t}^{z} a \bigg)^q w(z) \,dz \bigg)^{\frac{1}{q}} \\
		& \hspace{-7.5cm} + \sup_{x \in (t,\infty)} \bigg( \int_{(t,x]} \,  \bigg( \int_t^{y} a \bigg)^{p'} \, d \, \bigg( - \sup_{y \le \tau} u(\tau)^{p'} \bigg( \int_t^{\tau} v(s)^{1-p'}\,ds \bigg) \bigg) \bigg)^{\frac{1}{p'}} \, \bigg( \int_{x}^{\infty} w(z) \,dz \bigg)^{\frac{1}{q}} \\
		& \hspace{-7.5cm} + \bigg( \int_t^{\infty} \bigg( \int_t^{z} a \bigg)^q w(z) \,dz \bigg)^{\frac{1}{q}} \, \lim_{x \rightarrow \infty} \bigg(\sup_{x \le \tau} u(\tau) \bigg( \int_t^{\tau} v(s)^{1-p'}\,ds \bigg)^{\frac{1}{p'}}\bigg).
		\end{align*}
		
		\item If $q<p$, then
		\begin{align*}
		\sup_{f \in \mp^+ (t,\infty)} \frac{\bigg( \int_t^{\infty} \bigg( \int_t^x \bigg( \sup_{s \le \tau}u(\tau) \bigg(\int_t^{\tau} f(y)\,dy \bigg) \bigg) a(s)\,ds \bigg)^q w(x)\,dx \bigg)^{\frac{1}{q}}}{\bigg( \int_t^{\infty} f(s)^pv(s)\,ds\bigg)^{\frac{1}{p}}} & \\
		& \hspace{-9cm} \approx \, \bigg( \int_t^{\infty}  \bigg( \int_t^y v(x)^{1-p'} \,dx \bigg)^{\frac{p(q-1)}{p-q}} v(y)^{1-p'} \, \bigg( \int_{y}^{\infty} \bigg( \int_{y}^{z} \bigg( \sup_{s \le \tau} u(\tau)\bigg) a(s)\,ds \bigg)^{q} w(z)\, dz \bigg)^{\frac{p}{p-q}} \, dy \bigg)^{\frac{p-q}{pq}} \\
		& \hspace{-8.5cm} + \bigg( \int_t^{\infty} \bigg( \int_t^{y} v(x)^{1-p'} \bigg( \int_{x}^{y} \bigg( \sup_{s \le \tau} u(\tau) \bigg) a(s)\,ds \bigg)^{p'} \, dx \bigg)^{\frac{q(p-1)}{p-q}} \bigg( \int_{y}^{\infty} w(z) \,dz \bigg)^{\frac{q}{p-q}} w(y)\,dy \bigg)^{\frac{p-q}{pq}} \\
		& \hspace{-8.5cm} + \bigg( \int_t^{\infty} \bigg( \int_{[x,\infty)} \, d \, \bigg( - \bigg(\sup_{y \le \tau} u(\tau)^{p'} \, \bigg( \int_t^{\tau} v(s)^{1-p'}\,ds \bigg) \bigg) \bigg) \bigg)^{\frac{q(p-1)}{p-q}} 
		\, \bigg( \int_t^{x} \bigg( \int_t^{z} a \bigg)^q w(z) \,dz \bigg)^{\frac{q}{p-q}} \bigg( \int_t^{x} a\bigg)^q w(x) \,dx\bigg)^{\frac{p-q}{pq}} \\
		& \hspace{-8.5cm} + \bigg( \int_t^{\infty} \bigg( \int_{(t,x]} \,  \bigg( \int_t^{y} a \bigg)^{p'} \, d \, \bigg( - \bigg(\sup_{y \le \tau} u(\tau)^{p'} \, \bigg( \int_t^{\tau} v(s)^{1-p'}\,ds \bigg) \bigg) \bigg)  \bigg)^{\frac{q(p-1)}{p-q}} \, \bigg( \int_{x}^{\infty} w(z) \,dz \bigg)^{\frac{q}{p-q}} w(x)\,dx \bigg)^{\frac{p-q}{pq}} \\
		& \hspace{-8.5cm} + \bigg( \int_t^{\infty} \bigg( \int_t^{z} a \bigg)^q w(z) \,dz \bigg)^{\frac{1}{q}} \, \lim_{x \rightarrow \infty} \bigg(\sup_{x \le \tau} u(\tau) \bigg( \int_t^{\tau} v(s)^{1-p'}\,ds \bigg)^{\frac{1}{p'}}\bigg).
		\end{align*}
	\end{itemize}
\end{theorem}

\begin{proof}
	The statement was formulated in \cite[Theorem 3.1]{musbil_2} for $t = 0$. Assume that $t > 0$. The proof uses multiple changes of variables of the type $x + t = y$ several times, which we will not specify at any step exactly.
	
	It is easy to see that, given a point $t \in (0,\infty)$, the inequality
	\begin{align}
	\bigg( \int_t^{\infty} \bigg( \int_t^x \bigg( \sup_{s \le \tau}u(\tau) \bigg(\int_t^{\tau} f(y)\,dy \bigg) \bigg) a(s)\,ds \bigg)^q w(x)\,dx \bigg)^{\frac{1}{q}} \le C \bigg( \int_t^{\infty} f(s)^pv(s)\,ds\bigg)^{\frac{1}{p}}, \qquad f \in \mp^+ (t,\infty) \label{ineq2} 
	\end{align}
	is equivalent to the inequality
	\begin{equation}\label{eq.5}
	\bigg(\int_{0}^{\infty}\bigg(\int_0^x \bigg( \sup_{s \le \tau}u_t(\tau) \bigg(\int_0^{\tau} f(y)\,dy\bigg) \bigg) a_t(s) \,ds \bigg)^q\,w_t(x) \, dx \bigg)^\frac{1}{q} \le C\, \bigg(\int_0^{\infty} f(s)^p v_t(s)\,ds\bigg)^\frac{1}{p}, \qquad f \in \mp^+ (0,\infty),
	\end{equation}
	where the weight functions $u_t$, $w_t$, $v_t$ and $a_t$ are defined for $x > 0$ as follows:
	\begin{align*}
	u_t(x):= u(x+t),\qquad w_t(x):= \,w(x+t),\qquad v_t(x):= v(x+t), \qquad a_t(x):=a(x+t).
	\end{align*}
	Indeed: Applying changes of variables, we get that 
	\begin{alignat*}{2}
	\int_{t}^{\infty}\bigg(\int_t^x \bigg( \sup_{s \le \tau}u(\tau) \bigg(\int_t^{\tau} f(z)\,dz \bigg) \bigg) a(s) \,ds\bigg)^q w(x) \, dx & \\
	& \hspace{-6cm} = \int_{t}^{\infty}\bigg(\int_t^x \bigg( \sup_{s \le \tau}u(\tau) \bigg(\int_0^{\tau-t} f(z+t)\,dz \bigg) \bigg)a(s) \,ds \bigg)^q w(x) \, dx \\
	& \hspace{-6cm} = \int_{t}^{\infty}\bigg(\int_t^x \bigg( \sup_{s-t \le \tau-t}u(\tau) \bigg( \int_0^{\tau-t} f(z+t)\,dz\bigg)\bigg) a(s) \,ds \bigg)^q w(x) \, dx \\
	& \hspace{-6cm} = \int_{t}^{\infty}\bigg(\int_t^x \bigg( \sup_{s-t \le \tau}u(\tau+t) \bigg(\int_0^{\tau} f(z+t)\,dz\bigg)\bigg) a(s) \,ds \bigg)^q w(x) \, dx \\
	& \hspace{-6cm} = \int_{t}^{\infty}\bigg(\int_0^{x-t} \bigg( \sup_{s \le \tau}u(\tau+t) \bigg( \int_0^{\tau} f(z+t)\,dz\bigg)\bigg) a(s+t) \,ds \bigg)^q w(x) \, dx \\
	& \hspace{-6cm} = \int_{0}^{\infty}\bigg(\int_0^{x} \bigg( \sup_{s \le \tau}u(\tau+t) \bigg(\int_0^{\tau} f(z+t)\,dz\bigg) \bigg) a(s+t)\,ds \bigg)^q\,w(x + t) \, dx \\
	& \hspace{-6cm} = \int_{0}^{\infty}\bigg(\int_0^{x} \bigg( \sup_{s \le \tau}u_t(\tau) \bigg(\int_0^{\tau} f(z+t)\,dz\bigg) \bigg) a_t(s)\,ds \bigg)^q\,w_t(x) \, dx
	\end{alignat*}
	and
	\begin{alignat*}{2}
	\int_t^{\infty} f(s)^p v(s)\,ds = \int_0^{\infty} f(s+t)^p v(s+t)\,ds = \int_0^{\infty} f(s+t)^p v_t(s)\,ds.
	\end{alignat*}
	Hence inequality \eqref{ineq2} can be rewritten as follows:
	\begin{equation*}
	\bigg(\int_{0}^{\infty}\bigg(\int_0^{y} \bigg( \sup_{s \le \tau}u_t(\tau) \bigg(\int_{\tau}^{\infty} f(z+t)\,dz\bigg) \bigg) a_t(s) \,ds \bigg)^q\,w_t(y) \, dy \bigg)^\frac{1}{q} \le C\, \bigg(\int_0^{\infty} f(s+t)^p v_t(s)\,ds\bigg)^\frac{1}{p}, \qquad f \in \mp^+ (t,\infty).
	\end{equation*}
	It remains note that the latter is equivalent to inequality \eqref{eq.5}.

	Let $p\le q$. By \cite[Theorem 3.1]{musbil_2}, we have for any $t \in (0,\infty)$ that
	\begin{align*}
	\sup_{f \in \mp^+ (0,\infty)} \frac{\bigg(\int_{0}^{\infty}\bigg(\int_0^x \bigg( \sup_{s \le \tau}u_t(\tau) \bigg(\int_0^{\tau} f(y)\,dy\bigg) \bigg) a_t(s) \,ds \bigg)^q\,w_t(x) \, dx \bigg)^\frac{1}{q}}{\bigg(\int_0^{\infty} f(s)^p v_t(s)\,ds\bigg)^\frac{1}{p}} & \\
	& \hspace{-9cm} \approx \, \sup_{y \in (0,\infty)} \bigg( \int_0^y v_t(x)^{1-p'} \, \bigg( \int_x^y \bigg( \sup_{s \le \tau}u_t(\tau)\bigg) a_t(s)\,ds \bigg)^{p'} \, dx \bigg)^{\frac{1}{p'}} \, \bigg( \int_y^{\infty} w_t(z) \,dz \bigg)^{\frac{1}{q}}  \\
	& \hspace{-8.5cm} + \sup_{y \in (0,\infty)} \bigg( \int_0^y v_t(x)^{1-p'}  \, dx \bigg)^{\frac{1}{p'}} \, \bigg( \int_y^{\infty} \, \bigg( \int_y^z \bigg( \sup_{s \le \tau}u_t(\tau)\bigg) a_t(s)\,ds \bigg)^{q} w_t(z) \,dz \bigg)^{\frac{1}{q}} \\
	& \hspace{-8.5cm} + \, \sup_{x \in (0,\infty)} \bigg( \int_{[x,\infty)} \, d \, \bigg( - \sup_{y \le \tau} u_t(\tau)^{p'} \bigg( \int_0^{\tau} v_t(s)^{1-p'}\,ds \bigg) \bigg) \bigg)^{\frac{1}{p'}} \, \bigg( \int_0^x A_t(z)^q w_t(z) \,dz \bigg)^{\frac{1}{q}} \\
	& \hspace{-8.5cm} + \sup_{x \in (0,\infty)} \bigg( \int_{(0,x]} \,  A_t(y)^{p'} \, d \, \bigg( - \sup_{y \le \tau} u_t(\tau)^{p'} \bigg( \int_0^{\tau} v_t(s)^{1-p'}\,ds \bigg) \bigg) \bigg)^{\frac{1}{p'}} \, \bigg( \int_x^{\infty} w_t(z) \,dz \bigg)^{\frac{1}{q}} \\
	& \hspace{-8.5cm} + \bigg( \int_0^{\infty} A_t(z)^q w_t(z) \,dz \bigg)^{\frac{1}{q}} \, \lim_{x \rightarrow \infty} \bigg(\sup_{x \le \tau} u_t(\tau) \bigg( \int_0^{\tau} v_t(s)^{1-p'}\,ds \bigg)^{\frac{1}{p'}}\bigg).
	\end{align*}
	
	Since
	\begin{align*}
	\sup_{y \in (0,\infty)} \bigg( \int_0^y v_t(x)^{1-p'} \, \bigg( \int_x^y \bigg( \sup_{s \le \tau}u_t(\tau)\bigg) a_t(s)\,ds \bigg)^{p'} \, dx \bigg)^{\frac{1}{p'}} \, \bigg( \int_y^{\infty} w_t(z) \,dz \bigg)^{\frac{1}{q}} & \\
	& \hspace{-9cm} = \sup_{y \in (0,\infty)} \bigg( \int_0^y v(x + t)^{1-p'} \, \bigg( \int_x^y \bigg( \sup_{s \le \tau}u(\tau + t)\bigg) a(s + t)\,ds \bigg)^{p'} \, dx \bigg)^{\frac{1}{p'}} \, \bigg( \int_y^{\infty} w (z + t) \,dz \bigg)^{\frac{1}{q}} \\
	& \hspace{-9cm} = \sup_{y \in (0,\infty)} \bigg( \int_0^y v(x + t)^{1-p'} \, \bigg( \int_x^y \bigg( \sup_{s + t \le \tau}u(\tau)\bigg) a(s + t)\,ds \bigg)^{p'} \, dx \bigg)^{\frac{1}{p'}} \, \bigg( \int_y^{\infty} w (z + t) \,dz \bigg)^{\frac{1}{q}} \\
	& \hspace{-9cm} = \sup_{y \in (0,\infty)} \bigg( \int_0^y v(x + t)^{1-p'} \, \bigg( \int_{x + t}^{y + t} \bigg( \sup_{s \le \tau}u(\tau)\bigg) a(s)\,ds \bigg)^{p'} \, dx \bigg)^{\frac{1}{p'}} \, \bigg( \int_y^{\infty} w (z + t) \,dz \bigg)^{\frac{1}{q}} \\
	& \hspace{-9cm} = \sup_{y \in (0,\infty)} \bigg( \int_t^{y + t} v(x)^{1-p'} \, \bigg( \int_{x}^{y + t} \bigg( \sup_{s \le \tau}u(\tau)\bigg) a(s)\,ds \bigg)^{p'} \, dx \bigg)^{\frac{1}{p'}} \, \bigg( \int_{y + t}^{\infty} w (z) \,dz \bigg)^{\frac{1}{q}} \\
	& \hspace{-9cm} = \sup_{y \in (t,\infty)} \bigg( \int_t^{y} v(x)^{1-p'} \, \bigg( \int_{x}^{y} \bigg( \sup_{s \le \tau}u(\tau)\bigg) a(s)\,ds \bigg)^{p'} \, dx \bigg)^{\frac{1}{p'}} \, \bigg( \int_{y}^{\infty} w (z) \,dz \bigg)^{\frac{1}{q}},
	\end{align*}
	
	\begin{align*}
	\sup_{y \in (0,\infty)} \bigg( \int_0^y v_t(x)^{1-p'}  \, dx \bigg)^{\frac{1}{p'}} \, \bigg( \int_y^{\infty} \, \bigg( \int_y^z \bigg( \sup_{s \le \tau}u_t(\tau)\bigg) a_t(s)\,ds \bigg)^{q} w_t(z) \,dz \bigg)^{\frac{1}{q}} & \\
	& \hspace{-10cm} = \sup_{y \in (0,\infty)} \bigg( \int_0^y v(x + t)^{1-p'}  \, dx \bigg)^{\frac{1}{p'}} \, \bigg( \int_y^{\infty} \, \bigg( \int_y^z \bigg( \sup_{s \le \tau}u(\tau + t)\bigg) a(s + t)\,ds \bigg)^{q} w(z + t) \,dz \bigg)^{\frac{1}{q}} \\
	& \hspace{-10cm} = \sup_{y \in (0,\infty)} \bigg( \int_t^{y + t} v(x)^{1-p'}  \, dx \bigg)^{\frac{1}{p'}} \, \bigg( \int_y^{\infty} \, \bigg( \int_{y + t}^{z + t} \bigg( \sup_{s \le \tau}u(\tau)\bigg) a(s)\,ds \bigg)^{q} w(z + t) \,dz \bigg)^{\frac{1}{q}} \\
	& \hspace{-10cm} = \sup_{y \in (0,\infty)} \bigg( \int_t^{y + t} v(x)^{1-p'}  \, dx \bigg)^{\frac{1}{p'}} \, \bigg( \int_{y + t}^{\infty} \, \bigg( \int_{y + t}^{z} \bigg( \sup_{s \le \tau}u(\tau)\bigg) a(s)\,ds \bigg)^{q} w(z) \,dz \bigg)^{\frac{1}{q}} \\
	& \hspace{-10cm} = \sup_{y \in (t,\infty)} \bigg( \int_t^{y} v(x)^{1-p'}  \, dx \bigg)^{\frac{1}{p'}} \, \bigg( \int_{y}^{\infty} \, \bigg( \int_{y}^{z} \bigg( \sup_{s \le \tau}u(\tau)\bigg) a(s)\,ds \bigg)^{q} w(z) \,dz \bigg)^{\frac{1}{q}},
	\end{align*}
	
	\begin{align*}
	\sup_{x \in (0,\infty)} \bigg( \int_{[x,\infty)} \, d \, \bigg( - \sup_{y \le \tau} u_t(\tau)^{p'} \bigg( \int_0^{\tau} v_t(s)^{1-p'}\,ds \bigg) \bigg) \bigg)^{\frac{1}{p'}} \, \bigg( \int_0^x A_t(z)^q w_t(z) \,dz \bigg)^{\frac{1}{q}} & \\
	& \hspace{-10.5cm} = \sup_{x \in (0,\infty)} \bigg( \int_{[x,\infty)} \, d \, \bigg( - \sup_{y \le \tau} u(\tau + t)^{p'} \bigg( \int_0^{\tau} v(s + t)^{1-p'}\,ds \bigg) \bigg) \bigg)^{\frac{1}{p'}} \, \bigg( \int_0^x \bigg( \int_{t}^{z+t} a \bigg)^q w(z + t) \,dz \bigg)^{\frac{1}{q}} \\
	& \hspace{-10.5cm} = \sup_{x \in (0,\infty)} \bigg( \int_{[x,\infty)} \, d \, \bigg( - \sup_{y \le \tau} u(\tau + t)^{p'} \bigg( \int_t^{\tau + t} v(s)^{1-p'}\,ds \bigg) \bigg) \bigg)^{\frac{1}{p'}} \, \bigg( \int_0^x \bigg( \int_{t}^{z+t} a \bigg)^q w(z + t) \,dz \bigg)^{\frac{1}{q}} \\
	& \hspace{-10.5cm} = \sup_{x \in (0,\infty)} \bigg( \int_{[x,\infty)} \, d \, \bigg( - \sup_{y + t \le \tau} u(\tau)^{p'} \bigg( \int_t^{\tau} v(s)^{1-p'}\,ds \bigg) \bigg) \bigg)^{\frac{1}{p'}} \, \bigg( \int_0^x \bigg( \int_{t}^{z+t} a \bigg)^q w(z + t) \,dz \bigg)^{\frac{1}{q}} \\
	& \hspace{-10.5cm} = \sup_{x \in (0,\infty)} \bigg( \int_{[x + t,\infty)} \, d \, \bigg( - \sup_{y \le \tau} u(\tau)^{p'} \bigg( \int_t^{\tau} v(s)^{1-p'}\,ds \bigg) \bigg) \bigg)^{\frac{1}{p'}} \, \bigg( \int_t^{x + t} \bigg( \int_{t}^{z} a \bigg)^q w(z) \,dz \bigg)^{\frac{1}{q}} \\
	& \hspace{-10.5cm} = \sup_{x \in (t,\infty)} \bigg( \int_{[x,\infty)} \, d \, \bigg( - \sup_{y \le \tau} u(\tau)^{p'} \bigg( \int_t^{\tau} v(s)^{1-p'}\,ds \bigg) \bigg) \bigg)^{\frac{1}{p'}} \, \bigg( \int_t^{x} \bigg( \int_{t}^{z} a \bigg)^q w(z) \,dz \bigg)^{\frac{1}{q}},
	\end{align*}
	
	\begin{align*}
	\sup_{x \in (0,\infty)} \bigg( \int_{(0,x]} \,  A_t(y)^{p'} \, d \, \bigg( - \sup_{y \le \tau} u_t(\tau)^{p'} \bigg( \int_0^{\tau} v_t(s)^{1-p'}\,ds \bigg) \bigg) \bigg)^{\frac{1}{p'}} \, \bigg( \int_x^{\infty} w_t(z) \,dz \bigg)^{\frac{1}{q}} & \\
	& \hspace{-11cm} = \sup_{x \in (0,\infty)} \bigg( \int_{(0,x]} \,  \bigg( \int_t^{y + t} a \bigg)^{p'} \, d \, \bigg( - \sup_{y \le \tau} u(\tau + t)^{p'} \bigg( \int_0^{\tau} v(s + t)^{1-p'}\,ds \bigg) \bigg) \bigg)^{\frac{1}{p'}} \, \bigg( \int_x^{\infty} w(z + t) \,dz \bigg)^{\frac{1}{q}} \\
	& \hspace{-11cm} = \sup_{x \in (0,\infty)} \bigg( \int_{(0,x]} \,  \bigg( \int_t^{y + t} a \bigg)^{p'} \, d \, \bigg( - \sup_{y \le \tau} u(\tau + t)^{p'} \bigg( \int_t^{\tau + t} v(s)^{1-p'}\,ds \bigg) \bigg) \bigg)^{\frac{1}{p'}} \, \bigg( \int_x^{\infty} w(z + t) \,dz \bigg)^{\frac{1}{q}} \\
	& \hspace{-11cm} = \sup_{x \in (0,\infty)} \bigg( \int_{(0,x]} \,  \bigg( \int_t^{y + t} a \bigg)^{p'} \, d \, \bigg( - \sup_{y + t \le \tau} u(\tau)^{p'} \bigg( \int_t^{\tau} v(s)^{1-p'}\,ds \bigg) \bigg) \bigg)^{\frac{1}{p'}} \, \bigg( \int_x^{\infty} w(z + t) \,dz \bigg)^{\frac{1}{q}} \\
	& \hspace{-11cm} = \sup_{x \in (0,\infty)} \bigg( \int_{(t,x + t]} \,  \bigg( \int_t^{y} a \bigg)^{p'} \, d \, \bigg( - \sup_{y \le \tau} u(\tau)^{p'} \bigg( \int_t^{\tau} v(s)^{1-p'}\,ds \bigg) \bigg) \bigg)^{\frac{1}{p'}} \, \bigg( \int_{x + t}^{\infty} w(z) \,dz \bigg)^{\frac{1}{q}} \\
	& \hspace{-11cm} = \sup_{x \in (t,\infty)} \bigg( \int_{(t,x]} \,  \bigg( \int_t^{y} a \bigg)^{p'} \, d \, \bigg( - \sup_{y \le \tau} u(\tau)^{p'} \bigg( \int_t^{\tau} v(s)^{1-p'}\,ds \bigg) \bigg) \bigg)^{\frac{1}{p'}} \, \bigg( \int_{x}^{\infty} w(z) \,dz \bigg)^{\frac{1}{q}},
	\end{align*}
	
	\begin{align*}
	\bigg( \int_0^{\infty} A_t(z)^q w_t(z) \,dz \bigg)^{\frac{1}{q}} \, \lim_{x \rightarrow \infty} \bigg(\sup_{x \le \tau} u_t(\tau) \bigg( \int_0^{\tau} v_t(s)^{1-p'}\,ds \bigg)^{\frac{1}{p'}}\bigg) & \\
	& \hspace{-8cm} = \bigg( \int_0^{\infty} \bigg( \int_t^{z + t} a \bigg)^q w(z + t) \,dz \bigg)^{\frac{1}{q}} \, \lim_{x \rightarrow \infty} \bigg(\sup_{x \le \tau} u(\tau + t) \bigg( \int_0^{\tau} v(s + t)^{1-p'}\,ds \bigg)^{\frac{1}{p'}}\bigg) \\
	& \hspace{-8cm} = \bigg( \int_t^{\infty} \bigg( \int_t^{z} a \bigg)^q w(z) \,dz \bigg)^{\frac{1}{q}} \, \lim_{x \rightarrow \infty} \bigg(\sup_{x \le \tau} u(\tau + t) \bigg( \int_t^{\tau + t} v(s)^{1-p'}\,ds \bigg)^{\frac{1}{p'}}\bigg) \\
	& \hspace{-8cm} = \bigg( \int_t^{\infty} \bigg( \int_t^{z} a \bigg)^q w(z) \,dz \bigg)^{\frac{1}{q}} \, \lim_{x \rightarrow \infty} \bigg(\sup_{x + t \le \tau} u(\tau) \bigg( \int_t^{\tau} v(s)^{1-p'}\,ds \bigg)^{\frac{1}{p'}}\bigg) \\
	& \hspace{-8cm} = \bigg( \int_t^{\infty} \bigg( \int_t^{z} a \bigg)^q w(z) \,dz \bigg)^{\frac{1}{q}} \, \lim_{x \rightarrow \infty} \bigg(\sup_{x \le \tau} u(\tau) \bigg( \int_t^{\tau} v(s)^{1-p'}\,ds \bigg)^{\frac{1}{p'}}\bigg),
	\end{align*}
	we arrive at
	\begin{align*}
	\sup_{f \in \mp^+ (t,\infty)} \frac{\bigg( \int_t^{\infty} \bigg( \int_t^x \bigg( \sup_{s \le \tau}u(\tau) \bigg(\int_t^{\tau} f(y)\,dy \bigg) \bigg) a(s)\,ds \bigg)^q w(x)\,dx \bigg)^{\frac{1}{q}}}{\bigg( \int_t^{\infty} f(s)^pv(s)\,ds\bigg)^{\frac{1}{p}}} & \\
	& \hspace{-9cm} \approx \, \sup_{y \in (t,\infty)} \bigg( \int_t^{y} v(x)^{1-p'} \, \bigg( \int_{x}^{y} \bigg( \sup_{s \le \tau}u(\tau)\bigg) a(s)\,ds \bigg)^{p'} \, dx \bigg)^{\frac{1}{p'}} \, \bigg( \int_{y}^{\infty} w (z) \,dz \bigg)^{\frac{1}{q}} \\
	& \hspace{-8.5cm} + \sup_{y \in (t,\infty)} \bigg( \int_t^{y} v(x)^{1-p'}  \, dx \bigg)^{\frac{1}{p'}} \, \bigg( \int_{y}^{\infty} \, \bigg( \int_{y}^{z} \bigg( \sup_{s \le \tau}u(\tau)\bigg) a(s)\,ds \bigg)^{q} w(z) \,dz \bigg)^{\frac{1}{q}} \\
	& \hspace{-8.5cm} + \sup_{x \in (t,\infty)} \bigg( \int_{[x,\infty)} \, d \, \bigg( - \sup_{y \le \tau} u(\tau)^{p'} \bigg( \int_t^{\tau} v(s)^{1-p'}\,ds \bigg) \bigg) \bigg)^{\frac{1}{p'}} \, \bigg( \int_t^{x} \bigg( \int_{t}^{z} a \bigg)^q w(z) \,dz \bigg)^{\frac{1}{q}} \\
	& \hspace{-8.5cm} + \sup_{x \in (t,\infty)} \bigg( \int_{(t,x]} \,  \bigg( \int_t^{y} a \bigg)^{p'} \, d \, \bigg( - \sup_{y \le \tau} u(\tau)^{p'} \bigg( \int_t^{\tau} v(s)^{1-p'}\,ds \bigg) \bigg) \bigg)^{\frac{1}{p'}} \, \bigg( \int_{x}^{\infty} w(z) \,dz \bigg)^{\frac{1}{q}} \\
	& \hspace{-8.5cm} + \bigg( \int_t^{\infty} \bigg( \int_t^{z} a \bigg)^q w(z) \,dz \bigg)^{\frac{1}{q}} \, \lim_{x \rightarrow \infty} \bigg(\sup_{x \le \tau} u(\tau) \bigg( \int_t^{\tau} v(s)^{1-p'}\,ds \bigg)^{\frac{1}{p'}}\bigg).
	\end{align*}
	
	Let $q < p$. By \cite[Theorem 3.1]{musbil_2}, we have for any $t \in (0,\infty)$ that
	\begin{align*}
	\sup_{f \in \mp^+ (0,\infty)} \frac{\bigg(\int_{0}^{\infty}\bigg(\int_0^x \bigg( \sup_{s \le \tau}u_t(\tau) \bigg(\int_0^{\tau} f(y)\,dy\bigg) \bigg) a_t(s) \,ds \bigg)^q\,w_t(x) \, dx \bigg)^\frac{1}{q}}{\bigg(\int_0^{\infty} f(s)^p v_t(s)\,ds\bigg)^\frac{1}{p}} & \\
	& \hspace{-9.5cm} \approx \, \bigg( \int_0^{\infty}  \bigg( \int_0^y v_t(x)^{1-p'} \,dx \bigg)^{\frac{p(q-1)}{p-q}} v_t(y)^{1-p'} \, \bigg( \int_y^{\infty} \bigg( \int_y^z \bigg( \sup_{s \le \tau} u_t(\tau)\bigg) a_t(s)\,ds \bigg)^{q} w_t(z)\, dz \bigg)^{\frac{p}{p-q}} \, dy \bigg)^{\frac{p-q}{pq}}  \\
	& \hspace{-9cm} + \bigg( \int_0^{\infty} \bigg( \int_0^y v_t(x)^{1-p'} \bigg( \int_x^y \bigg( \sup_{s \le \tau} u_t(\tau) \bigg) a_t(s)\,ds \bigg)^{p'} \, dx \bigg)^{\frac{q(p-1)}{p-q}} \bigg( \int_y^{\infty} w_t(z) \,dz \bigg)^{\frac{q}{p-q}} w_t(y)\,dy \bigg)^{\frac{p-q}{pq}} \\
	& \hspace{-9cm} + \, \bigg( \int_0^{\infty} \bigg( \int_{[x,\infty)} \, d \, \bigg( - \bigg(\sup_{y \le \tau} u_t(\tau)^{p'} \, \bigg( \int_0^{\tau} v_t(s)^{1-p'}\,ds \bigg) \bigg) \bigg) \bigg)^{\frac{q(p-1)}{p-q}} 
	\, \bigg( \int_0^x A_t(z)^q w_t(z) \,dz \bigg)^{\frac{q}{p-q}} A_t(x)^q w_t(x) \,dx\bigg)^{\frac{p-q}{pq}}  \\
	& \hspace{-9cm} + \bigg( \int_0^{\infty} \bigg( \int_{(0,x]} \,  A_t(y)^{p'} \, d \, \bigg( - \bigg(\sup_{y \le \tau} u_t(\tau)^{p'} \, \bigg( \int_0^{\tau} v_t(s)^{1-p'}\,ds \bigg) \bigg) \bigg)  \bigg)^{\frac{q(p-1)}{p-q}} \, \bigg( \int_x^{\infty} w_t(z) \,dz \bigg)^{\frac{q}{p-q}} w_t(x)\,dx \bigg)^{\frac{p-q}{pq}}  \\
	& \hspace{-9cm} + \bigg( \int_0^{\infty} A_t(z)^q w_t(z) \,dz \bigg)^{\frac{1}{q}} \, \lim_{x \rightarrow \infty} \bigg(\sup_{x \le \tau} u_t(\tau) \bigg( \int_0^{\tau} v_t(s)^{1-p'}\,ds \bigg)^{\frac{1}{p'}}\bigg).
	\end{align*}
	
	Since
	\begin{align*}
	\bigg( \int_0^{\infty}  \bigg( \int_0^y v_t(x)^{1-p'} \,dx \bigg)^{\frac{p(q-1)}{p-q}} v_t(y)^{1-p'} \, \bigg( \int_y^{\infty} \bigg( \int_y^z \bigg( \sup_{s \le \tau} u_t(\tau)\bigg) a_t(s)\,ds \bigg)^{q} w_t(z)\, dz \bigg)^{\frac{p}{p-q}} \, dy \bigg)^{\frac{p-q}{pq}} & \\
	& \hspace{-13cm} = \bigg( \int_0^{\infty}  \bigg( \int_0^{y} v(x + t)^{1-p'} \,dx \bigg)^{\frac{p(q-1)}{p-q}} v(y + t)^{1-p'} \, \bigg( \int_y^{\infty} \bigg( \int_y^z \bigg( \sup_{s \le \tau} u(\tau + t)\bigg) a(s + t)\,ds \bigg)^{q} w(z + t)\, dz \bigg)^{\frac{p}{p-q}} \, dy \bigg)^{\frac{p-q}{pq}} \\
	& \hspace{-13cm} = \bigg( \int_0^{\infty}  \bigg( \int_t^{y + t} v(x)^{1-p'} \,dx \bigg)^{\frac{p(q-1)}{p-q}} v(y + t)^{1-p'} \, \bigg( \int_y^{\infty} \bigg( \int_y^z \bigg( \sup_{s + t \le \tau} u(\tau)\bigg) a(s + t)\,ds \bigg)^{q} w(z + t)\, dz \bigg)^{\frac{p}{p-q}} \, dy \bigg)^{\frac{p-q}{pq}} \\
	& \hspace{-13cm} = \bigg( \int_0^{\infty}  \bigg( \int_t^{y + t} v(x)^{1-p'} \,dx \bigg)^{\frac{p(q-1)}{p-q}} v(y + t)^{1-p'} \, \bigg( \int_y^{\infty} \bigg( \int_{y + t}^{z + t} \bigg( \sup_{s \le \tau} u(\tau)\bigg) a(s)\,ds \bigg)^{q} w(z + t)\, dz \bigg)^{\frac{p}{p-q}} \, dy \bigg)^{\frac{p-q}{pq}} \\
	& \hspace{-13cm} = \bigg( \int_0^{\infty}  \bigg( \int_t^{y + t} v(x)^{1-p'} \,dx \bigg)^{\frac{p(q-1)}{p-q}} v(y + t)^{1-p'} \, \bigg( \int_{y + t}^{\infty} \bigg( \int_{y + t}^{z} \bigg( \sup_{s \le \tau} u(\tau)\bigg) a(s)\,ds \bigg)^{q} w(z)\, dz \bigg)^{\frac{p}{p-q}} \, dy \bigg)^{\frac{p-q}{pq}} \\
	& \hspace{-13cm} = \bigg( \int_t^{\infty}  \bigg( \int_t^y v(x)^{1-p'} \,dx \bigg)^{\frac{p(q-1)}{p-q}} v(y)^{1-p'} \, \bigg( \int_{y}^{\infty} \bigg( \int_{y}^{z} \bigg( \sup_{s \le \tau} u(\tau)\bigg) a(s)\,ds \bigg)^{q} w(z)\, dz \bigg)^{\frac{p}{p-q}} \, dy \bigg)^{\frac{p-q}{pq}},
	\end{align*}
	
	\begin{align*}
	\bigg( \int_0^{\infty} \bigg( \int_0^y v_t(x)^{1-p'} \bigg( \int_x^y \bigg( \sup_{s \le \tau} u_t(\tau) \bigg) a_t(s)\,ds \bigg)^{p'} \, dx \bigg)^{\frac{q(p-1)}{p-q}} \bigg( \int_y^{\infty} w_t(z) \,dz \bigg)^{\frac{q}{p-q}} w_t(y)\,dy \bigg)^{\frac{p-q}{pq}} & \\
	& \hspace{-12cm} = \bigg( \int_0^{\infty} \bigg( \int_0^y v(x + t)^{1-p'} \bigg( \int_x^y \bigg( \sup_{s \le \tau} u(\tau + t) \bigg) a(s + t)\,ds \bigg)^{p'} \, dx \bigg)^{\frac{q(p-1)}{p-q}} \bigg( \int_y^{\infty} w(z + t) \,dz \bigg)^{\frac{q}{p-q}} w(y + t)\,dy \bigg)^{\frac{p-q}{pq}}  \\
	& \hspace{-12cm} = \bigg( \int_0^{\infty} \bigg( \int_0^y v(x + t)^{1-p'} \bigg( \int_x^y \bigg( \sup_{s + t \le \tau} u(\tau) \bigg) a(s + t)\,ds \bigg)^{p'} \, dx \bigg)^{\frac{q(p-1)}{p-q}} \bigg( \int_{y + t}^{\infty} w(z) \,dz \bigg)^{\frac{q}{p-q}} w(y + t)\,dy \bigg)^{\frac{p-q}{pq}}  \\
	& \hspace{-12cm} = \bigg( \int_0^{\infty} \bigg( \int_0^y v(x + t)^{1-p'} \bigg( \int_{x + t}^{y + t} \bigg( \sup_{s \le \tau} u(\tau) \bigg) a(s)\,ds \bigg)^{p'} \, dx \bigg)^{\frac{q(p-1)}{p-q}} \bigg( \int_{y + t}^{\infty} w(z) \,dz \bigg)^{\frac{q}{p-q}} w(y + t)\,dy \bigg)^{\frac{p-q}{pq}}  \\
	& \hspace{-12cm} = \bigg( \int_0^{\infty} \bigg( \int_t^{y + t} v(x)^{1-p'} \bigg( \int_{x}^{y + t} \bigg( \sup_{s \le \tau} u(\tau) \bigg) a(s)\,ds \bigg)^{p'} \, dx \bigg)^{\frac{q(p-1)}{p-q}} \bigg( \int_{y + t}^{\infty} w(z) \,dz \bigg)^{\frac{q}{p-q}} w(y + t)\,dy \bigg)^{\frac{p-q}{pq}}  \\
	& \hspace{-12cm} = \bigg( \int_t^{\infty} \bigg( \int_t^{y} v(x)^{1-p'} \bigg( \int_{x}^{y} \bigg( \sup_{s \le \tau} u(\tau) \bigg) a(s)\,ds \bigg)^{p'} \, dx \bigg)^{\frac{q(p-1)}{p-q}} \bigg( \int_{y}^{\infty} w(z) \,dz \bigg)^{\frac{q}{p-q}} w(y)\,dy \bigg)^{\frac{p-q}{pq}},
	\end{align*}
	
	\begin{align*}
	\bigg( \int_0^{\infty} \bigg( \int_{[x,\infty)} \, d \, \bigg( - \bigg(\sup_{y \le \tau} u_t(\tau)^{p'} \, \bigg( \int_0^{\tau} v_t(s)^{1-p'}\,ds \bigg) \bigg) \bigg) \bigg)^{\frac{q(p-1)}{p-q}} 
	\, \bigg( \int_0^x A_t(z)^q w_t(z) \,dz \bigg)^{\frac{q}{p-q}} A_t(x)^q w_t(x) \,dx\bigg)^{\frac{p-q}{pq}}  & \\
	& \hspace{-14.8cm}  = \bigg( \int_0^{\infty} \bigg( \int_{[x,\infty)} \, d \, \bigg( - \bigg(\sup_{y \le \tau} u(\tau + t)^{p'} \, \bigg( \int_0^{\tau} v(s + t)^{1-p'}\,ds \bigg) \bigg) \bigg) \bigg)^{\frac{q(p-1)}{p-q}} 
	\, \bigg( \int_0^x \bigg( \int_t^{z + t} a \bigg)^q w(z + t) \,dz \bigg)^{\frac{q}{p-q}} \bigg( \int_t^{x + t} a\bigg)^q w(x + t) \,dx\bigg)^{\frac{p-q}{pq}}  \\
	& \hspace{-14.8cm}  = \bigg( \int_0^{\infty} \bigg( \int_{[x,\infty)} \, d \, \bigg( - \bigg(\sup_{y \le \tau} u(\tau + t)^{p'} \, \bigg( \int_t^{\tau + t} v(s)^{1-p'}\,ds \bigg) \bigg) \bigg) \bigg)^{\frac{q(p-1)}{p-q}} 
	\, \bigg( \int_t^{x + t} \bigg( \int_t^{z} a \bigg)^q w(z) \,dz \bigg)^{\frac{q}{p-q}} \bigg( \int_t^{x + t} a\bigg)^q w(x + t) \,dx\bigg)^{\frac{p-q}{pq}}  \\
	& \hspace{-14.8cm}  = \bigg( \int_0^{\infty} \bigg( \int_{[x,\infty)} \, d \, \bigg( - \bigg(\sup_{y + t \le \tau} u(\tau)^{p'} \, \bigg( \int_t^{\tau} v(s)^{1-p'}\,ds \bigg) \bigg) \bigg) \bigg)^{\frac{q(p-1)}{p-q}} 
	\, \bigg( \int_t^{x + t} \bigg( \int_t^{z} a \bigg)^q w(z) \,dz \bigg)^{\frac{q}{p-q}} \bigg( \int_t^{x + t} a\bigg)^q w(x + t) \,dx\bigg)^{\frac{p-q}{pq}}  \\
	& \hspace{-14.8cm}  = \bigg( \int_0^{\infty} \bigg( \int_{[x + t,\infty)} \, d \, \bigg( - \bigg(\sup_{y \le \tau} u(\tau)^{p'} \, \bigg( \int_t^{\tau} v(s)^{1-p'}\,ds \bigg) \bigg) \bigg) \bigg)^{\frac{q(p-1)}{p-q}} 
	\, \bigg( \int_t^{x + t} \bigg( \int_t^{z} a \bigg)^q w(z) \,dz \bigg)^{\frac{q}{p-q}} \bigg( \int_t^{x + t} a\bigg)^q w(x + t) \,dx\bigg)^{\frac{p-q}{pq}}  \\
	& \hspace{-14.8cm}  = \bigg( \int_t^{\infty} \bigg( \int_{[x,\infty)} \, d \, \bigg( - \bigg(\sup_{y \le \tau} u(\tau)^{p'} \, \bigg( \int_t^{\tau} v(s)^{1-p'}\,ds \bigg) \bigg) \bigg) \bigg)^{\frac{q(p-1)}{p-q}} 
	\, \bigg( \int_t^{x} \bigg( \int_t^{z} a \bigg)^q w(z) \,dz \bigg)^{\frac{q}{p-q}} \bigg( \int_t^{x} a\bigg)^q w(x) \,dx\bigg)^{\frac{p-q}{pq}},
	\end{align*}
	
	\begin{align*}
	\bigg( \int_0^{\infty} \bigg( \int_{(0,x]} \,  A_t(y)^{p'} \, d \, \bigg( - \bigg(\sup_{y \le \tau} u_t(\tau)^{p'} \, \bigg( \int_0^{\tau} v_t(s)^{1-p'}\,ds \bigg) \bigg) \bigg)  \bigg)^{\frac{q(p-1)}{p-q}} \, \bigg( \int_x^{\infty} w_t(z) \,dz \bigg)^{\frac{q}{p-q}} w_t(x)\,dx \bigg)^{\frac{p-q}{pq}} & \\
	& \hspace{-13cm} = \bigg( \int_0^{\infty} \bigg( \int_{(0,x]} \,  \bigg( \int_t^{y + t} a \bigg)^{p'} \, d \, \bigg( - \bigg(\sup_{y \le \tau} u(\tau + t)^{p'} \, \bigg( \int_0^{\tau} v(s + t)^{1-p'}\,ds \bigg) \bigg) \bigg)  \bigg)^{\frac{q(p-1)}{p-q}} \, \bigg( \int_x^{\infty} w(z + t) \,dz \bigg)^{\frac{q}{p-q}} w(x + t)\,dx \bigg)^{\frac{p-q}{pq}} \\
	& \hspace{-13cm} = \bigg( \int_0^{\infty} \bigg( \int_{(0,x]} \,  \bigg( \int_t^{y + t} a \bigg)^{p'} \, d \, \bigg( - \bigg(\sup_{y + t \le \tau} u(\tau)^{p'} \, \bigg( \int_t^{\tau} v(s)^{1-p'}\,ds \bigg) \bigg) \bigg)  \bigg)^{\frac{q(p-1)}{p-q}} \, \bigg( \int_{x + t}^{\infty} w(z) \,dz \bigg)^{\frac{q}{p-q}} w(x + t)\,dx \bigg)^{\frac{p-q}{pq}} \\
	& \hspace{-13cm} = \bigg( \int_0^{\infty} \bigg( \int_{(t,x + t]} \,  \bigg( \int_t^{y} a \bigg)^{p'} \, d \, \bigg( - \bigg(\sup_{y \le \tau} u(\tau)^{p'} \, \bigg( \int_t^{\tau} v(s)^{1-p'}\,ds \bigg) \bigg) \bigg)  \bigg)^{\frac{q(p-1)}{p-q}} \, \bigg( \int_{x + t}^{\infty} w(z) \,dz \bigg)^{\frac{q}{p-q}} w(x + t)\,dx \bigg)^{\frac{p-q}{pq}} \\
	& \hspace{-13cm} = \bigg( \int_t^{\infty} \bigg( \int_{(t,x]} \,  \bigg( \int_t^{y} a \bigg)^{p'} \, d \, \bigg( - \bigg(\sup_{y \le \tau} u(\tau)^{p'} \, \bigg( \int_t^{\tau} v(s)^{1-p'}\,ds \bigg) \bigg) \bigg)  \bigg)^{\frac{q(p-1)}{p-q}} \, \bigg( \int_{x}^{\infty} w(z) \,dz \bigg)^{\frac{q}{p-q}} w(x)\,dx \bigg)^{\frac{p-q}{pq}},
	\end{align*}
	we arrive at 
	\begin{align*}
	\sup_{f \in \mp^+ (t,\infty)} \frac{\bigg( \int_t^{\infty} \bigg( \int_t^x \bigg( \sup_{s \le \tau}u(\tau) \bigg(\int_t^{\tau} f(y)\,dy \bigg) \bigg) a(s)\,ds \bigg)^q w(x)\,dx \bigg)^{\frac{1}{q}}}{\bigg( \int_t^{\infty} f(s)^pv(s)\,ds\bigg)^{\frac{1}{p}}} & \\
	& \hspace{-9cm} \approx \, \bigg( \int_t^{\infty}  \bigg( \int_t^y v(x)^{1-p'} \,dx \bigg)^{\frac{p(q-1)}{p-q}} v(y)^{1-p'} \, \bigg( \int_{y}^{\infty} \bigg( \int_{y}^{z} \bigg( \sup_{s \le \tau} u(\tau)\bigg) a(s)\,ds \bigg)^{q} w(z)\, dz \bigg)^{\frac{p}{p-q}} \, dy \bigg)^{\frac{p-q}{pq}} \\
	& \hspace{-8.5cm} + \bigg( \int_t^{\infty} \bigg( \int_t^{y} v(x)^{1-p'} \bigg( \int_{x}^{y} \bigg( \sup_{s \le \tau} u(\tau) \bigg) a(s)\,ds \bigg)^{p'} \, dx \bigg)^{\frac{q(p-1)}{p-q}} \bigg( \int_{y}^{\infty} w(z) \,dz \bigg)^{\frac{q}{p-q}} w(y)\,dy \bigg)^{\frac{p-q}{pq}} \\
	& \hspace{-8.5cm} + \bigg( \int_t^{\infty} \bigg( \int_{[x,\infty)} \, d \, \bigg( - \bigg(\sup_{y \le \tau} u(\tau)^{p'} \, \bigg( \int_t^{\tau} v(s)^{1-p'}\,ds \bigg) \bigg) \bigg) \bigg)^{\frac{q(p-1)}{p-q}} 
	\, \bigg( \int_t^{x} \bigg( \int_t^{z} a \bigg)^q w(z) \,dz \bigg)^{\frac{q}{p-q}} \bigg( \int_t^{x} a\bigg)^q w(x) \,dx\bigg)^{\frac{p-q}{pq}} \\
	& \hspace{-8.5cm} + \bigg( \int_t^{\infty} \bigg( \int_{(t,x]} \,  \bigg( \int_t^{y} a \bigg)^{p'} \, d \, \bigg( - \bigg(\sup_{y \le \tau} u(\tau)^{p'} \, \bigg( \int_t^{\tau} v(s)^{1-p'}\,ds \bigg) \bigg) \bigg)  \bigg)^{\frac{q(p-1)}{p-q}} \, \bigg( \int_{x}^{\infty} w(z) \,dz \bigg)^{\frac{q}{p-q}} w(x)\,dx \bigg)^{\frac{p-q}{pq}} \\
	& \hspace{-8.5cm} + \bigg( \int_t^{\infty} \bigg( \int_t^{z} a \bigg)^q w(z) \,dz \bigg)^{\frac{1}{q}} \, \lim_{x \rightarrow \infty} \bigg(\sup_{x \le \tau} u(\tau) \bigg( \int_t^{\tau} v(s)^{1-p'}\,ds \bigg)^{\frac{1}{p'}}\bigg).
	\end{align*}
	The proof is completed.
\end{proof}	

\begin{theorem}\label{aux.thm.3} 
	Let $1 < p,\, q < \infty$. Given $t \ge 0$ assume that $u \in \W (t,\infty) \cap C (t,\infty)$ and  $a, v,\,w \in \W (t,\infty)$. Moreover, assume that $0 < \int_{x}^{\infty} v(\tau)^{1-p'}\,d\tau< \infty,\, x > t$.
	Denote by
	$$
	{\mathcal V}(x) : = \bigg( \int_{x}^{\infty} v^{1 - p'}\bigg)^{- \frac{2p'}{p' + 1}}\,v(x)^{1 - p'}, \quad {\mathcal U} (x) : = u(x)\bigg( \int_x^{\infty} v^{1 - p'}\bigg)^{\frac{2}{p' + 1}}, \qquad x \in (0,\infty).
	$$
	
	\begin{itemize}
		\item If $p\le q$, then 
		\begin{align*}
		\sup_{f \in \mp^+ (t,\infty)} \frac{\bigg( \int_t^{\infty} \bigg( \int_t^y \bigg( \sup_{s \le \tau}u(\tau)  \bigg(\int_{\tau}^{\infty} f(z)\,dz \bigg) \bigg) a(s)\,ds \bigg)^q w(y)\,dy \bigg)^{\frac{1}{q}}}{\bigg( \int_t^{\infty} f(s)^pv(s)\,ds\bigg)^{\frac{1}{p}}} & \\
		&\hspace{-9cm} \approx \, \sup_{y \in (t,\infty)} \bigg(\int_{t}^{y} {\mathcal V}(x) \bigg(\int_{x}^{y}
		\bigg(\sup_{s \le \tau} {\mathcal U} (\tau) \bigg) a(s)\,ds\bigg)^{p'}\,dx\bigg)^\frac{1}{p'}\bigg(\int_{y}^{\infty}\,w(x)\,dx\bigg)^\frac{1}{q} \\
		&\hspace{-8.5cm} + \, \sup_{y \in (t,\infty)} \bigg(\int_{t}^{y} {\mathcal V}(x)\,dx\bigg)^\frac{1}{p'} \bigg(\int_{y}^{\infty}\bigg(\int_{y}^{z}\bigg(\sup_{s \le \tau} {\mathcal U} (\tau) \bigg)a(s)\,ds\bigg)^{q}\,w(z)\,dz\bigg)^\frac{1}{q} \\
		&\hspace{-8.5cm} + \sup_{x \in (t,\infty)} \bigg( \int_{[x,\infty)} \, d \, \bigg( - \sup_{s \le \tau} {\mathcal U} (\tau)^{p'} \, \bigg( \int_t^{\tau}{\mathcal V}(y)\,dy \bigg) \bigg) \bigg)^{\frac{1}{p'}} \, \bigg( \int_t^{x} \bigg( \int_{t}^{y} a\bigg)^q \,w(y)\,dy\bigg)^{\frac{1}{q}} \\
		&\hspace{-8.5cm} + \sup_{x \in (t,\infty)} \bigg( \int_{(t,x]} \bigg(\int_{t}^{s} a\bigg)^{p'} \, d \, \bigg( - \sup_{s \le \tau} {\mathcal U} (\tau)^{p'}\,\bigg( \int_t^{\tau} {\mathcal V}(y)\,dy  \bigg) \bigg) \bigg)^{\frac{1}{p'}} \, \bigg( \int_{x}^{\infty}\,w(z)\,dz\bigg)^{\frac{1}{q}} \\
		&\hspace{-8.5cm} + \bigg( \int_t^{\infty} \bigg(\int_{t}^{z} a\bigg) ^q \,w(z)\,dz\bigg)^{\frac{1}{q}}\lim_{s\rightarrow\infty}\bigg(\sup_{s \le\tau } {\mathcal U} (\tau) \bigg(\int_t^{\tau}{\mathcal V}(y) \,dy\bigg)^{\frac{1}{p'}}\bigg) \\
		&\hspace{-8.5cm} + \bigg( \int_t^{\infty} v(s)^{1 - p'}\,ds\bigg)^{- \frac{1}{p(p' + 1)}} \bigg(\int_t^{\infty}\bigg(\int_t^{x}\bigg(\sup_{s \le\tau } {\mathcal U} (\tau) \bigg)a(s)\, ds\bigg)^q \,w(x)\,dx\bigg)^{\frac{1}{q}}.
		\end{align*}
		
		\item If $q<p$, then
		\begin{align*}
		\sup_{f \in \mp^+ (t,\infty)} \frac{\bigg( \int_t^{\infty} \bigg( \int_t^y \bigg( \sup_{s \le \tau}u(\tau)  \bigg(\int_{\tau}^{\infty} f(z)\,dz \bigg) \bigg) a(s)\,ds \bigg)^q w(y)\,dy \bigg)^{\frac{1}{q}}}{\bigg( \int_t^{\infty} f(s)^pv(s)\,ds\bigg)^{\frac{1}{p}}} & \\
		&\hspace{-9cm} \approx \, \bigg(\int_{t}^{\infty}\bigg(\int_{t}^{y} {\mathcal V}(x)\,dx\bigg)^\frac{p(q-1)}{p-q} {\mathcal V}(y) \bigg(\int_{y}^{\infty}\bigg(\int_{y}^{z}\bigg(\sup_{s \le \tau} {\mathcal U}(\tau) \bigg)a(s)\,ds\bigg)^{q}\,w(z)\,dz\bigg)^\frac{p}{p-q}\,dy\bigg)^\frac{p-q}{pq} \\
		&\hspace{-8.5cm} + \, \bigg(\int_{t}^{\infty}\bigg(\int_{t}^{y} {\mathcal V} (x) \bigg(\int_{x}^{y}\bigg(\sup_{s \le \tau} {\mathcal U} (\tau)\bigg)a(s)\,ds\bigg)^{p'}\,dx\bigg)^\frac{q(p-1)}{p-q}\bigg(\int_{y}^{\infty}\,w(z)\,dz\bigg)^\frac{q}{p-q}\,w(y)\,dy\bigg)^\frac{p-q}{pq} \\
		&\hspace{-8.5cm} + \bigg(\int_{t}^{\infty} \bigg( \int_{[x,\infty)} \, d \, \bigg( -\bigg( \sup_{s \le \tau} {\mathcal U}(\tau)^{p'} \bigg( \int_t^{\tau} {\mathcal V} (y) \,dy \bigg) \bigg) \bigg)\bigg)^\frac{q(p-1)}{p-q}\, \bigg( \int_t^{x} \bigg( \int_{t}^{z} a \bigg)^q \,w(z)\,dz\bigg)^{\frac{q}{p-q}} \bigg( \int_{t}^{x} a \bigg)^q \,w(x)\,dx\bigg)^\frac{p-q}{pq} \\
		&\hspace{-8.5cm} + \bigg(\int_{t}^{\infty} \bigg( \int_{(t,x]} \bigg( \int_{t}^{s} a \bigg)^{p'}\, d \, \bigg( -\bigg( \sup_{s \le \tau} {\mathcal U}(\tau)^{p'} \bigg( \int_t^{\tau} {\mathcal V} (y) \,dy \bigg) \bigg) \bigg)\bigg)^\frac{q(p-1)}{p-q}\, \bigg( \int_{x}^{\infty} w(z)\,dz\bigg)^{\frac{q}{p-q}} w(x)\,dx\bigg)^\frac{p-q}{pq} \\
		&\hspace{-8.5cm} + \bigg( \int_t^{\infty} \bigg(\int_{t}^{z} a\bigg) ^q \,w(z)\,dz\bigg)^{\frac{1}{q}}\lim_{s\rightarrow\infty}\bigg(\sup_{s \le\tau } {\mathcal U} (\tau) \bigg(\int_t^{\tau}{\mathcal V}(y) \,dy\bigg)^{\frac{1}{p'}}\bigg) \\
		&\hspace{-8.5cm} + \bigg( \int_t^{\infty} v(s)^{1 - p'}\,ds\bigg)^{- \frac{1}{p(p' + 1)}} \bigg(\int_t^{\infty}\bigg(\int_t^{x}\bigg(\sup_{s \le\tau } {\mathcal U} (\tau) \bigg)a(s)\, ds\bigg)^q \,w(x)\,dx\bigg)^{\frac{1}{q}}.
		\end{align*}
	\end{itemize}
	
\end{theorem}

\begin{proof}
	The statement was formulated in \cite[Theorem 3.3]{musbil_2} for $t = 0$. Note that, given a point $t \in (0,\infty)$, the inequality
	\begin{equation} \label{ineq1}
	\bigg( \int_t^{\infty} \bigg( \int_t^y \bigg( \sup_{s \le \tau}u(\tau)  \bigg(\int_{\tau}^{\infty} f(z)\,dz \bigg) \bigg) a(s)\,ds \bigg)^q w(y)\,dy \bigg)^{\frac{1}{q}} \le C \,\bigg( \int_t^{\infty} f(s)^pv(s)\,ds\bigg)^{\frac{1}{p}}, \qquad f \in \mp^+ (t,\infty)
	\end{equation}
	is equivalent to the inequality
	\begin{equation}\label{eq2}
	\bigg(\int_{0}^{\infty}\bigg(\int_0^{y} \bigg( \sup_{s \le \tau}u_t(\tau) \bigg(\int_{\tau}^{\infty} f(z)\,dz\bigg) \bigg) a_t(s) \,ds \bigg)^q\,w_t(y) \, dy \bigg)^\frac{1}{q} \le C\, \bigg(\int_0^{\infty} f(s)^p v_t(s)\,ds\bigg)^\frac{1}{p}, \qquad f \in \mp^+ (0,\infty).
	\end{equation}
	Indeed: Applying changes of variables, we get that 
	\begin{alignat*}{2}
	\int_{t}^{\infty}\bigg(\int_t^y \bigg( \sup_{s \le \tau}u(\tau) \bigg(\int_{\tau}^{\infty} f(z)\,dz \bigg) \bigg) a(s) \,ds\bigg)^q w(y) \, dy & \\
	& \hspace{-5cm} = \int_{t}^{\infty}\bigg(\int_t^y \bigg( \sup_{s \le \tau}u(\tau) \bigg(\int_{\tau-t}^{\infty} f(z+t)\,dz \bigg) \bigg)a(s) \,ds \bigg)^q w(y) \, dy \\
	& \hspace{-5cm} = \int_{t}^{\infty}\bigg(\int_t^y \bigg( \sup_{s-t \le \tau-t}u(\tau) \bigg( \int_{\tau-t}^{\infty} f(z+t)\,dz\bigg)\bigg) a(s) \,ds \bigg)^q w(y) \, dy \\
	& \hspace{-5cm} = \int_{t}^{\infty}\bigg(\int_t^y \bigg( \sup_{s-t \le \tau}u(\tau+t) \bigg(\int_{\tau}^{\infty} f(z+t)\,dz\bigg)\bigg) a(s) \,ds \bigg)^q w(y) \, dy \\
	& \hspace{-5cm} = \int_{t}^{\infty}\bigg(\int_0^{y-t} \bigg( \sup_{s \le \tau}u(\tau+t) \bigg( \int_{\tau}^{\infty} f(z+t)\,dz\bigg)\bigg) a(s+t) \,ds \bigg)^q w(y) \, dy \\
	& \hspace{-5cm} = \int_{0}^{\infty}\bigg(\int_0^{y} \bigg( \sup_{s \le \tau}u(\tau+t) \bigg(\int_{\tau}^{\infty} f(z+t)\,dz\bigg) \bigg) a(s+t)\,ds \bigg)^q\,w(y + t) \, dy \\
	& \hspace{-5cm} = \int_{0}^{\infty}\bigg(\int_0^{y} \bigg( \sup_{s \le \tau}u_t(\tau) \bigg(\int_{\tau}^{\infty} f(z+t)\,dz\bigg) \bigg) a_t(s)\,ds \bigg)^q\,w_t(y) \, dy
	\end{alignat*}
	and
	\begin{alignat*}{2}
	\int_t^{\infty} f(s)^p v(s)\,ds = \int_0^{\infty} f(s+t)^p v(s+t)\,ds = \int_0^{\infty} f(s+t)^p v_t(s)\,ds.
	\end{alignat*}
	Hence inequality \eqref{ineq1} can be rewritten as follows:
	\begin{equation*}
	\bigg(\int_{0}^{\infty}\bigg(\int_0^{y} \bigg( \sup_{s \le \tau}u_t(\tau) \bigg(\int_{\tau}^{\infty} f(z+t)\,dz\bigg) \bigg) a_t(s) \,ds \bigg)^q\,w_t(y) \, dy \bigg)^\frac{1}{q} \le C\, \bigg(\int_0^{\infty} f(s+t)^p v_t(s)\,ds\bigg)^\frac{1}{p}, \qquad f \in \mp^+ (t,\infty).
	\end{equation*}
	It remains note that the latter is equivalent to inequality \eqref{eq2}.
	
	Let $p\le q$. By \cite[Theorem 3.3]{musbil_2}, we have for any $t \in (0,\infty)$ that
	\begin{align*}
	\sup_{f \in \mp^+ (0,\infty)} \frac{\bigg(\int_{0}^{\infty}\bigg(\int_0^{y} \bigg( \sup_{s \le \tau}u_t(\tau) \bigg(\int_{\tau}^{\infty} f(z)\,dz\bigg) \bigg) a_t(s) \,ds \bigg)^q\,w_t(y) \, dy \bigg)^\frac{1}{q}}{\bigg(\int_0^{\infty} f(s)^p v_t(s)\,ds\bigg)^\frac{1}{p}} & \notag \\
	& \hspace{-9.5cm} \approx \sup_{y \in (0,\infty)} \bigg(\int_{0}^{y} \Psi_t(x)^{-p'} \psi_t(x)\bigg(\int_{x}^{y}\bigg(\sup_{s \le \tau}u_t(\tau)\Psi_t(\tau)^2\bigg)a_t(s)\,ds\bigg)^{p'}\,dx\bigg)^\frac{1}{p'}\bigg(\int_{y}^{\infty}\,w_t(x)\,dx\bigg)^\frac{1}{q}	\notag \\
	& \hspace{-9cm} + \sup_{y \in (0,\infty)} \bigg(\int_{0}^{y} \Psi_t(x)^{-p'} \psi_t(x)\bigg)^\frac{1}{p'} \bigg(\int_{y}^{\infty}\bigg(\int_{y}^{z}\bigg(\sup_{s \le \tau} u_t(\tau)\Psi_t(\tau)^2\bigg)a_t(s)\,ds\bigg)^{q}\,w_t(z)\,dz\bigg)^\frac{1}{q} \notag \\
	& \hspace{-9cm} + \sup_{x \in (0,\infty)} \bigg( \int_{[x,\infty)} \, d \, \bigg( - \sup_{s \le \tau} u_t(\tau)^{p'}\,\Psi_t(\tau)^{2p'} \bigg( \int_0^{\tau}\Psi_t(y)^{-p'} \psi_t(y) \,dy \bigg) \bigg) \bigg)^{\frac{1}{p'}} \, \bigg( \int_0^{x} A_t(y)^q \,w_t(y)\,dy\bigg)^{\frac{1}{q}} \notag \\
	& \hspace{-9cm} + \sup_{x \in (0,\infty)} \bigg( \int_{(0,x]} A_t(s)^{p'} \, d \, \bigg( - \sup_{s \le \tau} u_t(\tau)^{p'}\,\Psi_t(\tau)^{2p'} \bigg( \int_0^{\tau}\Psi_t(y)^{-p'} \psi_t(y) \,dy \bigg) \bigg) \bigg)^{\frac{1}{p'}} \, \bigg( \int_x^{\infty}\,w_t(z)\,dz\bigg)^{\frac{1}{q}} \notag \\
	& \hspace{-9cm} + \bigg( \int_0^{\infty} A_t(z)^q \,w_t(z)\,dz\bigg)^{\frac{1}{q}}\lim_{s\rightarrow\infty}\bigg(\sup_{s \le\tau }u_t(\tau)\Psi_t(\tau)^2\bigg(\int_0^{\tau}\Psi_t(y)^{-p'} \psi_t(y) \,dy\bigg)^{\frac{1}{p'}}\bigg) \notag \\
	& \hspace{-9cm} + \bigg( \int_0^{\infty} \psi_t(z)\,dz\bigg)^{-\frac{1}{p}} \bigg(\int_0^{\infty}\bigg(\int_0^{x}\bigg(\sup_{s \le\tau }u_t(\tau)\Psi_t(\tau)^2\bigg)a_t(s)\, ds\bigg)^q \,w_t(x)\,dx\bigg)^{\frac{1}{q}}, 	
	\end{align*}
	where
	$$
	A_t(x) : = \int_{0}^{x} a_t(z)\,dz, \quad \psi_t(x) : = \bigg( \int_x^{\infty} v_t(s)^{1 - p'}\,ds\bigg)^{- \frac{p'}{p' + 1}}\,v_t(x)^{1 - p'}, \quad \Psi_t(x) : = \bigg( \int_x^{\infty} v_t(s)^{1 - p'}\,ds\bigg)^{\frac{1}{p' + 1}}, \quad x\in (0,\infty).
	$$
	Using changes of variables, we get for any $x > 0$ that	
	$$ 
	A_t(x) = \int_{0}^{x} a(z+t)\,dz = \int_{t}^{x+t} a(z)\,dz, 
	$$
	\begin{align*}
	\psi_t(x) & =\bigg( \int_x^{\infty} v(s+t)^{1 - p'}\,ds\bigg)^{- \frac{p'}{p' + 1}}\,v(x+t)^{1 - p'} \\
	& = \bigg( \int_{x+t}^{\infty} v(s)^{1 - p'}\,ds\bigg)^{- \frac{p'}{p' + 1}}\,v(x+t)^{1 - p'},
	\end{align*}
	and 
	\begin{align*}
	\Psi_t(x) & =\bigg( \int_{x+t}^{\infty} v(s)^{1 - p'}\,ds\bigg)^{\frac{1}{p' + 1}}.
	\end{align*}
	
	Since
	\begin{align*}
	\sup_{y \ge  0} \bigg(\int_{0}^{y} \Psi_t(x)^{-p'} \psi_t(x)\bigg(\int_{x}^{y}\bigg(\sup_{s \le \tau}u_t(\tau)\Psi_t(\tau)^2\bigg)a_t(s)\,ds\bigg)^{p'}\,dx\bigg)^\frac{1}{p'}\bigg(\int_{y}^{\infty}\,w_t(x)\,dx\bigg)^\frac{1}{q} & \\
	& \hspace{-11cm} = \, \sup_{y \in (0,\infty)} \bigg(\int_{0}^{y} {\mathcal V}(x + t) \bigg(\int_{x}^{y}\bigg(\sup_{s \le\tau} {\mathcal U}(\tau+t) \bigg) a(s+t)\,ds\bigg)^{p'}\,dx\bigg)^\frac{1}{p'}\bigg(\int_{y}^{\infty}\,w(x+t)\,dx\bigg)^\frac{1}{q} \\
	&\hspace{-11cm} = \, \sup_{y \in (0,\infty)} \bigg(\int_{0}^{y} {\mathcal V}(x + t) \bigg(\int_{x}^{y}
	\bigg(\sup_{s+t \le \tau}{\mathcal U} (\tau) \bigg) a(s+t)\,ds\bigg)^{p'}\,dx\bigg)^\frac{1}{p'}\bigg(\int_{y}^{\infty}\,w(x+t)\,dx\bigg)^\frac{1}{q}\\
	&\hspace{-11cm} = \, \sup_{y \in (0,\infty)} \bigg(\int_{0}^{y} {\mathcal V}(x + t) \bigg(\int_{x+t}^{y+t}
	\bigg(\sup_{s \le \tau} {\mathcal U} (\tau) \bigg) a(s)\,ds\bigg)^{p'} \,dx\bigg)^\frac{1}{p'}\bigg(\int_{y}^{\infty}\,w(x+t)\,dx\bigg)^\frac{1}{q}\\
	&\hspace{-11cm} = \, \sup_{y \in (0,\infty)} \bigg(\int_{t}^{y+t} {\mathcal V}(x) \bigg(\int_{x}^{y+t}
	\bigg(\sup_{s \le \tau} {\mathcal U} (\tau) \bigg) a(s)\,ds\bigg)^{p'}\,dx\bigg)^\frac{1}{p'}\bigg(\int_{y+t}^{\infty}\,w(x)\,dx\bigg)^\frac{1}{q}\\
	&\hspace{-11cm} = \, \sup_{y \in (t,\infty)} \bigg(\int_{t}^{y} {\mathcal V}(x) \bigg(\int_{x}^{y}
	\bigg(\sup_{s \le \tau} {\mathcal U} (\tau) \bigg) a(s)\,ds\bigg)^{p'}\,dx\bigg)^\frac{1}{p'}\bigg(\int_{y}^{\infty}\,w(x)\,dx\bigg)^\frac{1}{q},
	\end{align*}
	
	\begin{align*}
	\sup_{y \in (0,\infty)} \bigg(\int_{0}^{y} \Psi_t(x)^{-p'} \psi_t(x)\bigg)^\frac{1}{p'} \bigg(\int_{y}^{\infty}\bigg(\int_{y}^{z}\bigg(\sup_{s \le \tau} u_t(\tau)\Psi_t(\tau)^2\bigg)a_t(s)\,ds\bigg)^{q}\,w_t(z)\,dz\bigg)^\frac{1}{q} & \\
	& \hspace{-11cm} = \, \sup_{y \in (0,\infty)} \bigg(\int_{0}^{y} {\mathcal V}(x + t)\,dx \bigg)^\frac{1}{p'} \bigg(\int_{y}^{\infty}\bigg(\int_{y}^{z}\bigg(\sup_{s \le \tau} {\mathcal U} (\tau+t) \bigg)a(s+t)\,ds\bigg)^{q}\,w(z+t)\,dz\bigg)^\frac{1}{q} \\
	& \hspace{-11cm} = \, \sup_{y \in (0,\infty)} \bigg(\int_{0}^{y} {\mathcal V}(x + t)\,dx \bigg)^\frac{1}{p'} \bigg(\int_{y}^{\infty}\bigg(\int_{y}^{z}\bigg(\sup_{s+t \le \tau} {\mathcal U} (\tau) \bigg)a(s+t)\,ds\bigg)^{q}\,w(z+t)\,dz\bigg)^\frac{1}{q} \\
	&\hspace{-11cm} = \, \sup_{y \in (0,\infty)} \bigg(\int_{0}^{y} {\mathcal V}(x + t) \,dx\bigg)^\frac{1}{p'} \bigg(\int_{y}^{\infty}\bigg(\int_{y+t}^{z+t}\bigg(\sup_{s \le \tau} {\mathcal U} (\tau) \bigg)a(s)\,ds\bigg)^{q}\,w(z+t)\,dz\bigg)^\frac{1}{q} \\
	&\hspace{-11cm} = \, \sup_{y \in (0,\infty)} \bigg(\int_{t}^{y+t} {\mathcal V}(x)\,dx \bigg)^\frac{1}{p'} \bigg(\int_{y}^{\infty}\bigg(\int_{y+t}^{z+t}\bigg(\sup_{s \le \tau} {\mathcal U} (\tau) \bigg)a(s)\,ds\bigg)^{q}\,w(z+t)\,dz\bigg)^\frac{1}{q}\\
	&\hspace{-11cm} = \, \sup_{y \in (0,\infty)} \bigg(\int_{t}^{y+t} {\mathcal V}(x)\,dx \bigg)^\frac{1}{p'} \bigg(\int_{y+t}^{\infty}\bigg(\int_{y+t}^{z}\bigg(\sup_{s \le \tau} {\mathcal U} (\tau) \bigg)a(s)\,ds\bigg)^{q}\,w(z)\,dz\bigg)^\frac{1}{q}\\
	&\hspace{-11cm} = \, \sup_{y \in (t,\infty)} \bigg(\int_{t}^{y} {\mathcal V}(x)\,dx \bigg)^\frac{1}{p'} \bigg(\int_{y}^{\infty}\bigg(\int_{y}^{z}\bigg(\sup_{s \le \tau} {\mathcal U} (\tau) \bigg)a(s)\,ds\bigg)^{q}\,w(z)\,dz\bigg)^\frac{1}{q},
	\end{align*}
	
	\begin{align*}
	\sup_{x \in (0,\infty)} \bigg( \int_{[x,\infty)} \, d \, \bigg( - \sup_{s \le \tau} u_t(\tau)^{p'}\,\Psi_t(\tau)^{2p'} \bigg( \int_0^{\tau}\Psi_t(y)^{-p'} \psi_t(y) \,dy \bigg) \bigg) \bigg)^{\frac{1}{p'}} \, \bigg( \int_0^{x} A_t(y)^q \,w_t(y)\,dy\bigg)^{\frac{1}{q}} & \\
	&\hspace{-12.5cm} =\sup_{x \in (0,\infty)} \bigg( \int_{[x,\infty)} \, d \, \bigg( - \sup_{s \le \tau} {\mathcal U} (\tau+t)^{p'} \bigg( \int_0^{\tau} {\mathcal V} (y + t) \,dy \bigg) \bigg) \bigg)^{\frac{1}{p'}} \, \bigg( \int_t^{x+t} \bigg( \int_{t}^{y+t} a\bigg)^q \,w(y+t)\,dy\bigg)^{\frac{1}{q}}\\
	&\hspace{-12.5cm} =\sup_{x \in (0,\infty)} \bigg( \int_{[x,\infty)} \, d \, \bigg( - \sup_{s \le \tau} {\mathcal U} (\tau+t)^{p'} \bigg( \int_t^{\tau+t} {\mathcal V} (y)\,dy \bigg) \bigg) \bigg)^{\frac{1}{p'}} \, \bigg( \int_t^{x+t} \bigg( \int_{t}^{y} a\bigg)^q \,w(y)\,dy\bigg)^{\frac{1}{q}}\\
	&\hspace{-12.5cm} =\sup_{x \in (0,\infty)} \bigg( \int_{[x,\infty)} \, d \, \bigg( - \sup_{s+t \le \tau} {\mathcal U} (\tau)^{p'} \bigg( \int_t^{\tau} {\mathcal V} (y) \,dy \bigg) \bigg) \bigg)^{\frac{1}{p'}} \, \bigg( \int_t^{x+t} \bigg( \int_{t}^{y} a\bigg)^q \,w(y)\,dy\bigg)^{\frac{1}{q}}\\
	&\hspace{-12.5cm} =\sup_{x \in (0,\infty)} \bigg( \int_{[x+t,\infty)} \, d \, \bigg( - \sup_{s \le \tau} {\mathcal U} (\tau)^{p'} \bigg( \int_t^{\tau} {\mathcal V} (y) \,dy \bigg) \bigg) \bigg)^{\frac{1}{p'}} \, \bigg( \int_t^{x+t} \bigg( \int_{t}^{y} a\bigg)^q \,w(y)\,dy\bigg)^{\frac{1}{q}}\\
	&\hspace{-12.5cm} = \sup_{x \in (t,\infty)} \bigg( \int_{[x,\infty)} \, d \, \bigg( - \sup_{s \le \tau} {\mathcal U} (\tau)^{p'} \bigg( \int_t^{\tau} {\mathcal V} (y)\,dy \bigg) \bigg) \bigg)^{\frac{1}{p'}} \, \bigg( \int_t^{x} \bigg( \int_{t}^{y} a\bigg)^q \,w(y)\,dy\bigg)^{\frac{1}{q}},
	\end{align*}
	
	\begin{align*}
	\sup_{x \in (0,\infty)} \bigg( \int_{(0,x]} A_t(s)^{p'} \, d \, \bigg( - \sup_{s \le \tau} u_t(\tau)\,\Psi_t(\tau)^{2p'} \bigg( \int_0^{\tau}\Psi_t(y)^{-p'} \psi_t(y) \,dy \bigg) \bigg) \bigg)^{\frac{1}{p'}} \, \bigg( \int_x^{\infty}\,w_t(z)\,dz\bigg)^{\frac{1}{q}} & \\
	&\hspace{-12.5cm} =\sup_{x \in (0,\infty)} \bigg( \int_{(0,x]} \bigg(\int_{t}^{s+t} a\bigg)^{p'} \, d \, \bigg( - \sup_{s \le \tau} {\mathcal U} (\tau+t)^{p'} \bigg( \int_0^{\tau} {\mathcal V} (y + t)\,dy  \bigg) \bigg) \bigg)^{\frac{1}{p'}} \, \bigg( \int_x^{\infty}\,w(z+t)\,dz\bigg)^{\frac{1}{q}}\\
	&\hspace{-12.5cm} =\sup_{x \in (0,\infty)} \bigg( \int_{(0,x]} \bigg(\int_{t}^{s+t} a\bigg)^{p'} \, d \, \bigg( - \sup_{s \le \tau} {\mathcal U} (\tau+t)^{p'} \bigg( \int_t^{\tau+t} {\mathcal V} (y) \,dy  \bigg) \bigg) \bigg)^{\frac{1}{p'}} \, \bigg( \int_{x+t}^{\infty}\,w(z)\,dz\bigg)^{\frac{1}{q}}\\
	&\hspace{-12.5cm} =\sup_{x \in (0,\infty)} \bigg( \int_{(0,x]} \bigg(\int_{t}^{s+t} a\bigg)^{p'} \, d \, \bigg( - \sup_{s+t \le \tau} {\mathcal U} (\tau)^{p'} \bigg( \int_t^{\tau} {\mathcal V} (y) \,dy  \bigg) \bigg) \bigg)^{\frac{1}{p'}} \, \bigg( \int_{x+t}^{\infty}\,w(z)\,dz\bigg)^{\frac{1}{q}}\\
	&\hspace{-12.5cm} =\sup_{x \in (0,\infty)} \bigg( \int_{(t,x+t]} \bigg(\int_{t}^{s} a\bigg)^{p'} \, d \, \bigg( - \sup_{s \le \tau} {\mathcal U} (\tau)^{p'} \bigg( \int_t^{\tau} {\mathcal V} (y) \,dy  \bigg) \bigg) \bigg)^{\frac{1}{p'}} \, \bigg( \int_{x+t}^{\infty}\,w(z)\,dz\bigg)^{\frac{1}{q}}\\
	&\hspace{-12.5cm} = \sup_{x \in (t,\infty)} \bigg( \int_{(t,x]} \bigg(\int_{t}^{s} a\bigg)^{p'} \, d \, \bigg( - \sup_{s \le \tau} {\mathcal U} (\tau)^{p'} \bigg( \int_t^{\tau} {\mathcal V} (y) \,dy  \bigg) \bigg) \bigg)^{\frac{1}{p'}} \, \bigg( \int_{x}^{\infty}\,w(z)\,dz\bigg)^{\frac{1}{q}},
	\end{align*}
	
	\begin{align*}
	\bigg( \int_0^{\infty} A_t(z)^q \,w_t(z)\,dz\bigg)^{\frac{1}{q}}\lim_{s\rightarrow\infty}\bigg(\sup_{s \le\tau }u_t(\tau)\Psi_t(\tau)^2\bigg(\int_0^{\tau}\Psi_t(y)^{-p'} \psi_t(y) \,dy\bigg)^{\frac{1}{p'}}\bigg) & \\
	& \hspace{-10cm} = \bigg( \int_0^{\infty} \bigg(\int_{t}^{z+t} a\bigg) ^q \,w(z+t)\,dz\bigg)^{\frac{1}{q}}\lim_{s\rightarrow\infty}\bigg(\sup_{s \le\tau } {\mathcal U} (\tau+t) \bigg(\int_0^{\tau} {\mathcal V} (y + t) \,dy\bigg)^{\frac{1}{p'}}\bigg)\\
	& \hspace{-10cm} = \bigg( \int_t^{\infty} \bigg(\int_{t}^{z} a\bigg) ^q \,w(z)\,dz\bigg)^{\frac{1}{q}}\lim_{s\rightarrow\infty}\bigg(\sup_{s \le\tau } {\mathcal U} (\tau+t) \bigg(\int_t^{\tau+t} {\mathcal V} (y) \,dy\bigg)^{\frac{1}{p'}}\bigg)\\
	& \hspace{-10cm} = \bigg( \int_t^{\infty} \bigg(\int_{t}^{z} a\bigg) ^q \,w(z)\,dz\bigg)^{\frac{1}{q}}\lim_{s\rightarrow\infty}\bigg(\sup_{s+t \le\tau } {\mathcal U} (\tau) \bigg(\int_t^{\tau} {\mathcal V} (y) \,dy\bigg)^{\frac{1}{p'}}\bigg)\\
	& \hspace{-10cm} = \bigg( \int_t^{\infty} \bigg(\int_{t}^{z} a\bigg) ^q \,w(z)\,dz\bigg)^{\frac{1}{q}}\lim_{s\rightarrow\infty}\bigg(\sup_{s \le\tau } {\mathcal U} (\tau) \bigg(\int_t^{\tau} {\mathcal V} (y) \,dy\bigg)^{\frac{1}{p'}}\bigg),
	\end{align*}
	
	\begin{align*}
	\bigg( \int_0^{\infty} \psi_t(z)\,dz\bigg)^{-\frac{1}{p}} \bigg(\int_0^{\infty}\bigg(\int_0^{x}\bigg(\sup_{s \le\tau }u_t(\tau)\Psi_t(\tau)^2\bigg)a_t(s)\, ds\bigg)^q \,w_t(x)\,dx\bigg)^{\frac{1}{q}} & \\
	&\hspace{-9cm} =	\bigg( \int_0^{\infty} {\mathcal V} (z + t)\,dz\bigg)^{-\frac{1}{p}} \bigg(\int_0^{\infty}\bigg(\int_0^{x}\bigg(\sup_{s \le\tau }{\mathcal U} (\tau+t) \bigg)a(s+t)\, ds\bigg)^q \,w(x+t)\,dx\bigg)^{\frac{1}{q}}\\
	&\hspace{-9cm} =	\bigg( \int_t^{\infty} {\mathcal V} (z) \,dz\bigg)^{-\frac{1}{p}} \bigg(\int_0^{\infty}\bigg(\int_0^{x}\bigg(\sup_{s \le\tau } {\mathcal U} (\tau+t) \bigg)a(s+t)\, ds\bigg)^q \,w(x+t)\,dx\bigg)^{\frac{1}{q}}\\
	&\hspace{-9cm} =	\bigg( \int_t^{\infty} {\mathcal V} (z) \,dz\bigg)^{-\frac{1}{p}} \bigg(\int_0^{\infty}\bigg(\int_0^{x}\bigg(\sup_{s+t \le\tau } {\mathcal U} (\tau) \bigg)a(s+t)\, ds\bigg)^q \,w(x+t)\,dx\bigg)^{\frac{1}{q}}\\
	&\hspace{-9cm} =	\bigg( \int_t^{\infty} {\mathcal V} (z) \,dz\bigg)^{-\frac{1}{p}} \bigg(\int_0^{\infty}\bigg(\int_t^{x+t}\bigg(\sup_{s \le\tau } {\mathcal U} (\tau) \bigg)a(s)\, ds\bigg)^q \,w(x+t)\,dx\bigg)^{\frac{1}{q}}\\
	&\hspace{-9cm} = \bigg( \int_t^{\infty} {\mathcal V} (z) \,dz\bigg)^{-\frac{1}{p}} \bigg(\int_t^{\infty}\bigg(\int_t^{x}\bigg(\sup_{s \le\tau } {\mathcal W} (\tau) \bigg)a(s)\, ds\bigg)^q \,w(x)\,dx\bigg)^{\frac{1}{q}},
	\end{align*}
	
	combining, we arrive at
	\begin{align*}
	\sup_{f \in \mp^+ (t,\infty)} \frac{\bigg( \int_t^{\infty} \bigg( \int_t^y \bigg( \sup_{s \le \tau}u(\tau)  \bigg(\int_{\tau}^{\infty} f(z)\,dz \bigg) \bigg) a(s)\,ds \bigg)^q w(y)\,dy \bigg)^{\frac{1}{q}}}{\bigg( \int_t^{\infty} f(s)^pv(s)\,ds\bigg)^{\frac{1}{p}}} & \\
	&\hspace{-9cm} \approx \, \sup_{y \in (t,\infty)} \bigg(\int_{t}^{y} {\mathcal V}(x) \bigg(\int_{x}^{y}
	\bigg(\sup_{s \le \tau} {\mathcal U} (\tau) \bigg) a(s)\,ds\bigg)^{p'}\,dx\bigg)^\frac{1}{p'}\bigg(\int_{y}^{\infty}\,w(x)\,dx\bigg)^\frac{1}{q} \\
	&\hspace{-8.5cm} + \, \sup_{y \in (t,\infty)} \bigg(\int_{t}^{y} {\mathcal V} (x) \,dx \bigg)^\frac{1}{p'} \bigg(\int_{y}^{\infty}\bigg(\int_{y}^{z}\bigg(\sup_{s \le \tau} {\mathcal U} (\tau) \bigg)a(s)\,ds\bigg)^{q}\,w(z)\,dz\bigg)^\frac{1}{q} \\
	&\hspace{-8.5cm} + \sup_{x \in (t,\infty)} \bigg( \int_{[x,\infty)} \, d \, \bigg( - \sup_{s \le \tau} {\mathcal U} (\tau)^{p'} \bigg( \int_t^{\tau} {\mathcal V} (y) \,dy \bigg) \bigg) \bigg)^{\frac{1}{p'}} \, \bigg( \int_t^{x} \bigg( \int_{t}^{y} a\bigg)^q \,w(y)\,dy\bigg)^{\frac{1}{q}} \\
	&\hspace{-8.5cm} + \sup_{x \in (t,\infty)} \bigg( \int_{(t,x]} \bigg(\int_{t}^{s} a\bigg)^{p'} \, d \, \bigg( - \sup_{s \le \tau} {\mathcal U} (\tau)^{p'} \bigg( \int_t^{\tau} {\mathcal V} (y) \,dy  \bigg) \bigg) \bigg)^{\frac{1}{p'}} \, \bigg( \int_{x}^{\infty}\,w(z)\,dz\bigg)^{\frac{1}{q}} \\
	&\hspace{-8.5cm} + \bigg( \int_t^{\infty} \bigg(\int_{t}^{z} a\bigg) ^q \,w(z)\,dz\bigg)^{\frac{1}{q}}\lim_{s\rightarrow\infty}\bigg(\sup_{s \le\tau } {\mathcal U} (\tau) \bigg(\int_t^{\tau} {\mathcal V} (y) \,dy\bigg)^{\frac{1}{p'}}\bigg) \\
	&\hspace{-8.5cm} + \bigg( \int_t^{\infty} {\mathcal V} (z)\,dz\bigg)^{-\frac{1}{p}} \bigg(\int_t^{\infty}\bigg(\int_t^{x}\bigg(\sup_{s \le\tau } {\mathcal U} (\tau) \bigg)a(s)\, ds\bigg)^q \,w(x)\,dx\bigg)^{\frac{1}{q}}.
	\end{align*}
	
	Let $q < p$. By \cite[Theorem 3.3]{musbil_2}, we have for any $t \in (0,\infty)$ that
	\begin{align*}
	\sup_{f \in \mp^+ (0,\infty)} \frac{\bigg(\int_{0}^{\infty}\bigg(\int_0^{y} \bigg( \sup_{s \le \tau}u_t(\tau) \bigg(\int_{\tau}^{\infty} f(z)\,dz\bigg) \bigg) a_t(s) \,ds \bigg)^q\,w_t(y) \, dy \bigg)^\frac{1}{q}}{\bigg(\int_0^{\infty} f(s)^p v_t(s)\,ds\bigg)^\frac{1}{p}} & \notag \\
	&\hspace{-9.5cm} \approx \, \bigg(\int_{0}^{\infty}\bigg(\int_{0}^{y} \Psi_t(x)^{-p'} \psi_t(x)\bigg)^\frac{p(q-1)}{p-q}\Psi_t(y)^{-p'} \psi_t(y) \bigg(\int_{y}^{\infty}\bigg(\int_{y}^{z}\bigg(\sup_{s \le \tau} u_t(\tau)\Psi_t(\tau)^2\bigg)a_t(s)\,ds\bigg)^{q}\,w_t(z)\,dz\bigg)^\frac{p}{p-q}\,dy\bigg)^\frac{p-q}{pq} \notag \\
	&\hspace{-9.3cm} + \, \bigg(\int_{0}^{\infty}\bigg(\int_{0}^{y} \Psi_t(x)^{-p'} \psi_t(x)\bigg(\int_{x}^{y}\bigg(\sup_{s \le \tau}u_t(\tau)\Psi_t(\tau)^2\bigg)a_t(s)\,ds\bigg)^{p'}\,dx\bigg)^\frac{q(p-1)}{p-q}\bigg(\int_{y}^{\infty}\,w_t(z)\,dz\bigg)^\frac{q}{p-q}\,w_t(y)\,dy\bigg)^\frac{p-q}{pq} \notag \\
	&\hspace{-9.3cm} + \,\bigg(\int_{0}^{\infty} \bigg( \int_{[x,\infty)} \, d \, \bigg( -\bigg( \sup_{s \le \tau} u_t(\tau)^{p'}\,\Psi_t(\tau)^{2p'} \bigg( \int_0^{\tau}\Psi_t(y)^{-p'} \psi_t(y) \,dy \bigg) \bigg) \bigg)\bigg)^{\frac{q(p-1)}{p-q}} \, \bigg( \int_0^{x} A_t(z)^q \,w_t(z)\,dz\bigg)^{\frac{q}{p-q}}A_t(x)^q \,w_t(x)\,dx\bigg)^\frac{p-q}{pq} \notag \\
	&\hspace{-9.3cm} + \,\bigg(\int_{0}^{\infty} \bigg( \int_{(0,x]} A_t(s)^{p'}\, d \, \bigg( -\bigg( \sup_{s \le \tau} u_t(\tau)^{p'}\,\Psi_t(\tau)^{2p'} \bigg( \int_0^{\tau}\Psi_t(y)^{-p'} \psi_t(y) \,dy \bigg) \bigg) \bigg)\bigg)^{\frac{q(p-1)}{p-q}} \, \bigg( \int_x^{\infty} w_t(z)\,dz\bigg)^{\frac{q}{p-q}} w_t(x)\,dx\bigg)^\frac{p-q}{pq} \notag \\
	&\hspace{-9.3cm} +\, \bigg( \int_0^{\infty} A_t(z)^q \,w_t(z)\,dz\bigg)^{\frac{1}{q}}\lim_{s\rightarrow\infty}\bigg(\sup_{s \le\tau }u_t(\tau)\Psi_t(\tau)^2\bigg(\int_0^{\tau}\Psi_t(y)^{-p'} \psi_t(y) \,dy\bigg)^{\frac{1}{p'}}\bigg) \notag \\
	&\hspace{-9.3cm} +\,\bigg( \int_0^{\infty} \psi_t(t) \,dt\bigg)^{-\frac{1}{p}}\bigg(\int_0^{\infty}\bigg(\int_0^{x}\bigg(\sup_{s \le\tau }u_t(\tau)\Psi_t(\tau)^2\bigg)a_t(s)\, ds\bigg)^q \,w_t(x)\,dx\bigg)^{\frac{1}{q}}. 
	\end{align*}
	
	Since
	\begin{align*}
	\bigg(\int_{0}^{\infty}\bigg(\int_{0}^{y} \Psi_t(x)^{-p'} \psi_t(x)\bigg)^\frac{p(q-1)}{p-q}\Psi_t(y)^{-p'} \psi_t(y) \bigg(\int_{y}^{\infty}\bigg(\int_{y}^{z}\bigg(\sup_{s \le \tau} u_t(\tau)\Psi_t(\tau)^2\bigg)a_t(s)\,ds\bigg)^{q}\,w_t(z)\,dz\bigg)^\frac{p}{p-q}\,dy\bigg)^\frac{p-q}{pq} & \\
	&\hspace{-15cm} = \bigg(\int_{0}^{\infty}\bigg(\int_{0}^{y} {\mathcal V}(x + t)\,dx\bigg)^\frac{p(q-1)}{p-q} {\mathcal V}(y + t) \bigg(\int_{y}^{\infty}\bigg(\int_{y}^{z}\bigg(\sup_{s \le \tau} {\mathcal U}(\tau+t) \bigg)a(s + t)\,ds\bigg)^{q}\,w(z + t)\,dz\bigg)^\frac{p}{p-q}\,dy\bigg)^\frac{p-q}{pq} \\
	&\hspace{-15cm} = \bigg(\int_{0}^{\infty}\bigg(\int_{t}^{y + t} {\mathcal V}(x)\,dx\bigg)^\frac{p(q-1)}{p-q} {\mathcal V}(y + t) \bigg(\int_{y}^{\infty}\bigg(\int_{y}^{z}\bigg(\sup_{s + t \le \tau} {\mathcal U}(\tau) \bigg)a(s + t)\,ds\bigg)^{q}\,w(z + t)\,dz\bigg)^\frac{p}{p-q}\,dy\bigg)^\frac{p-q}{pq} \\
	&\hspace{-15cm} = \bigg(\int_{0}^{\infty}\bigg(\int_{t}^{y + t} {\mathcal V}(x)\,dx\bigg)^\frac{p(q-1)}{p-q} {\mathcal V}(y + t) \bigg(\int_{y}^{\infty}\bigg(\int_{y + t}^{z + t}\bigg(\sup_{s \le \tau} {\mathcal U}(\tau) \bigg)a(s)\,ds\bigg)^{q}\,w(z + t)\,dz\bigg)^\frac{p}{p-q}\,dy\bigg)^\frac{p-q}{pq} \\
	&\hspace{-15cm} = \bigg(\int_{0}^{\infty}\bigg(\int_{t}^{y + t} {\mathcal V}(x)\,dx\bigg)^\frac{p(q-1)}{p-q} {\mathcal V}(y + t) \bigg(\int_{y + t}^{\infty}\bigg(\int_{y + t}^{z}\bigg(\sup_{s \le \tau} {\mathcal U}(\tau) \bigg)a(s)\,ds\bigg)^{q}\,w(z)\,dz\bigg)^\frac{p}{p-q}\,dy\bigg)^\frac{p-q}{pq} \\
	&\hspace{-15cm} = \bigg(\int_{t}^{\infty}\bigg(\int_{t}^{y} {\mathcal V}(x)\,dx\bigg)^\frac{p(q-1)}{p-q} {\mathcal V}(y) \bigg(\int_{y}^{\infty}\bigg(\int_{y}^{z}\bigg(\sup_{s \le \tau} {\mathcal U}(\tau) \bigg)a(s)\,ds\bigg)^{q}\,w(z)\,dz\bigg)^\frac{p}{p-q}\,dy\bigg)^\frac{p-q}{pq},
	\end{align*}
	
	\begin{align*}
	\bigg(\int_{0}^{\infty}\bigg(\int_{0}^{y} \Psi_t(x)^{-p'} \psi_t(x)\bigg(\int_{x}^{y}\bigg(\sup_{s \le \tau}u_t(\tau)\Psi_t(\tau)^2\bigg)a_t(s)\,ds\bigg)^{p'}\,dx\bigg)^\frac{q(p-1)}{p-q}\bigg(\int_{y}^{\infty}\,w_t(z)\,dz\bigg)^\frac{q}{p-q}\,w_t(y)\,dy\bigg)^\frac{p-q}{pq} & \\
	&\hspace{-14cm} = \bigg(\int_{0}^{\infty}\bigg(\int_{0}^{y} {\mathcal V} (x + t) \bigg(\int_{x}^{y}\bigg(\sup_{s \le \tau} {\mathcal U} (\tau + t)\bigg)a(s + t)\,ds\bigg)^{p'}\,dx\bigg)^\frac{q(p-1)}{p-q}\bigg(\int_{y}^{\infty}\,w(z + t)\,dz\bigg)^\frac{q}{p-q}\,w(y + t)\,dy\bigg)^\frac{p-q}{pq} \\
	&\hspace{-14cm} = \bigg(\int_{0}^{\infty}\bigg(\int_{0}^{y} {\mathcal V} (x + t) \bigg(\int_{x}^{y}\bigg(\sup_{s + t \le \tau} {\mathcal U} (\tau)\bigg)a(s + t)\,ds\bigg)^{p'}\,dx\bigg)^\frac{q(p-1)}{p-q}\bigg(\int_{y}^{\infty}\,w(z + t)\,dz\bigg)^\frac{q}{p-q}\,w(y + t)\,dy\bigg)^\frac{p-q}{pq} \\
	&\hspace{-14cm} = \bigg(\int_{0}^{\infty}\bigg(\int_{0}^{y} {\mathcal V} (x + t) \bigg(\int_{x + t}^{y + t}\bigg(\sup_{s \le \tau} {\mathcal U} (\tau)\bigg)a(s)\,ds\bigg)^{p'}\,dx\bigg)^\frac{q(p-1)}{p-q}\bigg(\int_{y + t }^{\infty}\,w(z)\,dz\bigg)^\frac{q}{p-q}\,w(y + t)\,dy\bigg)^\frac{p-q}{pq} \\
	&\hspace{-14cm} = \bigg(\int_{0}^{\infty}\bigg(\int_{t}^{y + t} {\mathcal V} (x) \bigg(\int_{x}^{y + t}\bigg(\sup_{s \le \tau} {\mathcal U} (\tau)\bigg)a(s)\,ds\bigg)^{p'}\,dx\bigg)^\frac{q(p-1)}{p-q}\bigg(\int_{y + t }^{\infty}\,w(z)\,dz\bigg)^\frac{q}{p-q}\,w(y + t)\,dy\bigg)^\frac{p-q}{pq} \\
	&\hspace{-14cm} = \bigg(\int_{t}^{\infty}\bigg(\int_{t}^{y} {\mathcal V} (x) \bigg(\int_{x}^{y}\bigg(\sup_{s \le \tau} {\mathcal U} (\tau)\bigg)a(s)\,ds\bigg)^{p'}\,dx\bigg)^\frac{q(p-1)}{p-q}\bigg(\int_{y}^{\infty}\,w(z)\,dz\bigg)^\frac{q}{p-q}\,w(y)\,dy\bigg)^\frac{p-q}{pq},
	\end{align*}
	
	\begin{align*}
	\bigg(\int_{0}^{\infty} \bigg( \int_{[x,\infty)} \, d \, \bigg( -\bigg( \sup_{s \le \tau} u_t(\tau)^{p'}\,\Psi_t(\tau)^{2p'} \bigg( \int_0^{\tau}\Psi_t(y)^{-p'} \psi_t(y) \,dy\bigg) \bigg) \bigg)\bigg)^\frac{q(p-1)}{p-q} \, \bigg( \int_0^{x} A_t(z)^q \,w_t(z)\,dz\bigg)^{\frac{q}{p-q}}A_t(x)^q \,w_t(x)\,dx\bigg)^\frac{p-q}{pq} & \\
	&\hspace{-16.8cm} = \bigg(\int_{0}^{\infty} \bigg( \int_{[x,\infty)} \, d \, \bigg( -\bigg( \sup_{s \le \tau} {\mathcal U}(\tau + t)^{p'} \bigg( \int_0^{\tau} {\mathcal V} (y + t) \,dy \bigg) \bigg) \bigg)\bigg)^\frac{q(p-1)}{p-q} \, \bigg( \int_0^{x} \bigg( \int_{t}^{z + t} a \bigg)^q \,w(z + t)\,dz\bigg)^{\frac{q}{p-q}} \bigg( \int_{t}^{x+t} a \bigg)^q \,w(x + t)\,dx\bigg)^\frac{p-q}{pq} \\
	&\hspace{-16.8cm} = \bigg(\int_{0}^{\infty} \bigg( \int_{[x,\infty)} \, d \, \bigg( -\bigg( \sup_{s \le \tau} {\mathcal U}(\tau + t)^{p'} \bigg( \int_t^{\tau + t} {\mathcal V} (y) \,dy \bigg) \bigg) \bigg)\bigg)^\frac{q(p-1)}{p-q} \, \bigg( \int_t^{x + t} \bigg( \int_{t}^{z} a \bigg)^q \,w(z)\,dz\bigg)^{\frac{q}{p-q}} \bigg( \int_{t}^{x+t} a \bigg)^q \,w(x + t)\,dx\bigg)^\frac{p-q}{pq} \\
	&\hspace{-16.8cm} = \bigg(\int_{0}^{\infty} \bigg( \int_{[x,\infty)} \, d \, \bigg( - \bigg( \sup_{s + t \le \tau} {\mathcal U}(\tau)^{p'} \bigg( \int_t^{\tau} {\mathcal V} (y) \,dy \bigg) \bigg) \bigg)\bigg)^\frac{q(p-1)}{p-q} \, \bigg( \int_t^{x + t} \bigg( \int_{t}^{z} a \bigg)^q \,w(z)\,dz\bigg)^{\frac{q}{p-q}} \bigg( \int_{t}^{x+t} a \bigg)^q \,w(x + t)\,dx\bigg)^\frac{p-q}{pq} \\
	&\hspace{-16.8cm} = \bigg(\int_{0}^{\infty} \bigg( \int_{[x + t,\infty)} \, d \, \bigg( -\bigg( \sup_{s \le \tau} {\mathcal U}(\tau)^{p'} \bigg( \int_t^{\tau} {\mathcal V} (y) \,dy \bigg) \bigg) \bigg)\bigg)^\frac{q(p-1)}{p-q} \, \bigg( \int_t^{x + t} \bigg( \int_{t}^{z} a \bigg)^q \,w(z)\,dz\bigg)^{\frac{q}{p-q}} \bigg( \int_{t}^{x+t} a \bigg)^q \,w(x + t)\,dx\bigg)^\frac{p-q}{pq} \\
	&\hspace{-16.8cm} = \bigg(\int_{t}^{\infty} \bigg( \int_{[x,\infty)} \, d \, \bigg( -\bigg( \sup_{s \le \tau} {\mathcal U}(\tau)^{p'} \bigg( \int_t^{\tau} {\mathcal V} (y) \,dy \bigg) \bigg) \bigg)\bigg)^\frac{q(p-1)}{p-q} \, \bigg( \int_t^{x} \bigg( \int_{t}^{z} a \bigg)^q \,w(z)\,dz\bigg)^{\frac{q}{p-q}} \bigg( \int_{t}^{x} a \bigg)^q \,w(x)\,dx\bigg)^\frac{p-q}{pq},
	\end{align*}
	
	\begin{align*}
	\bigg(\int_{0}^{\infty} \bigg( \int_{(0,x]} A_t(s)^{p'}\, d \, \bigg( -\bigg( \sup_{s \le \tau} u_t(\tau)^{p'}\,\Psi_t(\tau)^{2p'} \bigg( \int_0^{\tau}\Psi_t(y)^{-p'} \psi_t(y) \,dy \bigg) \bigg) \bigg)\bigg)^\frac{q(p-1)}{p-q} \, \bigg( \int_x^{\infty} w_t(z)\,dz\bigg)^{\frac{q}{p-q}} w_t(x)\,dx\bigg)^\frac{p-q}{pq} & \\
	&\hspace{-15.5cm} = \bigg(\int_{0}^{\infty} \bigg( \int_{(0,x]} \bigg( \int_{t}^{s+t} a \bigg)^{p'}\, d \, \bigg( -\bigg( \sup_{s \le \tau} {\mathcal U}(\tau + t)^{p'} \bigg( \int_0^{\tau}  {\mathcal V}(y + t) \,dy \bigg) \bigg) \bigg)\bigg)^\frac{q(p-1)}{p-q} \, \bigg( \int_x^{\infty} w(z + t)\,dz\bigg)^{\frac{q}{p-q}} w(x + t)\,dx\bigg)^\frac{p-q}{pq} \\
	&\hspace{-15.5cm} = \bigg(\int_{0}^{\infty} \bigg( \int_{(0,x]} \bigg( \int_{t}^{s+t} a \bigg)^{p'}\, d \, \bigg( -\bigg( \sup_{s + t \le \tau} {\mathcal U}(\tau)^{p'} \bigg( \int_t^{\tau} {\mathcal V} (y) \,dy \bigg) \bigg) \bigg)\bigg)^\frac{q(p-1)}{p-q} \, \bigg( \int_{x + t}^{\infty} w(z)\,dz\bigg)^{\frac{q}{p-q}} w(x + t)\,dx\bigg)^\frac{p-q}{pq} \\
	&\hspace{-15.5cm} = \bigg(\int_{0}^{\infty} \bigg( \int_{(t,x + t]} \bigg( \int_{t}^{s} a \bigg)^{p'}\, d \, \bigg( -\bigg( \sup_{s \le \tau} {\mathcal U}(\tau)^{p'} \bigg( \int_t^{\tau} {\mathcal V} (y) \,dy \bigg) \bigg) \bigg)\bigg)^\frac{q(p-1)}{p-q} \, \bigg( \int_{x + t}^{\infty} w(z)\,dz\bigg)^{\frac{q}{p-q}} w(x + t)\,dx\bigg)^\frac{p-q}{pq} \\
	&\hspace{-15.5cm} = \bigg(\int_{t}^{\infty} \bigg( \int_{(t,x]} \bigg( \int_{t}^{s} a \bigg)^{p'}\, d \, \bigg( -\bigg( \sup_{s \le \tau} {\mathcal U}(\tau)^{p'} \bigg( \int_t^{\tau} {\mathcal V} (y) \,dy \bigg) \bigg) \bigg)\bigg)^\frac{q(p-1)}{p-q} \, \bigg( \int_{x}^{\infty} w(z)\,dz\bigg)^{\frac{q}{p-q}} w(x)\,dx\bigg)^\frac{p-q}{pq},
	\end{align*}
	
	combining, we arrive at
	\begin{align*}
	\sup_{f \in \mp^+ (t,\infty)} \frac{\bigg( \int_t^{\infty} \bigg( \int_t^y \bigg( \sup_{s \le \tau}u(\tau)  \bigg(\int_{\tau}^{\infty} f(z)\,dz \bigg) \bigg) a(s)\,ds \bigg)^q w(y)\,dy \bigg)^{\frac{1}{q}}}{\bigg( \int_t^{\infty} f(s)^pv(s)\,ds\bigg)^{\frac{1}{p}}} & \\
	&\hspace{-9cm} \approx \, \bigg(\int_{t}^{\infty}\bigg(\int_{t}^{y} {\mathcal V}(x)\,dx\bigg)^\frac{p(q-1)}{p-q} {\mathcal V}(y) \bigg(\int_{y}^{\infty}\bigg(\int_{y}^{z}\bigg(\sup_{s \le \tau} {\mathcal U}(\tau) \bigg)a(s)\,ds\bigg)^{q}\,w(z)\,dz\bigg)^\frac{p}{p-q}\,dy\bigg)^\frac{p-q}{pq} \\
	&\hspace{-8.5cm} + \, \bigg(\int_{t}^{\infty}\bigg(\int_{t}^{y} {\mathcal V} (x) \bigg(\int_{x}^{y}\bigg(\sup_{s \le \tau} {\mathcal U} (\tau)\bigg)a(s)\,ds\bigg)^{p'}\,dx\bigg)^\frac{q(p-1)}{p-q}\bigg(\int_{y}^{\infty}\,w(z)\,dz\bigg)^\frac{q}{p-q}\,w(y)\,dy\bigg)^\frac{p-q}{pq} \\
	&\hspace{-8.5cm} + \bigg(\int_{t}^{\infty} \bigg( \int_{[x,\infty)} \, d \, \bigg( -\bigg( \sup_{s \le \tau} {\mathcal U}(\tau)^{p'} \bigg( \int_t^{\tau} {\mathcal V} (y) \,dy \bigg) \bigg) \bigg)\bigg)^\frac{q(p-1)}{p-q}\, \bigg( \int_t^{x} \bigg( \int_{t}^{z} a \bigg)^q \,w(z)\,dz\bigg)^{\frac{q}{p-q}} \bigg( \int_{t}^{x} a \bigg)^q \,w(x)\,dx\bigg)^\frac{p-q}{pq} \\
	&\hspace{-8.5cm} + \bigg(\int_{t}^{\infty} \bigg( \int_{(t,x]} \bigg( \int_{t}^{s} a \bigg)^{p'}\, d \, \bigg( -\bigg( \sup_{s \le \tau} {\mathcal U}(\tau)^{p'} \bigg( \int_t^{\tau} {\mathcal V} (y) \,dy \bigg) \bigg) \bigg)\bigg)^\frac{q(p-1)}{p-q}\, \bigg( \int_{x}^{\infty} w(z)\,dz\bigg)^{\frac{q}{p-q}} w(x)\,dx\bigg)^\frac{p-q}{pq} \\
	&\hspace{-8.5cm} + \bigg( \int_t^{\infty} \bigg(\int_{t}^{z} a\bigg) ^q \,w(z)\,dz\bigg)^{\frac{1}{q}}\lim_{s\rightarrow\infty}\bigg(\sup_{s \le\tau } {\mathcal U} (\tau) \bigg(\int_t^{\tau}{\mathcal V}(y) \,dy\bigg)^{\frac{1}{p'}}\bigg) \\
	&\hspace{-8.5cm} + \bigg( \int_t^{\infty} v(s)^{1 - p'}\,ds\bigg)^{- \frac{1}{p(p' + 1)}} \bigg(\int_t^{\infty}\bigg(\int_t^{x}\bigg(\sup_{s \le\tau } {\mathcal U} (\tau) \bigg)a(s)\, ds\bigg)^q \,w(x)\,dx\bigg)^{\frac{1}{q}}.
	\end{align*}
	The proof is completed.
	
\end{proof}


\section{Characterization of the restricted inequality}\label{MR}

We start this section with some historical remarks concerning restricted inequalities related to the operator $T_{u,b}$. 

Note that the inequality 
\begin{equation}\label{Tub.thm.1.eq.1}
\|T_{u,b}f \|_{q,w,(0,\infty)} \le C \| f \|_{p,v,(0,\infty)}, \qquad f \in
\mp^{+,\dn}(0,\infty)
\end{equation}
was characterized in \cite[Theorem 3.5]{gop} under condition
$$
\sup_{t \in (0,\infty)} \frac{u(t)}{B(t)} \int_0^t
\frac{b(\tau)}{u(\tau)}\,d\tau < \infty.
$$
However, the case when $0 < p \le 1 < q < \infty$ was not considered in \cite{gop}. It is also worth to mention that in the case when $1 < p < \infty$, $0 < q < p < \infty$, $q \neq 1$ \cite[Theorem 3.5]{gop} contains only discrete condition. In
\cite{gogpick2007} the new reduction theorem was obtained when $0 < p \le 1$, and this technique allowed to characterize inequality \eqref{Tub.thm.1.eq.1} when $b \equiv 1$, and in the case when $0 < q< p \le 1$, \cite{gogpick2007} contains only discrete condition. The complete characterizations of inequality  \eqref{Tub.thm.1.eq.1} for $0 < q \le \infty$, $0 < p \le \infty$ were given in \cite{GogMusISI} and \cite{musbil}. Using the results in  \cites{PS_Proc_2013,PS_Dokl_2013,PS_Dokl_2014,P_Dokl_2015}, another characterization of  \eqref{Tub.thm.1.eq.1}  was obtained  in  \cite{StepSham} and \cite{Sham}. 

Now we present characterization of inequality \eqref{main.ineq.}.

Denote the best constant in inequality \eqref{main.ineq.} by $K$, that is, 
$$
K : = \sup_{f \in \mp^+} \frac{\bigg( \int_0^{\infty} \bigg( \int_0^x \big[ T_{u,b}f^* (t)\big]^r\,dt\bigg)^{\frac{q}{r}} w(x)\,dx\bigg)^{\frac{1}{q}}}
{\bigg( \int_0^{\infty} \bigg( \int_0^x [f^* (\tau)]^p\,d\tau \bigg)^{\frac{m}{p}} v(x)\,dx \bigg)^{\frac{1}{m}}} \equiv \sup_{f \in \mp^+} \frac{\bigg( \int_0^{\infty} \bigg( \int_0^x \big[ T_{u,b}f^* (t)\big]^r\,dt\bigg)^{\frac{q}{r}} w(x)\,dx\bigg)^{\frac{1}{q}}}
{\|f\|_{\GG(p,m,v)}}.
$$

The following two reduction lemmas hold true.
\begin{lemma}\label{R1}
	Let $1 < r < q < \infty$, $0 < p < \infty$, $0 < m < \infty$ and $b \in \W(0,\infty)$ be such that the function 
	$B(t)$ satisfies  $0 < B(t) < \infty$ for every $t \in (0,\infty)$. Assume that $u \in \W(0,\infty) \cap C(0,\infty)$ and $v,\,w \in \W(0,\infty)$. Then
	\begin{align*}
	K = \sup_{g \in \mp^+} \frac{1}
	{\|g\|_{\frac{q}{q-r},w^{\frac{r}{r-q}},(0,\infty)}^{\frac{1}{r}}} \sup_{h:\, \int_0^x h \le \int_0^x \big( \int_{\tau}^{\infty} g\big)\,d\tau} \sup_{\vp \in {\mathfrak M}^+} \frac{1}{\|\vp\|_{r',h^{1-r'},(0,\infty)}}
	\sup_{f \in \mp^+} \frac{\int_0^{\infty} f^* (y) b(y) \int_y^{\infty} \vp(x) \frac{u(x)}{B(x)} \,dx \,dy}{\|f\|_{\GG(p,m,v)}}.
	\end{align*}	
\end{lemma}	

\begin{proof}
	By duality, on using Fubini's Theorem, we have that
	\begin{align*}
	K & = \sup_{f \in \mp^+} \frac{1}{\|f\|_{\GG(p,m,v)}} \left\{\sup_{g \in \mp^+} \frac{ \int_0^{\infty} \bigg( \int_0^x \big[ T_{u,b}f^* (t)\big]^r\,dt\bigg) g(x) \,dx}
	{\|g\|_{\frac{q}{q-r},w^{\frac{r}{r-q}},(0,\infty)}^{\frac{1}{r}}}\right\}^{\frac{1}{r}} \\
	& = \sup_{f \in \mp^+} \frac{1}{\|f\|_{\GG(p,m,v)}} \left\{\sup_{g \in \mp^+} \frac{ \int_0^{\infty} \bigg(\int_t^{\infty} g(x) \,dx \bigg) \bigg[\sup_{\tau \ge t} \bigg( \frac{u(\tau)}{B(\tau)} \int_0^{\tau} f^* (y) b(y)\,dy \bigg)^r\bigg]\,dt }
	{\|g\|_{\frac{q}{q-r},w^{\frac{r}{r-q}},(0,\infty)}^{\frac{1}{r}}}\right\}^{\frac{1}{r}}.
	\end{align*}
	
	On using Theorem \ref{transfermon}, we arrive at
	\begin{align*}
	K = \sup_{f \in \mp^+} \frac{1}{\|f\|_{\GG(p,m,v)}} \sup_{g \in \mp^+} \frac{ 1 }
	{\|g\|_{\frac{q}{q-r},w^{\frac{r}{r-q}},(0,\infty)}^{\frac{1}{r}}} \sup_{h:\, \int_0^x h \le \int_0^x \big( \int_{\tau}^{\infty} g\big)\,d\tau} \left\{\int_0^{\infty} h(x) \bigg( \frac{u(x)}{B(x)} \int_0^x f^* (y) b(y)\,dy \bigg)^r\,dx \right\}^{\frac{1}{r}}.
	\end{align*} 
	
	By duality, we obtain that
	\begin{align*}
	K = \sup_{f \in \mp^+} \frac{1}{\|f\|_{\GG(p,m,v)}} \sup_{g \in \mp^+} \frac{ 1 }
	{\|g\|_{\frac{q}{q-r},w^{\frac{r}{r-q}},(0,\infty)}^{\frac{1}{r}}} \sup_{h:\, \int_0^x h \le \int_0^x \big( \int_{\tau}^{\infty} g\big)\,d\tau} \sup_{\vp \in {\mathfrak M}^+}
	\frac{\int_0^{\infty} \vp(x) \frac{u(x)}{B(x)} \int_0^x f^* (y) b(y)\,dy \,dx}{\|\vp\|_{r',h^{1-r'},(0,\infty)}}.
	\end{align*} 
	
	By Fubini's Theorem, interchanging the suprema, we arrive at
	\begin{align*}
	K & = \sup_{f \in \mp^+} \frac{1}{\|f\|_{\GG(p,m,v)}} \sup_{g \in \mp^+} \frac{ 1 }
	{\|g\|_{\frac{q}{q-r},w^{\frac{r}{r-q}},(0,\infty)}^{\frac{1}{r}}} \sup_{h:\, \int_0^x h \le \int_0^x \big( \int_{\tau}^{\infty} g\big)\,d\tau} \sup_{\vp \in {\mathfrak M}^+}
	\frac{\int_0^{\infty} f^* (y) b(y) \int_y^{\infty} \vp(x) \frac{u(x)}{B(x)} \,dx \,dy}{\|\vp\|_{r',h^{1-r'},(0,\infty)}} \\
	& = \sup_{g \in \mp^+} \frac{ 1 }
	{\|g\|_{\frac{q}{q-r},w^{\frac{r}{r-q}},(0,\infty)}^{\frac{1}{r}}} \sup_{h:\, \int_0^x h \le \int_0^x \big( \int_{\tau}^{\infty} g\big)\,d\tau} \sup_{\vp \in {\mathfrak M}^+} \frac{1}{\|\vp\|_{r',h^{1-r'},(0,\infty)}}
	\sup_{f \in \mp^+} \frac{\int_0^{\infty} f^* (y) b(y) \int_y^{\infty} \vp(x) \frac{u(x)}{B(x)} \,dx \,dy}{\|f\|_{\GG(p,m,v)}}.
	\end{align*} 
	This completes the proof.
\end{proof}

\begin{lemma}\label{R2}
	Let $1 < r < q < \infty$, $1 < p < \infty$, $1 < m < \infty$ and $b \in \W (0,\infty) \cap \mp^+ ((0,\infty);\dn)$ be such that the function $B(t)$ satisfies  $0 < B(t) < \infty$ for every $t \in (0,\infty)$. Suppose that $u \in \W(0,\infty) \cap C(0,\infty)$, $v \in \W_{m,p}(0,\infty)$ and $w \in \W(0,\infty)$. Denote by
	\begin{equation}\label{defof_v2}
	v_2(t) : = \frac{t^{\frac{m'}{p'}}v_0(t)}{v_1(t)^{m' + 1}}, \qquad t \in (0,\infty),
	\end{equation} 
	where $v_0$ and $v_1$ are defined by \eqref{defof_v} and \eqref{defof_u}, respectively.
	Then
	\begin{align*}
	K \approx A + B,
	\end{align*}	
	where
	\begin{align*}
	A: = & \, \sup_{g \in \mp^+} \frac{ 1 }
	{\|g\|_{\frac{q}{q-r},w^{\frac{r}{r-q}},(0,\infty)}^{\frac{1}{r}}} \sup_{h:\, \int_0^x h \le \int_0^x \big( \int_{\tau}^{\infty} g\big)\,d\tau} \, \sup_{\vp \in {\mathfrak M}^+} \frac{1}{\|\vp\|_{r',\big(\frac{B}{u}\big)^{r'}h^{1-r'},(0,\infty)}}
	\bigg( \int_0^{\infty} \bigg( \int_t^{\infty} \bigg( \int_s^{\infty} \vp  \bigg)^{p'} \bigg( \frac{B(s)}{s} \bigg)^{p'}\,ds \bigg)^{\frac{m'}{p'}} v_2(t)\,dt \bigg)^{\frac{1}{m'}} \\
	\intertext{and}
	B: = & \sup_{g \in \mp^+} \frac{ 1 }
	{\|g\|_{\frac{q}{q-r},w^{\frac{r}{r-q}},(0,\infty)}^{\frac{1}{r}}} \sup_{h:\, \int_0^x h \le \int_0^x \big( \int_{\tau}^{\infty} g\big)\,d\tau} \, \sup_{\vp \in {\mathfrak M}^+} \frac{1}{\|\vp\|_{r',u^{-r'}h^{1-r'},(0,\infty)}}
	\bigg( \int_0^{\infty} \bigg( \int_t^{\infty} \bigg( \int_0^s \vp \bigg)^{p'}\,\frac{ds}{s^{p'}} \bigg)^{\frac{m'}{p'}} v_2(t)\,dt \bigg)^{\frac{1}{m'}}.
	\end{align*} 
\end{lemma}	

\begin{proof}
	By Lemma \ref{R1} and Theorem \ref{assosGG}, we have that
	\begin{align*}
	K & \approx \sup_{g \in \mp^+} \frac{ 1 }
	{\|g\|_{\frac{q}{q-r},w^{\frac{r}{r-q}},(0,\infty)}^{\frac{1}{r}}} \sup_{h: \int_0^x h \le \int_0^x \big( \int_{\tau}^{\infty} g\big)d\tau} \sup_{\vp \in {\mathfrak M}^+} \frac{1}{\|\vp\|_{r',h^{1-r'},(0,\infty)}}
	\bigg( \int_0^{\infty} \bigg( \int_t^{\infty} \bigg( \frac{1}{s} \int_0^s b(y) \bigg( \int_y^{\infty} \vp \frac{u}{B} \bigg) \, dy  \bigg)^{p'} \,ds \bigg)^{\frac{m'}{p'}} v_2(t) \, dt \bigg)^{\frac{1}{m'}}.
	\end{align*} 
	Since
\begin{align*}
	\frac{1}{s} \int_0^s b(y) \bigg( \int_y^{\infty} \vp(x) \frac{u(x)}{B(x)} \,dx \bigg) \,dy & = \frac{1}{s} \int_0^s b(y) \bigg( \int_y^s \vp(x) \frac{u(x)}{B(x)} \,dx \bigg) \,dy + \frac{1}{s} \int_0^s b(y) \,dy \, \int_s^{\infty} \vp(x) \frac{u(x)}{B(x)} \,dx  \\
	&  = \frac{1}{s} \int_0^s \vp(x) u(x) \,dx  + \frac{B(s)}{s} \int_s^{\infty} \vp(x) \frac{u(x)}{B(x)} \,dx,
\end{align*} 
we arrive at
\begin{align*}
	K \approx & \, \sup_{g \in \mp^+} \frac{ 1 }
	{\|g\|_{\frac{q}{q-r},w^{\frac{r}{r-q}},(0,\infty)}^{\frac{1}{r}}} \sup_{h:\, \int_0^x h \le \int_0^x \big( \int_{\tau}^{\infty} g\big)\,d\tau} \\
	& \qquad  \sup_{\vp \in {\mathfrak M}^+} \frac{1}{\|\vp\|_{r',h^{1-r'},(0,\infty)}}
	\bigg( \int_0^{\infty} \bigg( \int_t^{\infty} \bigg( \frac{B(s)}{s}  \int_s^{\infty} \vp(x) \frac{u(x)}{B(x)} \,dx  \bigg)^{p'}\,ds \bigg)^{\frac{m'}{p'}} v_2(t)\,dt \bigg)^{\frac{1}{m'}} \\
	& \,\, + \sup_{g \in \mp^+} \frac{ 1 }
	{\|g\|_{\frac{q}{q-r},w^{\frac{r}{r-q}},(0,\infty)}^{\frac{1}{r}}} \sup_{h:\, \int_0^x h \le \int_0^x \big( \int_{\tau}^{\infty} g\big)\,d\tau} \\
	& \qquad  \sup_{\vp \in {\mathfrak M}^+} \frac{1}{\|\vp\|_{r',h^{1-r'},(0,\infty)}}
	\bigg( \int_0^{\infty} \bigg( \int_t^{\infty} \bigg( \frac{1}{s}  \int_0^s \vp(x) u(x) \,dx  \bigg)^{p'}\,ds \bigg)^{\frac{m'}{p'}} v_2(t)\,dt \bigg)^{\frac{1}{m'}}.
\end{align*} 
Observe that the inequality
$$
\bigg( \int_0^{\infty} \bigg( \int_t^{\infty} \bigg( \frac{B(s)}{s}  \int_s^{\infty} \vp(x) \frac{u(x)}{B(x)} \,dx  \bigg)^{p'}\,ds \bigg)^{\frac{m'}{p'}} v_2(t)\,dt \bigg)^{\frac{1}{m'}} \le c \, \bigg( \int_0^{\infty} \vp(x)^{r'} h(x)^{1-r'}\,dx \bigg)^{\frac{1}{r'}}
$$
holds true for all $h \in \mp^+$ if and only if the inequality
$$
\bigg( \int_0^{\infty} \bigg( \int_t^{\infty} \bigg( \int_s^{\infty} \vp(x) \,dx  \bigg)^{p'} \bigg( \frac{B(s)}{s} \bigg)^{p'}\,ds \bigg)^{\frac{m'}{p'}} v_2(t)\,dt \bigg)^{\frac{1}{m'}} \le c \, \bigg( \int_0^{\infty} \vp(x)^{r'} \bigg( \frac{B(x)}{u(x)} \bigg)^{r'}h(x)^{1-r'}\,dx \bigg)^{\frac{1}{r'}}
$$
holds for all $h \in \mp^+$.

Similarly, the inequality
$$
\bigg( \int_0^{\infty} \bigg( \int_t^{\infty} \bigg( \frac{1}{s}  \int_0^s \vp(x) u(x) \,dx  \bigg)^{p'}\,ds \bigg)^{\frac{m'}{p'}} v_2(t)\,dt \bigg)^{\frac{1}{m'}}
\le c \, \bigg( \int_0^{\infty} \vp(x)^{r'} h(x)^{1-r'}\,dx \bigg)^{\frac{1}{r'}}
$$
holds true for all $h \in \mp^+ $ if and only if the inequality
$$
\bigg( \int_0^{\infty} \bigg( \int_t^{\infty} \bigg( \int_0^s \vp(x)\,dx  \bigg)^{p'}\,\frac{ds}{s^{p'}} \bigg)^{\frac{m'}{p'}} v_2(t)\,dt \bigg)^{\frac{1}{m'}} \le c \, \bigg( \int_0^{\infty} \vp(x)^{r'} \bigg( \frac{1}{u(x)} \bigg)^{r'}h(x)^{1-r'}\,dx \bigg)^{\frac{1}{r'}}
$$
holds for all $h \in \mp^+$.

Thus,
\begin{align*}
	K \approx & \, \sup_{g \in \mp^+} \frac{ 1 }
	{\|g\|_{\frac{q}{q-r},w^{\frac{r}{r-q}},(0,\infty)}^{\frac{1}{r}}} \sup_{h:\, \int_0^x h \le \int_0^x \big( \int_{\tau}^{\infty} g\big)\,d\tau} \\
	& \qquad  \sup_{\vp \in {\mathfrak M}^+} \frac{1}{\|\vp\|_{r',\big(\frac{B}{u}\big)^{r'}h^{1-r'},(0,\infty)}}
	\bigg( \int_0^{\infty} \bigg( \int_t^{\infty} \bigg( \int_s^{\infty} \vp(x) \,dx  \bigg)^{p'} \bigg( \frac{B(s)}{s} \bigg)^{p'}\,ds \bigg)^{\frac{m'}{p'}} v_2(t)\,dt \bigg)^{\frac{1}{m'}} \\
	& \,\, + \sup_{g \in \mp^+} \frac{ 1 }
	{\|g\|_{\frac{q}{q-r},w^{\frac{r}{r-q}},(0,\infty)}^{\frac{1}{r}}} \sup_{h:\, \int_0^x h \le \int_0^x \big( \int_{\tau}^{\infty} g\big)\,d\tau} \\
	& \qquad  \sup_{\vp \in {\mathfrak M}^+} \frac{1}{\|\vp\|_{r',u^{-r'}h^{1-r'},(0,\infty)}}
	\bigg( \int_0^{\infty} \bigg( \int_t^{\infty} \bigg( \int_0^s \vp(x)\,dx  \bigg)^{p'}\,\frac{ds}{s^{p'}} \bigg)^{\frac{m'}{p'}} v_2(t)\,dt \bigg)^{\frac{1}{m'}} = A + B
\end{align*} 
holds.	
\end{proof}	

\begin{theorem}\label{thm.prev}
Let $1 < m < p \le r < q < \infty$ and $b \in \W (0,\infty) \cap \mp^+ ((0,\infty);\dn)$ be such that the function $B(t)$ satisfies  $0 < B(t) < \infty$ for every $t \in (0,\infty)$. Suppose that $u \in \W(0,\infty) \cap C(0,\infty)$, $v \in \W_{m.p}(0,\infty)$ and $w \in \W(0,\infty)$. Suppose that
$$
0 < \int_0^t \bigg( \int_s^t \bigg( \frac{B(y)}{y} \bigg)^{p'}\,dy \bigg)^{\frac{m'}{p'}} v_2(s)\,ds < \infty, \qquad t \in (0,\infty),
$$
where $v_2$ is defined by \eqref{defof_v2}.
Then
\begin{align*}
	K 	& \approx  \, \sup_{t \in (0,\infty)} \bigg( \int_0^t \bigg( \int_s^t \bigg( \frac{B(y)}{y} \bigg)^{p'}\,dy \bigg)^{\frac{m'}{p'}}
	v_2(s)\,ds \bigg)^{\frac{1}{m'}} \bigg( \sup_{t \le \tau} \frac{u(\tau)}{B(\tau)} \bigg) \bigg( \int_0^{\infty} \bigg(\frac{yt}{y+t}\bigg)^{\frac{q}{r}}w(y)\,dy \bigg)^{\frac{1}{q}} \\
	& +  \, \sup_{t \in (0,\infty)} \bigg( \int_0^t \bigg( \int_s^t \bigg( \frac{B(y)}{y} \bigg)^{p'}\,dy \bigg)^{\frac{m'}{p'}}
	v_2(s)\,ds \bigg)^{\frac{1}{m'}} \bigg(\int_t^{\infty} \bigg( \int_t^y \bigg( \sup_{x \le \tau}\bigg( \frac{u(\tau)}{B(\tau)} \bigg)^r  \bigg) \,dx \bigg)^{\frac{q}{r}}w(y)\,dy\bigg)^{\frac{1}{q}}\\
	& + \sup_{t \in (0,\infty)}  \bigg( \int_0^{\infty} \bigg( \frac{1}{s + t}\bigg)^{\frac{m'}{p}}v_2(s)\,ds \bigg)^{\frac{1}{m'}} \bigg(\int_0^t \bigg( \int_0^y  \bigg( \sup_{x \le \tau \le t} u(\tau)^r  \bigg) \,dx \bigg)^{\frac{q}{r}} w(y)\,dy\bigg)^{\frac{1}{q}} \\
	& + \sup_{t \in (0,\infty)}  \bigg( \int_0^{\infty} \bigg( \frac{1}{s + t}\bigg)^{\frac{m'}{p}}v_2(s)\,ds \bigg)^{\frac{1}{m'}} \bigg( \int_0^t  \bigg( \sup_{x \le \tau \le t} u(\tau)^r  \bigg) \,dx \bigg)^{\frac{1}{r}}\bigg(\int_t^{\infty}  w(y)\,dy\bigg)^{\frac{1}{q}}. 
\end{align*}
	
\end{theorem}

\begin{proof}
	By Lemma \ref{R2} we have that
	$$
	K \approx A + B.
	$$
We estimate $A$. By Theorem \ref{krepick}, (a), we have that
\begin{align*}
\sup_{\vp \in {\mathfrak M}^+} \frac{1}{\|\vp\|_{r',h^{1-r'},(0,\infty)}}
\bigg( \int_0^{\infty} \bigg( \int_t^{\infty} \bigg( \int_s^{\infty} \vp(x) \,dx  \bigg)^{p'} \bigg( \frac{B(s)}{s} \bigg)^{p'}\,ds \bigg)^{\frac{m'}{p'}} v_2(t)\,dt \bigg)^{\frac{1}{m'}} & \\
& \hspace{-7cm}\approx \sup_{t \in (0,\infty)}\bigg( \int_t^{\infty} h(x) \bigg( \frac{u(x)}{B(x)} \bigg)^r \,dx \bigg)^{\frac{1}{r}} \bigg( \int_0^t \bigg( \int_s^t \bigg( \frac{B(y)}{y} \bigg)^{p'}\,dy \bigg)^{\frac{m'}{p'}} v_2(s)\,ds \bigg)^{\frac{1}{m'}}.
\end{align*}

Hence
\begin{align*}
	A \approx & \, \sup_{g \in \mp^+} \frac{ 1 }
	{\|g\|_{\frac{q}{q-r},w^{\frac{r}{r-q}},(0,\infty)}^{\frac{1}{r}}} \sup_{h:\, \int_0^x h \le \int_0^x \big( \int_{\tau}^{\infty} g\big)\,d\tau} 
	\sup_{t \in (0,\infty)}\bigg( \int_t^{\infty} h(x) \bigg( \frac{u(x)}{B(x)} \bigg)^r \,dx \bigg)^{\frac{1}{r}} \bigg( \int_0^t \bigg( \int_s^t \bigg( \frac{B(y)}{y} \bigg)^{p'}\,dy \bigg)^{\frac{m'}{p'}}v_2(s)\,ds \bigg)^{\frac{1}{m'}}.
\end{align*}	

Interchanging the suprema, we get that
\begin{align*}
	A \approx & \, \sup_{g \in \mp^+} \frac{ 1 }
	{\|g\|_{\frac{q}{q-r},w^{\frac{r}{r-q}},(0,\infty)}^{\frac{1}{r}}} \sup_{t \in (0,\infty)} \bigg( \int_0^t \bigg( \int_s^t \bigg( \frac{B(y)}{y} \bigg)^{p'}\,dy \bigg)^{\frac{m'}{p'}} v_2(s)\,ds \bigg)^{\frac{1}{m'}} \left\{\sup_{h:\, \int_0^x h \le \int_0^x \big( \int_{\tau}^{\infty} g\big)\,d\tau} \int_t^{\infty} h(x) \bigg( \frac{u(x)}{B(x)} \bigg)^r \,dx\right\}^{\frac{1}{r}} \\
	\approx & \, \sup_{g \in \mp^+} \frac{ 1 } {\|g\|_{\frac{q}{q-r},w^{\frac{r}{r-q}},(0,\infty)}^{\frac{1}{r}}} \sup_{t \in (0,\infty)} \bigg( \int_0^t \bigg( \int_s^t \bigg( \frac{B(y)}{y} \bigg)^{p'}\,dy \bigg)^{\frac{m'}{p'}} v_2(s)\,ds \bigg)^{\frac{1}{m'}} \left\{ \sup_{h:\, \int_0^x h \le \int_0^x \big( \int_{\tau}^{\infty} g\big)\,d\tau} \int_0^{\infty} h(x) \bigg( \frac{u(x)}{B(x)} \bigg)^r \chi_{(t,\infty)}(x) \,dx\right\}^{\frac{1}{r}}.
\end{align*}

By Theorem \ref{transfermon}, we have that
\begin{align*}
	A \approx & \, \sup_{g \in \mp^+} \frac{ 1 } {\|g\|_{\frac{q}{q-r},w^{\frac{r}{r-q}},(0,\infty)}^{\frac{1}{r}}} \sup_{t \in (0,\infty)} \bigg( \int_0^t \bigg( \int_s^t \bigg( \frac{B(y)}{y} \bigg)^{p'}\,dy \bigg)^{\frac{m'}{p'}} v_2(s)\,ds \bigg)^{\frac{1}{m'}} \bigg(\int_0^{\infty} \bigg( \int_x^{\infty} g\bigg) \bigg( \sup_{x \le \tau}\bigg( \frac{u(\tau)}{B(\tau)} \bigg)^r \chi_{(t,\infty)}(\tau) \bigg) \,dx \bigg)^{\frac{1}{r}}.
\end{align*}

Applying Fubini's Theorem, interchanging the suprema, by duality, we get that
\begin{align*}
A \approx & \, \sup_{t \in (0,\infty)} \bigg( \int_0^t \bigg( \int_s^t \bigg( \frac{B(y)}{y} \bigg)^{p'}\,dy \bigg)^{\frac{m'}{p'}}
v_2(s)\,ds \bigg)^{\frac{1}{m'}} \left\{\sup_{g \in \mp^+} \frac{ \int_0^{\infty} g(y)\, \int_0^y \bigg( \sup_{x \le \tau}\bigg( \frac{u(\tau)}{B(\tau)} \bigg)^r \chi_{(t,\infty)}(\tau) \bigg) \,dx \,dy } {\|g\|_{\frac{q}{q-r},w^{\frac{r}{r-q}},(0,\infty)}}  \right\}^{\frac{1}{r}} \\
= & \, \sup_{t \in (0,\infty)} \bigg( \int_0^t \bigg( \int_s^t \bigg( \frac{B(y)}{y} \bigg)^{p'}\,dy \bigg)^{\frac{m'}{p'}}
v_2(s)\,ds \bigg)^{\frac{1}{m'}} \bigg(\int_0^{\infty} \bigg( \int_0^y \bigg( \sup_{x \le \tau}\bigg( \frac{u(\tau)}{B(\tau)} \bigg)^r \chi_{(t,\infty)}(\tau) \bigg) \,dx \bigg)^{\frac{q}{r}}w(y)\,dy \bigg)^{\frac{1}{q}}.
\end{align*}

Since 
\begin{align*}
\sup_{x \le \tau} \chi_{(t,\infty)}(\tau) \bigg( \frac{u(\tau)}{B(\tau)} \bigg)^r = \sup_{t \le \tau} \bigg( \frac{u(\tau)}{B(\tau)} \bigg)^r,
\end{align*}
when $0 < x \le t$, we get that

\begin{align*}
	A \approx & \, \sup_{t \in (0,\infty)} \bigg( \int_0^t \bigg( \int_s^t \bigg( \frac{B(y)}{y} \bigg)^{p'}\,dy \bigg)^{\frac{m'}{p'}}
	v_2(s)\,ds \bigg)^{\frac{1}{m'}} \bigg(\int_0^t \bigg( \int_0^y \bigg( \sup_{x \le \tau}\bigg( \frac{u(\tau)}{B(\tau)} \bigg)^r \chi_{(t,\infty)}(\tau) \bigg) \,dx \bigg)^{\frac{q}{r}}w(y)\,dy\bigg)^{\frac{1}{q}} \\
	& +  \, \sup_{t \in (0,\infty)} \bigg( \int_0^t \bigg( \int_s^t \bigg( \frac{B(y)}{y} \bigg)^{p'}\,dy \bigg)^{\frac{m'}{p'}}
	v_2(s)\,ds \bigg)^{\frac{1}{m'}} \bigg(\int_t^{\infty} \bigg( \int_0^y \bigg( \sup_{x \le \tau}\bigg( \frac{u(\tau)}{B(\tau)} \bigg)^r \chi_{(t,\infty)}(\tau) \bigg) \,dx \bigg)^{\frac{q}{r}}w(y)\,dy\bigg)^{\frac{1}{q}} \\
	&\approx  \, \sup_{t \in (0,\infty)} \bigg( \int_0^t \bigg( \int_s^t \bigg( \frac{B(y)}{y} \bigg)^{p'}\,dy \bigg)^{\frac{m'}{p'}}
	v_2(s)\,ds \bigg)^{\frac{1}{m'}} \bigg( \sup_{t \le \tau} \frac{u(\tau)}{B(\tau)} \bigg) \bigg( \int_0^t y^{\frac{q}{r}}w(y)\,dy \bigg)^{\frac{1}{q}} \\
	& +  \, \sup_{t \in (0,\infty)} \bigg( \int_0^t \bigg( \int_s^t \bigg( \frac{B(y)}{y} \bigg)^{p'}\,dy \bigg)^{\frac{m'}{p'}}
	v_2(s)\,ds \bigg)^{\frac{1}{m'}} t^{\frac{1}{r}} \bigg( \sup_{t \le \tau} \frac{u(\tau)}{B(\tau)} \bigg) \bigg(\int_t^{\infty} w(y)\,dy\bigg)^{\frac{1}{q}}\\
	& +  \, \sup_{t \in (0,\infty)} \bigg( \int_0^t \bigg( \int_s^t \bigg( \frac{B(y)}{y} \bigg)^{p'}\,dy \bigg)^{\frac{m'}{p'}}
	v_2(s)\,ds \bigg)^{\frac{1}{m'}} \bigg(\int_t^{\infty} \bigg( \int_t^y \bigg( \sup_{x \le \tau}\bigg( \frac{u(\tau)}{B(\tau)} \bigg)^r  \bigg) \,dx \bigg)^{\frac{q}{r}}w(y)\,dy\bigg)^{\frac{1}{q}} \\
	& \approx  \, \sup_{t \in (0,\infty)} \bigg( \int_0^t \bigg( \int_s^t \bigg( \frac{B(y)}{y} \bigg)^{p'}\,dy \bigg)^{\frac{m'}{p'}}
	v_2(s)\,ds \bigg)^{\frac{1}{m'}} \bigg( \sup_{t \le \tau} \frac{u(\tau)}{B(\tau)} \bigg) \bigg( \int_0^{\infty} \bigg(\frac{yt}{y+t}\bigg)^{\frac{q}{r}}w(y)\,dy \bigg)^{\frac{1}{q}} \\
	& +  \, \sup_{t \in (0,\infty)} \bigg( \int_0^t \bigg( \int_s^t \bigg( \frac{B(y)}{y} \bigg)^{p'}\,dy \bigg)^{\frac{m'}{p'}}
	v_2(s)\,ds \bigg)^{\frac{1}{m'}} \bigg(\int_t^{\infty} \bigg( \int_t^y \bigg( \sup_{x \le \tau}\bigg( \frac{u(\tau)}{B(\tau)} \bigg)^r  \bigg) \,dx \bigg)^{\frac{q}{r}}w(y)\,dy\bigg)^{\frac{1}{q}}.
\end{align*}

	To estimate $B$, by Theorem \ref{gks}, (a), we have that
	\begin{align*}
	\sup_{\vp \in {\mathfrak M}^+} \frac{1}{\|\vp\|_{r',h^{1-r'},(0,\infty)}}
	\bigg( \int_0^{\infty} \bigg( \int_t^{\infty} \bigg( \int_0^s \vp(x)\,dx  \bigg)^{p'}\,\frac{ds}{s^{p'}} \bigg)^{\frac{m'}{p'}} v_2(t)\,dt \bigg)^{\frac{1}{m'}} & \\
	& \hspace{-7cm} \approx \sup_{t \in (0,\infty)}  \bigg( \int_0^t v_2(\tau)\,d\tau \bigg)^{\frac{1}{m'}} \bigg( \int_t^{\infty} \frac{ds}{s^{p'}} \bigg)^{\frac{1}{p'}} \bigg( \int_0^t h(x) u(x)^r \,dx \bigg)^{\frac{1}{r}}  \\
	& \hspace{-6.5cm} + \sup_{t \in (0,\infty)}  \bigg( \int_t^{\infty} \bigg( \int_s^{\infty} \frac{dy}{y^{p'}} \bigg)^{\frac{m'}{p'}} v_2(s) \,ds \bigg)^{\frac{1}{m'}}  \bigg( \int_0^t h(x) u(x)^r \,dx \bigg)^{\frac{1}{r}} \\
	& \hspace{-7cm} \approx \sup_{t \in (0,\infty)} t^{-\frac{1}{p}}  \bigg( \int_0^t v_2(s) \,ds \bigg)^{\frac{1}{m'}} \bigg( \int_0^t h(x) u(x)^r \,dx \bigg)^{\frac{1}{r}}  \\
	& \hspace{-6.5cm} + \sup_{t \in (0,\infty)}  \bigg( \int_t^{\infty} s^{ - \frac{m'}{p}} v_2(s)\,ds \bigg)^{\frac{1}{m'}}  \bigg( \int_0^t h(x) u(x)^r \,dx \bigg)^{\frac{1}{r}} \\
	& \hspace{-7cm} \approx \sup_{t \in (0,\infty)}  \bigg( \int_0^{\infty} \bigg( \frac{1}{s + t}\bigg)^{\frac{m'}{p}} v_2(s)\,ds \bigg)^{\frac{1}{m'}}  \bigg( \int_0^t h(x) u(x)^r \,dx \bigg)^{\frac{1}{r}}.
	\end{align*}
	
	Hence, interchanging the suprema, we get that
	\begin{align*}
	B & \ap \sup_{g \in \mp^+} \frac{ 1 }
	{\|g\|_{\frac{q}{q-r},w^{\frac{r}{r-q}}, (0,\infty)}^{\frac{1}{r}}} \sup_{h:\, \int_0^x h \le \int_0^x \big( \int_{\tau}^{\infty} g\big)\,d\tau}   \sup_{t \in (0,\infty)}  \bigg( \int_0^{\infty} \bigg( \frac{1}{s + t}\bigg)^{\frac{m'}{p}}v_2(s)\,ds \bigg)^{\frac{1}{m'}}  \bigg( \int_0^t h(x) u(x)^r \,dx \bigg)^{\frac{1}{r}}  \\
	& = \sup_{g \in \mp^+} \frac{ 1 }
	{\|g\|_{\frac{q}{q-r},w^{\frac{r}{r-q}}, (0,\infty)}^{\frac{1}{r}}} \sup_{t \in (0,\infty)}  \bigg( \int_0^{\infty} \bigg( \frac{1}{s + t}\bigg)^{\frac{m'}{p}}v_2(s)\,ds \bigg)^{\frac{1}{m'}} \left\{ \sup_{h:\, \int_0^x h \le \int_0^x \big( \int_{\tau}^{\infty} g\big)\,d\tau}      \int_0^{\infty} h(x) u(x)^r \chi_{(0,t)}(x) \,dx\right\} ^{\frac{1}{r}}.
	\end{align*}
	
	By Theorem \ref{transfermon}, we have that
	\begin{align*}
	B \approx \sup_{g \in \mp^+} \frac{ 1 }
	{\|g\|_{\frac{q}{q-r},w^{\frac{r}{r-q}}, (0,\infty)}^{\frac{1}{r}}} \sup_{t \in (0,\infty)}  \bigg( \int_0^{\infty} \bigg( \frac{1}{s + t}\bigg)^{\frac{m'}{p}}v_2(s)\,ds \bigg)^{\frac{1}{m'}} \bigg( \int_0^{\infty} \bigg( \int_x^{\infty} g\bigg) \bigg( \sup_{x \le \tau} u(\tau)^r \chi_{(0,t)}(\tau) \bigg) \,dx\bigg)^{\frac{1}{r}}.
	\end{align*}
	
	By Fubini's Theorem, interchanging the suprema, by duality we get that
	\begin{align*}
	B &\approx \sup_{t \in (0,\infty)}  \bigg( \int_0^{\infty} \bigg( \frac{1}{s + t}\bigg)^{\frac{m'}{p}}v_2(s)\,ds \bigg)^{\frac{1}{m'}} \left\{\sup_{g \in \mp^+} \frac{ \int_0^{\infty} g (y) \int_0^y  \bigg( \sup_{x \le \tau} u(\tau)^r \chi_{(0,t)}(\tau) \bigg) \,dx \,dy }
	{\|g\|_{\frac{q}{q-r},w^{\frac{r}{r-q}}, (0,\infty)}} \right\} ^{\frac{1}{r}} \\
	 & \approx \sup_{t \in (0,\infty)}  \bigg( \int_0^{\infty} \bigg( \frac{1}{s + t}\bigg)^{\frac{m'}{p}}v_2(s)\,ds \bigg)^{\frac{1}{m'}} \bigg(\int_0^{\infty} \bigg( \int_0^y  \bigg( \sup_{x \le \tau} u(\tau)^r \chi_{(0,t)}(\tau) \bigg) \,dx \bigg)^{\frac{q}{r}} w(y)\,dy\bigg)^{\frac{1}{q}}.
	\end{align*}
	
	Since 
		$$
		\sup_{x \le \tau} u(\tau)^r \chi_{(0,t)}(\tau) = \sup_{x \le \tau \le t} u(\tau)^r
		$$
		for $0 < x \le t$, we arrive at
		\begin{align*}
		B \approx &  \sup_{t \in (0,\infty)}  \bigg( \int_0^{\infty} \bigg( \frac{1}{s + t}\bigg)^{\frac{m'}{p}}v_2(s)\,ds \bigg)^{\frac{1}{m'}} \bigg(\int_0^t \bigg( \int_0^y  \bigg( \sup_{x \le \tau \le t} u(\tau)^r  \bigg) \,dx \bigg)^{\frac{q}{r}} w(y)\,dy \bigg)^{\frac{1}{q}} \\
		& + \sup_{t \in (0,\infty)}  \bigg( \int_0^{\infty} \bigg( \frac{1}{s + t}\bigg)^{\frac{m'}{p}}v_2(s)\,ds \bigg)^{\frac{1}{m'}} \bigg( \int_0^t  \bigg( \sup_{x \le \tau \le t} u(\tau)^r  \bigg) \,dx \bigg)^{\frac{1}{r}} \bigg(\int_t^{\infty}  w(y)\,dy\bigg)^{\frac{1}{q}} .
		\end{align*}	
		
	Combining the estimates for $A$ and $B$, we obtain that
	\begin{align*}
		K 	& \approx  \, \sup_{t \in (0,\infty)} \bigg( \int_0^t \bigg( \int_s^t \bigg( \frac{B(y)}{y} \bigg)^{p'}\,dy \bigg)^{\frac{m'}{p'}}
		v_2(s)\,ds \bigg)^{\frac{1}{m'}} \bigg( \sup_{t \le \tau} \frac{u(\tau)}{B(\tau)} \bigg) \bigg( \int_0^{\infty} \bigg(\frac{yt}{y+t}\bigg)^{\frac{q}{r}}w(y)\,dy \bigg)^{\frac{1}{q}} \\
		& +  \, \sup_{t \in (0,\infty)} \bigg( \int_0^t \bigg( \int_s^t \bigg( \frac{B(y)}{y} \bigg)^{p'}\,dy \bigg)^{\frac{m'}{p'}}
		v_2(s)\,ds \bigg)^{\frac{1}{m'}} \bigg(\int_t^{\infty} \bigg( \int_t^y \bigg( \sup_{x \le \tau}\bigg( \frac{u(\tau)}{B(\tau)} \bigg)^r  \bigg) \,dx \bigg)^{\frac{q}{r}}w(y)\,dy\bigg)^{\frac{1}{q}}\\
		& + \sup_{t \in (0,\infty)}  \bigg( \int_0^{\infty} \bigg( \frac{1}{s + t}\bigg)^{\frac{m'}{p}}v_2(s)\,ds \bigg)^{\frac{1}{m'}} \bigg(\int_0^t \bigg( \int_0^y  \bigg( \sup_{x \le \tau \le t} u(\tau)^r  \bigg) \,dx \bigg)^{\frac{q}{r}} w(y)\,dy\bigg)^{\frac{1}{q}} \\
		& + \sup_{t \in (0,\infty)}  \bigg( \int_0^{\infty} \bigg( \frac{1}{s + t}\bigg)^{\frac{m'}{p}}v_2(s)\,ds \bigg)^{\frac{1}{m'}} \bigg( \int_0^t  \bigg( \sup_{x \le \tau \le t} u(\tau)^r  \bigg) \,dx \bigg)^{\frac{1}{r}}\bigg(\int_t^{\infty}  w(y)\,dy\bigg)^{\frac{1}{q}}. 
	\end{align*}
\end{proof}
	
\begin{theorem}\label{them.prev.2}
	Let $1 < m \le r < \min\{p,q\} < \infty$ and $b \in \W (0,\infty) \cap \mp^+ ((0,\infty);\dn)$ be such that the function $B(t)$ satisfies  $0 < B(t) < \infty$ for every $t \in (0,\infty)$. Assume that $u \in \W(0,\infty) \cap C(0,\infty)$, $v \in \W_{m.p}(0,\infty)$ and $w \in \W(0,\infty)$. 
	Suppose that
	$$
	0 < \int_0^t \bigg( \int_s^t \bigg( \frac{B(y)}{y} \bigg)^{p'}\,dy \bigg)^{\frac{m'}{p'}} v_2(s)\,ds < \infty, \qquad t \in (0,\infty),
	$$
	\begin{gather*}
    \int_0^t \frac{s^{\frac{m'}{p'}} v_0(s)}{v_1(s)^{m' + 1}}\,ds < \infty, \quad \int_t^{\infty} \frac{s^{m' \big( 1 - \frac{2}{p}\big)} v_0 (s)}{v_1(s)^{m' + 1}} \,ds < \infty \qquad t \in (0,\infty), \\
    \intertext{and} \int_0^1 \frac{s^{m' \big( 1 - \frac{2}{p}\big)} v_0(s)}{v_1(s)^{m' + 1}} \,ds = \int_1^{\infty} \frac{s^{\frac{m'}{p'}} v_0(s)}{v_1(s)^{m' + 1}}\,ds =  \infty,
    \end{gather*}
    where $v_0$, $v_1$ and $v_2$ are defined by \eqref{defof_v}, \eqref{defof_u} and \eqref{defof_v2}, respectively. 
    Denote by
    $$
    V_2(t) : = \bigg( \int_0^t \frac{s^{\frac{m'}{p'}}v_0(s)}{v_1(s)^{m' + 1}}\,ds \bigg)^{\frac{1}{m'}}, \quad V_3(t) = \bigg( \int_0^{\infty} \bigg( \frac{t}{y + t}\bigg)^{\frac{m'}{p}}v_2(y)\,dy \bigg)^{\frac{1}{m'}}, \qquad 0 < t < \infty,
    $$ 
    $$
    {\mathcal B}(t,s) = \bigg( \int_t^s \bigg( \frac{B(y)}{y} \bigg)^{p'}\,dy \bigg)^{\frac{p(r-1)}{p-r}}  \,  \bigg( \frac{B(s)}{s} \bigg)^{p'}, \qquad 0 < t < s < \infty,
    $$	
    and 
    $$
    {\mathcal K}(\tau,t) = u(\tau) (\tau+t)^{-\frac{2}{2p-r}}, \qquad 0 < \tau,\,t < \infty.
    $$
	
	{\rm (i)} If $p \le q$, then
	\begin{align*}
	K \approx 
	&  \, \sup_{t \in (0,\infty)} \bigg( \int_0^t \bigg( \int_s^t \bigg( \frac{B(y)}{y} \bigg)^{p'}\,dy \bigg)^{\frac{m'}{p'}}
	v_2(s)\,ds \bigg)^{\frac{1}{m'}} \bigg( \sup_{t \le \tau} \frac{u(\tau)}{B(\tau)} \bigg) \bigg( \int_0^{\infty} \bigg(\frac{yt}{y+t}\bigg)^{\frac{q}{r}}w(y)\,dy \bigg)^{\frac{1}{q}} \\
	& +  \, \sup_{t \in (0,\infty)} \bigg( \int_0^t \bigg( \int_s^t \bigg( \frac{B(y)}{y} \bigg)^{p'}\,dy \bigg)^{\frac{m'}{p'}}
	v_2(s)\,ds \bigg)^{\frac{1}{m'}} \bigg(\int_t^{\infty} \bigg( \int_t^y \bigg( \sup_{x \le \tau}\bigg( \frac{u(\tau)}{B(\tau)} \bigg)^r  \bigg) \,dx \bigg)^{\frac{q}{r}}w(y)\,dy\bigg)^{\frac{1}{q}} \\
	& +\, \sup_{t \in (0,\infty)} V_2(t)  \sup_{t \le \tau} \bigg( \frac{u(\tau)}{B(\tau)} \bigg) \bigg(\int_t^{\tau} {\mathcal B}(t,s) \, ds \bigg)^{\frac{p-r}{pr}} \bigg(\int_0^{\infty} \bigg(\frac{yt}{y+t}\bigg)^{\frac{q}{r}}w(y)\,dy\bigg)^{\frac{1}{q}}\\
	& + \, \sup_{t \in (0,\infty)} V_2(t)  \, \sup_{x \in (t,\infty)} \bigg( \int_t^x {\mathcal B}(t,y) \, \bigg( \int_y^x \sup_{s \le \tau} \bigg( \frac{u(\tau)}{B(\tau)} \bigg)^{r} \,ds \bigg)^{\frac{p}{p-r}} \, dy \bigg)^{\frac{p-r}{pr}} \, \bigg( \int_x^{\infty} w(y) \,dy \bigg)^{\frac{1}{q}} & \\
	& + \, \sup_{t \in (0,\infty)} V_2(t)  \, \sup_{x \in (t,\infty)} \bigg( \int_t^x {\mathcal B}(t,y)  \, dy \bigg)^{\frac{p-r}{pr}} \, \bigg( \int_x^{\infty} \, \bigg( \int_x^y \bigg( \sup_{s \le \tau}\frac{u(\tau)}{B(\tau)}\bigg)^r \,ds \bigg)^{\frac{q}{r}} w(y) \,dy \bigg)^{\frac{1}{q}} \\
	& + \, \sup_{t \in (0,\infty)} V_2(t)  \, \sup_{x \in (t,\infty)} \bigg( \int_{[x,\infty)} \, d \, \bigg( - \sup_{t \le \tau} \bigg(\frac{u(\tau)}{B(\tau)}\bigg)^{\frac{pr}{p-r}} \bigg( \int_t^{\tau} {\mathcal B}(t,y)\,dy \bigg) \bigg) \bigg)^{\frac{p-r}{pr}} \, \bigg( \int_t^x (y - t)^{\frac{q}{r}} w(y) \,dy \bigg)^{\frac{1}{q}} \\
	& + \, \sup_{t \in (0,\infty)} V_2(t)  \, \sup_{x \in (t,\infty)} \bigg( \int_{(t,x]} \,  (y - t)^{\frac{p}{p-r}} \, d \, \bigg( - \sup_{y \le \tau} \bigg(\frac{u(\tau)}{B(\tau)}\bigg)^{\frac{pr}{p-r}} \bigg( \int_t^{\tau} {\mathcal B}(t,y)  \, dy \bigg) \bigg) \bigg)^{\frac{p-r}{pr}} \, \bigg( \int_x^{\infty} w(y) \,dy \bigg)^{\frac{1}{q}} \\
	& + \, \sup_{t \in (0,\infty)} V_2(t)  \, \bigg( \int_t^{\infty} (y - t)^{\frac{q}{r}} w(y)\,dy \bigg)^{\frac{1}{q}} \lim_{x \rightarrow \infty} \bigg(\sup_{x \le \tau} \frac{u(\tau)}{B(\tau)} \bigg( \int_t^{\tau} {\mathcal B}(t,y)  \, dy \bigg)^{\frac{p-r}{pr}}\bigg) \\
	& + \sup_{t \in (0,\infty)}  V_3(t) \,
	\sup_{s \in (0,\infty)} \bigg( \int_0^s (y+t)^{\frac{p(3r -2p)}{(2p-r)(p-r)}} \bigg( \int_y^s \bigg( \sup_{x \le \tau} {\mathcal K}(\tau,t)^r \bigg)\,dx \bigg)^{\frac{p}{p-r}}\,dy\bigg)^{\frac{p-r}{pr}} \, \bigg( \int_s^{\infty} w(z)\,dz \bigg)^{\frac{1}{q}} 
	\\
	& + \sup_{t \in (0,\infty)}  V_3(t) \, \sup_{s \in (0,\infty)}\bigg(\int_0^s (y+t)^{\frac{p(3r-2p)}{(2p-r)(p-r)}}\,dy\bigg)^\frac{p-r}{pr} \bigg(\int_s^{\infty} \bigg(\int_s^y\bigg(\sup_{x \le \tau} {\mathcal K}(\tau,t)^r \bigg)\,dx\bigg)^\frac{q}{r}\, w(y)dy\bigg)^\frac{1}{q}
	\\
	& + \sup_{t \in (0,\infty)}  V_3(t) \, \sup_{x \in (0,\infty)} \bigg(\int_{[x,\infty)}\,d\bigg(-\sup_{s\le \tau} {\mathcal K}(\tau,t)^\frac{pr}{p-r} \bigg(\int_0^{\tau}(y+t)^{\frac{p(3r-2p)}{(2p-r)(p-r)}}\,dy\bigg)\bigg)\bigg)^\frac{p-r}{pr} \bigg(\int_0^{x}y^\frac{q}{r}w(y)\,dy\bigg)^\frac{1}{q}
	\\
	& + \sup_{t \in (0,\infty)}  V_3(t) \, \sup_{x \in (0,\infty)} \bigg(\int_{(0,x]} s^\frac{p}{p-r}\,d\bigg(-\sup_{s\le \tau} {\mathcal K}(\tau,t)^\frac{pr}{p-r} \bigg(\int_0^{\tau}(y+t)^{\frac{p(3r-2p)}{(2p-r)(p-r)}}\,dy\bigg)\bigg)\bigg)^\frac{p-r}{pr} \bigg(\int_{x}^{\infty}w(y)\,dy\bigg)^\frac{1}{q}
	\\
	& + \bigg(\int_{0}^{\infty}y^\frac{q}{r}w(y)\,dy\bigg)^\frac{1}{q} \,\sup_{t \in (0,\infty)}  V_3(t) \, \lim_{s \rightarrow \infty}\bigg(\sup_{s\le \tau} {\mathcal K}(\tau,t) \bigg(\int_0^{\tau}(y+t)^{\frac{p(3r-2p)}{(2p-r)(p-r)}}\,dy\bigg)^\frac{p-r}{pr}\bigg)
	\\
	& + \sup_{t \in (0,\infty)}  V_3(t) \, t^\frac{r}{p (2p - r)}  \, \bigg(\int_{0}^{\infty}\bigg(\int_{0}^{x}\bigg(\sup_{s\le \tau} {\mathcal K}(\tau,t)^r \bigg)\,ds\bigg)^\frac{q}{r}w(x)\,dx\bigg)^\frac{1}{q}.
	\end{align*}

	{\rm (ii)} If $q < p$, then	
	\begin{align*}
	K \approx 
	&  \, \sup_{t \in (0,\infty)} \bigg( \int_0^t \bigg( \int_s^t \bigg( \frac{B(y)}{y} \bigg)^{p'}\,dy \bigg)^{\frac{m'}{p'}}
	v_2(s)\,ds \bigg)^{\frac{1}{m'}} \bigg( \sup_{t \le \tau} \frac{u(\tau)}{B(\tau)} \bigg) \bigg( \int_0^{\infty} \bigg(\frac{yt}{y+t}\bigg)^{\frac{q}{r}}w(y)\,dy \bigg)^{\frac{1}{q}} \\
	& +  \, \sup_{t \in (0,\infty)} \bigg( \int_0^t \bigg( \int_s^t \bigg( \frac{B(y)}{y} \bigg)^{p'}\,dy \bigg)^{\frac{m'}{p'}}
	v_2(s)\,ds \bigg)^{\frac{1}{m'}} \bigg(\int_t^{\infty} \bigg( \int_t^y \bigg( \sup_{x \le \tau}\bigg( \frac{u(\tau)}{B(\tau)} \bigg)^r  \bigg) \,dx \bigg)^{\frac{q}{r}}w(y)\,dy\bigg)^{\frac{1}{q}} \\
	& + \sup_{t \in (0,\infty)} V_2(t)  \sup_{t \le \tau} \bigg( \frac{u(\tau)}{B(\tau)} \bigg) \bigg(\int_t^{\tau} {\mathcal B}(t,s) \, ds \bigg)^{\frac{p-r}{pr}} \bigg(\int_0^{\infty} \bigg(\frac{yt}{y+t}\bigg)^{\frac{q}{r}}w(y)\,dy\bigg)^{\frac{1}{q}}\\
	& + \sup_{t \in (0,\infty)} V_2(t) \,\bigg(\int_t^{\infty} \bigg( \int_t^{\tau} {\mathcal B}(t,s) \,ds \bigg)^{\frac{p(q-r)}{r(p-q)}} {\mathcal B}(t,\tau) \, \bigg( \int_{\tau}^{\infty} \bigg( \int_{\tau}^z \sup_{s \le y} \bigg( \frac{u(y)}{B(y)} \bigg)^r \,ds \bigg)^{\frac{q}{r}} w(z) \,dz \bigg)^{\frac{p}{p-q}} \,d\tau\bigg)^{\frac{p-q}{pq}} \\
	& + \sup_{t \in (0,\infty)} V_2(t)  \, \bigg(\int_t^{\infty} \bigg( \int_t^{\tau} {\mathcal B}(t,s) \, \bigg( \int_s^{\tau} \sup_{z \le y} \bigg( \frac{u(y)}{B(y)} \bigg)^r \,dz \bigg)^{\frac{p}{p-r}} \,ds \bigg)^{\frac{q(p-r)}{r(p-q)}}  \bigg( \int_{\tau}^{\infty} w(x)\,dx\bigg)^{\frac{q}{p-q}} \, w(\tau)\,d\tau\bigg)^{\frac{p-q}{pq}} \\
	& + \sup_{t \in (0,\infty)} V_2(t) \, \bigg(\int_t^{\infty} \bigg( \int_{[x,\infty)} d \bigg( - \bigg( \sup_{y \le \tau} \bigg( \frac{u(\tau)}{B(\tau)} \bigg)^{\frac{pr}{p-r}} \bigg( \int_t^{\tau} {\mathcal B}(t,s)\,ds \bigg) \bigg) \bigg) \bigg)^{\frac{q(p-r)}{r(p-q)}} \bigg( \int_t^x (z - t)^{\frac{q}{r}} w(z) \,dz \bigg)^{\frac{q}{p-q}} (x-t)^{\frac{q}{r}} w(x) \,dx\bigg)^{\frac{p-q}{pq}} \\
	& + \sup_{t \in (0,\infty)} V_2(t) \, \bigg(\int_t^{\infty} \bigg( \int_{(t,x]} (y - t)^{\frac{p}{p-r}}d \bigg( - \bigg( \sup_{y \le \tau} \bigg( \frac{u(\tau)}{B(\tau)} \bigg)^{\frac{pr}{p-r}} \bigg( \int_t^{\tau} {\mathcal B}(t,s)\,ds \bigg) \bigg) \bigg) \bigg)^{\frac{q(p-r)}{r(p-q)}} \bigg( \int_x^{\infty} w(z) \,dz \bigg)^{\frac{q}{p-q}} w(x) \,dx\bigg)^{\frac{p-q}{pq}} \\
	& + \sup_{t \in (0,\infty)} V_2(t) \, \bigg( \int_t^{\infty} (y - t)^{\frac{q}{r}} w(y)\,dy \bigg)^{\frac{1}{q}} \lim_{x \rightarrow \infty} \bigg(\sup_{x \le \tau} \bigg( \frac{u(\tau)}{B(\tau)} \bigg)\bigg( \int_t^{\tau} {\mathcal B}(t,s)  \, ds \bigg)^{\frac{p-r}{pr}}\bigg) \\
	& + \sup_{t \in (0,\infty)}  V_3(t) \, \bigg(\int_0^{\infty} \bigg( \int_0^x (\tau+t)^{\frac{p(3r - 2p)}{(2p-r)(p-r)}} \,d\tau \bigg)^{\frac{p(q-r)}{r(p-q)}} \, (x+t)^{\frac{p(3r - 2p)}{(2p-r)(p-r)}} \, \bigg( \int_x^{\infty} \bigg( \int_x^z \bigg( \sup_{y \le \tau} {\mathcal K}(\tau,t)^r \bigg) \,dy \bigg)^{\frac{q}{r}} \, w(z)\,dz \bigg)^{\frac{p}{p-q}} \,dx \bigg)^{\frac{p-q}{pq}} \\
	& + \sup_{t \in (0,\infty)}  V_3(t) \, \bigg( \int_0^{\infty} \bigg( \int_0^{y} (x+t)^{\frac{p(3r - 2p)}{(2p - r)(p-r)}} \bigg( \int_x^y \bigg( \sup_{s \le \tau} {\mathcal K}(\tau,t)^r \bigg)\,ds \bigg)^{\frac{p}{p-r}}\,dx \bigg)^{\frac{q(p-r)}{r(p-q)}} \bigg(\int_y^{\infty} w(z)\,dz\bigg)^{\frac{q}{p-q}} w(y)\,dy \bigg)^{\frac{p-q}{pq}} \\
	& + \sup_{t \in (0,\infty)}  V_3(t) \, \bigg( \int_0^{\infty} \bigg( \int_{[x,\infty)} d\,\bigg( - \bigg(\sup_{y \le \tau} {\mathcal K}(\tau,t)^{\frac{pr}{p-r}} \bigg(\int_0^{\tau} (z+t)^{\frac{p(3r-2p)}{(2p - r)(p-r)}}\,dz\bigg)\bigg)\bigg)\bigg)^{\frac{q(p-r)}{r(p-q)}} \bigg(\int_0^x z^{\frac{q}{r}} w(z) \,dz\bigg)^{\frac{q}{p-q}} x^{\frac{q}{r}} w(x)\,dx \bigg)^{\frac{p-q}{pq}} \\
	& + \sup_{t \in (0,\infty)}  V_3(t) \, \bigg( \int_0^{\infty} \bigg( \int_{(0,x]} y^{\frac{p}{p-r}} d \bigg( - \bigg(\sup_{y \le \tau} {\mathcal K}(\tau,t)^{\frac{pr}{p-r}} \bigg( \int_0^{\tau} (z+t)^{\frac{p(3r-2p)}{(2p-r)(p-r)}}\,dz \bigg)\bigg)\bigg)\bigg)^{\frac{q(p-r)}{r(p-q)}} \bigg(\int_x^{\infty} w\bigg)^{\frac{q}{p-q}} w(x)\,dx  \bigg)^{\frac{p-q}{pq}}\\
	& + \bigg(\int_{0}^{\infty}y^\frac{q}{r}w(y)\,dy\bigg)^\frac{1}{q} \,\sup_{t \in (0,\infty)}  V_3(t) \, \lim_{s \rightarrow \infty}\bigg(\sup_{s\le \tau} {\mathcal K}(\tau,t) \bigg(\int_0^{\tau}(y+t)^{\frac{p(3r-2p)}{(2p-r)(p-r)}}\,dy\bigg)^\frac{p-r}{pr}\bigg)
	\\
	& + \sup_{t \in (0,\infty)}  V_3(t) \, t^\frac{r}{p (2p - r)}  \, \bigg(\int_{0}^{\infty}\bigg(\int_{0}^{x}\bigg(\sup_{s\le \tau} {\mathcal K}(\tau,t)^r \bigg)\,ds\bigg)^\frac{q}{r}w(x)\,dx\bigg)^\frac{1}{q}.
	\end{align*}
\end{theorem}

\begin{proof}
Recall that, by Lemma \ref{R2}, we know that
$$
K \approx A + B.
$$

We estimate $A$. By Theorem \ref{krepick}, (b), we have that
\begin{align*}
\sup_{\vp \in {\mathfrak M}^+} \frac{1}{\|\vp\|_{r',h^{1-r'},(0,\infty)}}
	\bigg( \int_0^{\infty} \bigg( \int_t^{\infty} \bigg( \int_s^{\infty} \vp(x) \,dx  \bigg)^{p'} \bigg( \frac{B(s)}{s} \bigg)^{p'}\,ds \bigg)^{\frac{m'}{p'}} v_2(t)\,dt \bigg)^{\frac{1}{m'}} & \\
& \hspace{-11cm} \approx \sup_{t \in (0,\infty)} \bigg( \int_0^t \bigg( \int_s^t \bigg( \frac{B(y)}{y} \bigg)^{p'}\,dy \bigg)^{\frac{m'}{p'}} v_2(s)\,ds \bigg)^{\frac{1}{m'}} \, \bigg( \int_t^{\infty} h(x) \bigg( \frac{u(x)}{B(x)} \bigg)^r \,dx \bigg)^{\frac{1}{r}} \\
& \hspace{-10.5cm} + \sup_{t \in (0,\infty)} \bigg( \int_0^t v_2(s)\,ds \bigg)^{\frac{1}{m'}} \, \bigg(\int_t^{\infty} \bigg( \int_t^s \bigg( \frac{B(y)}{y} \bigg)^{p'}\,dy \bigg)^{\frac{p(r-1)}{p-r}}  \,  \bigg( \frac{B(s)}{s} \bigg)^{p'} \, \bigg( \int_s^{\infty} h(x) \bigg( \frac{u(x)}{B(x)} \bigg)^r \,dx \bigg)^{\frac{p}{p-r}} \,ds \bigg)^{\frac{p-r}{pr}} \\
& \hspace{-11cm} : = I + II.
\end{align*}

Consequently, we get that
\begin{align*}
	A \approx & \, \sup_{g \in \mp^+} \frac{ 1 } {\|g\|_{\frac{q}{q-r},w^{\frac{r}{r-q}}, (0,\infty)}^{\frac{1}{r}}} \sup_{h:\, \int_0^x h \le \int_0^x \big( \int_{\tau}^{\infty} g\big)\,d\tau} \big\{ I + II \big\} : = A_1 + A_2.
\end{align*} 

It has been already shown in the proof of Theorem \ref{thm.prev} that
\begin{align*}
	A_1 \approx	&  \, \sup_{t \in (0,\infty)} \bigg( \int_0^t \bigg( \int_s^t \bigg( \frac{B(y)}{y} \bigg)^{p'}\,dy \bigg)^{\frac{m'}{p'}}
	v_2(s)\,ds \bigg)^{\frac{1}{m'}} \bigg( \sup_{t \le \tau} \frac{u(\tau)}{B(\tau)} \bigg) \bigg( \int_0^{\infty} \bigg(\frac{yt}{y+t}\bigg)^{\frac{q}{r}}w(y)\,dy \bigg)^{\frac{1}{q}} \\
	& +  \, \sup_{t \in (0,\infty)} \bigg( \int_0^t \bigg( \int_s^t \bigg( \frac{B(y)}{y} \bigg)^{p'}\,dy \bigg)^{\frac{m'}{p'}}
	v_2(s)\,ds \bigg)^{\frac{1}{m'}} \bigg(\int_t^{\infty} \bigg( \int_t^y \bigg( \sup_{x \le \tau}\bigg( \frac{u(\tau)}{B(\tau)} \bigg)^r  \bigg) \,dx \bigg)^{\frac{q}{r}}w(y)\,dy\bigg)^{\frac{1}{q}}.
\end{align*}

Interchanging the suprema, we get that
\begin{align*}
	A_2 =  \sup_{g \in \mp^+} \frac{ 1 }
	{\|g\|_{\frac{q}{q-r},w^{\frac{r}{r-q}},(0,\infty)}^{\frac{1}{r}}} \sup_{t \in (0,\infty)} V_2(t) \sup_{h:\, \int_0^x h \le \int_0^x \big( \int_{\tau}^{\infty} g\big)\,d\tau} \bigg(\int_t^{\infty} {\mathcal B}(t,s)\, \bigg( \int_s^{\infty} h(x) \bigg( \frac{u(x)}{B(x)} \bigg)^r \,dx \bigg)^{\frac{p}{p-r}} \, ds \bigg)^{\frac{p-r}{pr}}.
\end{align*}

By duality (recall that $p / (p-r) = (p / r)'$), we have that
\begin{align*}
\bigg(\int_t^{\infty} {\mathcal B}(t,s) \,  \bigg( \int_s^{\infty} h(x) \bigg( \frac{u(x)}{B(x)} \bigg)^r \,dx \bigg)^{\frac{p}{p-r}} \,ds \bigg)^{\frac{p-r}{pr}} = \left\{ \sup_{\psi \in {\mathfrak M}^+(t,\infty)} \frac{\int_t^{\infty} \bigg( \int_s^{\infty} h(x) \bigg( \frac{u(x)}{B(x)} \bigg)^r \,dx \bigg) \, \psi (s) \,  {\mathcal B}(t,s) \,ds }{\bigg( \int_t^{\infty} \psi(s)^{\frac{p}{r}}  \, {\mathcal B}(t,s) \,  ds \bigg)^{\frac{r}{p}}} \right\}^{\frac{1}{r}}.
\end{align*}

Thus
\begin{align*}
	A_2 =  \sup_{g \in \mp^+} \frac{ 1 }
	{\|g\|_{\frac{q}{q-r},w^{\frac{r}{r-q}},(0,\infty)}^{\frac{1}{r}}}  \sup_{t \in (0,\infty)} V_2(t) & \\
	& \hspace{-3cm} \left\{\sup_{h:\, \int_0^x h \le \int_0^x \big( \int_{\tau}^{\infty} g\big)\,d\tau} \sup_{\psi \in {\mathfrak M}^+(t,\infty)} \frac{\int_t^{\infty} \bigg( \int_s^{\infty} h(x) \bigg( \frac{u(x)}{B(x)} \bigg)^r \,dx \bigg) \, \psi (s) \,  {\mathcal B}(t,s) \,ds }{\bigg( \int_t^{\infty} \psi(s)^{\frac{p}{r}}  \, {\mathcal B}(t,s) \,  ds \bigg)^{\frac{r}{p}}} \right\}^{\frac{1}{r}}.
\end{align*}

By Fubini's Theorem, interchanging the suprema, we get that
\begin{align*}
	A_2 =  \sup_{g \in \mp^+} \frac{ 1 }
	{\|g\|_{\frac{q}{q-r},w^{\frac{r}{r-q}},(0,\infty)}^{\frac{1}{r}}} \sup_{t \in (0,\infty)} V_2(t) \left\{ \sup_{\psi \in {\mathfrak M}^+(t,\infty)} \bigg( \int_t^{\infty} \psi(s)^{\frac{p}{r}}  \, {\mathcal B}(t,s) \,  ds \bigg)^{-\frac{r}{p}} \right. & \\
	& \hspace{-7cm}  \left. \sup_{h:\, \int_0^x h \le \int_0^x \big( \int_{\tau}^{\infty} g\big)\,d\tau}  \int_t^{\infty} h(x) \bigg( \frac{u(x)}{B(x)} \bigg)^r  \, \int_t^x \psi (s) \, {\mathcal B}(t,s) \, ds \,dx \right\}^{\frac{1}{r}}.
\end{align*}

By Theorem \ref{transfermon}, we have that
\begin{align*}
\sup_{h:\, \int_0^x h \le \int_0^x \big( \int_{\tau}^{\infty} g\big)\,d\tau}  \int_t^{\infty} h(x) \bigg( \frac{u(x)}{B(x)} \bigg)^r  \, \int_t^x \psi (s) \, {\mathcal B}(t,s) \, ds \,dx & \\
& \hspace{-7cm}  = \sup_{h:\, \int_0^x h \le \int_0^x \big( \int_{\tau}^{\infty} g\big)\,d\tau}  \int_0^{\infty} h(x) \chi_{(t,\infty)}(x) \bigg( \frac{u(x)}{B(x)} \bigg)^r  \, \int_t^x \psi (s) \, {\mathcal B}(t,s) \, ds \,dx \\
& \hspace{-7cm}  = \int_0^{\infty} \bigg( \int_x^{\infty} g(z)\,dz \bigg) \sup_{x \le \tau} \chi_{(t,\infty)}(\tau) \bigg( \frac{u(\tau)}{B(\tau)} \bigg)^r  \, \int_t^{\tau} \psi (s) \, {\mathcal B}(t,s) \, ds \,dx \\
& \hspace{-7cm}  = \int_0^t \bigg( \int_x^{\infty} g(z)\,dz \bigg) \sup_{x \le \tau} \chi_{(t,\infty)}(\tau) \bigg( \frac{u(\tau)}{B(\tau)} \bigg)^r  \, \int_t^{\tau} \psi (s) \, {\mathcal B}(t,s) \, ds \,dx \\
& \hspace{-6.5cm} + \int_t^{\infty} \bigg( \int_x^{\infty} g(z)\,dz \bigg) \sup_{x \le \tau} \chi_{(t,\infty)}(\tau) \bigg( \frac{u(\tau)}{B(\tau)} \bigg)^r  \, \int_t^{\tau} \psi (s) \, {\mathcal B}(t,s) \, ds \,dx.
\end{align*}
Since 
\begin{align*}
\sup_{x \le \tau} \chi_{(t,\infty)}(\tau) \bigg( \frac{u(\tau)}{B(\tau)} \bigg)^r  \, \int_t^{\tau} \psi (s) \, {\mathcal B}(t,s) \, ds  = \sup_{t \le \tau} \bigg( \frac{u(\tau)}{B(\tau)} \bigg)^r  \, \int_t^{\tau} \psi (s) \, {\mathcal B}(t,s) \, ds,
\end{align*}
when $0 < x \le t$, by Fubini's Theorem,  we get that
\begin{align*}
\sup_{h:\, \int_0^x h \le \int_0^x \big( \int_{\tau}^{\infty} g\big)\,d\tau}  \int_t^{\infty} h(x) \bigg( \frac{u(x)}{B(x)} \bigg)^r  \, \int_t^x \psi (s) \, {\mathcal B}(t,s) \, ds \,dx & \\
& \hspace{-7cm}  = \sup_{t \le \tau} \bigg( \frac{u(\tau)}{B(\tau)} \bigg)^r  \, \int_t^{\tau} \psi (s) \, {\mathcal B}(t,s) \, ds \, \int_0^t \bigg( \int_x^{\infty} g(z)\,dz \bigg) \,dx \\
& \hspace{-6.5cm} + \int_t^{\infty} \bigg( \int_x^{\infty} g(z)\,dz \bigg) \sup_{x \le \tau} \bigg( \frac{u(\tau)}{B(\tau)} \bigg)^r  \, \int_t^{\tau} \psi (s) \, {\mathcal B}(t,s) \, ds \,dx \\
& \hspace{-7cm}  = t \sup_{t \le \tau} \bigg( \frac{u(\tau)}{B(\tau)} \bigg)^r  \, \int_t^{\tau} \psi (s) \, {\mathcal B}(t,s) \, ds \, \int_t^{\infty} g(z)\,dz \\ 
& \hspace{-6.5cm} + \sup_{t \le \tau} \bigg( \frac{u(\tau)}{B(\tau)} \bigg)^r  \, \int_t^{\tau} \psi (s) \, {\mathcal B}(t,s) \, ds \, \int_0^t g(z) \,z\,dz \\
& \hspace{-6.5cm} + \int_t^{\infty} g(z) \int_t^z \sup_{x \le \tau} \bigg( \frac{u(\tau)}{B(\tau)} \bigg)^r  \, \int_t^{\tau} \psi (s) \, {\mathcal B}(t,s) \,ds \,dx \,dz\\
& \hspace{-7cm} \approx \sup_{t \le \tau} \bigg( \frac{u(\tau)}{B(\tau)} \bigg)^r  \, \int_t^{\tau} \psi (s) \, {\mathcal B}(t,s) \, ds \, \int_0^{\infty} g(z)\,\bigg(\frac{zt}{z+t}\bigg) \, dz \\
& \hspace{-6.5cm} + \int_t^{\infty} g(z) \int_t^z \sup_{x \le \tau} \bigg( \frac{u(\tau)}{B(\tau)} \bigg)^r  \, \int_t^{\tau} \psi (s) \, {\mathcal B}(t,s) \,ds \,dx \,dz
\end{align*}

Thus
\begin{align*}
	A_2 \approx & \sup_{g \in \mp^+} \frac{ 1 } {\|g\|_{\frac{q}{q-r},w^{\frac{r}{r-q}},(0,\infty)}^{\frac{1}{r}}} \sup_{t \in (0,\infty)} V_2(t)  \left\{ \sup_{\psi \in {\mathfrak M}^+(t,\infty)} \frac{\sup_{t \le \tau} \bigg( \frac{u(\tau)}{B(\tau)} \bigg)^r  \, \int_t^{\tau} \psi (s) \, {\mathcal B}(t,s) \, ds \, \int_0^{\infty} g(z)\,\bigg(\frac{zt}{z+t}\bigg)\,dz}{\bigg( \int_t^{\infty} \psi(s)^{\frac{p}{r}}  \, {\mathcal B}(t,s) \,  ds \bigg)^{\frac{r}{p}}} \right\}^{\frac{1}{r}} \\
	& + \sup_{g \in \mp^+} \frac{ 1 } 	{\|g\|_{\frac{q}{q-r},w^{\frac{r}{r-q}},(0,\infty)}^{\frac{1}{r}}} \sup_{t \in (0,\infty)} V_2(t) \left\{ \sup_{\psi \in {\mathfrak M}^+(t,\infty)} \frac{\int_t^{\infty} g(z) \int_t^z \sup_{x \le \tau} \bigg( \frac{u(\tau)}{B(\tau)} \bigg)^r  \, \int_t^{\tau} \psi (s) \, {\mathcal B}(t,s) \,ds \,dx \,dz}{\bigg( \int_t^{\infty} \psi(s)^{\frac{p}{r}}  \, {\mathcal B}(t,s) \,  ds \bigg)^{\frac{r}{p}}} \right\}^{\frac{1}{r}}.
\end{align*}
Interchanging the suprema, we get that
\begin{align*}
	A_2 \approx & \sup_{t \in (0,\infty)} V_2(t) \sup_{t \le \tau} \bigg( \frac{u(\tau)}{B(\tau)} \bigg)  \left\{\sup_{\psi \in {\mathfrak M}^+(t,\infty)} \frac{\int_t^{\tau} \psi (s) \, {\mathcal B}(t,s) \, ds}{\bigg( \int_t^{\infty} \psi(s)^{\frac{p}{r}}  \, {\mathcal B}(t,s) \,  ds \bigg)^{\frac{r}{p}}} \sup_{g \in \mp^+} \frac{ \int_0^{\infty} g(z)\,\bigg(\frac{zt}{z+t}\bigg)\,dz} {\|g\|_{\frac{q}{q-r},w^{\frac{r}{r-q}},(0,\infty)}} \right\}^{\frac{1}{r}} \\
	& + \sup_{t \in (0,\infty)} V_2(t) \left\{ \sup_{\psi \in {\mathfrak M}^+(t,\infty)} \frac{1}{\bigg( \int_t^{\infty} \psi(s)^{\frac{p}{r}}  \, {\mathcal B}(t,s) \,  ds \bigg)^{\frac{r}{p}}} \,\sup_{g \in \mp^+} \frac{ \int_t^{\infty} g(z) \int_t^z \sup_{x \le \tau} \bigg( \frac{u(\tau)}{B(\tau)} \bigg)^r  \, \int_t^{\tau} \psi (s) \, {\mathcal B}(t,s) \,ds \,dx \,dz }
	{\|g\|_{\frac{q}{q-r},w^{\frac{r}{r-q}},(0,\infty)}} \right\}^{\frac{1}{r}}\\
	\approx & \sup_{t \in (0,\infty)} V_2(t) \left\{ \sup_{t \le \tau} \bigg( \frac{u(\tau)}{B(\tau)} \bigg)^r \sup_{\psi \in {\mathfrak M}^+(t,\infty)} \frac{\int_t^{\infty} \psi (s) \, {\mathcal B}(t,s) \chi_{(t,\tau)}(s)\, ds}{\bigg( \int_t^{\infty} \psi(s)^{\frac{p}{r}}  \, {\mathcal B}(t,s) \,  ds \bigg)^{\frac{r}{p}}} \sup_{g \in \mp^+} \frac{ \int_0^{\infty} g(z)\,\bigg(\frac{zt}{z+t}\bigg)\,dz} {\|g\|_{\frac{q}{q-r},w^{\frac{r}{r-q}},(0,\infty)}} \right\}^{\frac{1}{r}} \\
	& + \sup_{t \in (0,\infty)} V_2(t) \left\{ \sup_{\psi \in {\mathfrak M}^+(t,\infty)} \frac{1}{\bigg( \int_t^{\infty} \psi(s)^{\frac{p}{r}}  \, {\mathcal B}(t,s) \,  ds \bigg)^{\frac{r}{p}}} \,\sup_{g \in \mp^+} \frac{ \int_0^{\infty} g(z) \int_t^z \sup_{x \le \tau} \bigg( \frac{u(\tau)}{B(\tau)} \bigg)^r  \, \int_t^{\tau} \psi (s) \, {\mathcal B}(t,s) \,ds \,dx \, \chi_{(t,\infty)}(z)\,dz }
	{\|g\|_{\frac{q}{q-r},w^{\frac{r}{r-q}},(0,\infty)}} \right\}^{\frac{1}{r}}.
\end{align*}
By duality, we arrive at
\begin{align*}
	A_2 \approx & \sup_{t \in (0,\infty)} V_2(t)  \sup_{t \le \tau} \bigg( \frac{u(\tau)}{B(\tau)} \bigg) \bigg(\int_t^{\tau} {\mathcal B}(t,s) \, ds \bigg)^{\frac{p-r}{pr}} \bigg(\int_0^{\infty} \bigg(\frac{yt}{y+t}\bigg)^{\frac{q}{r}}w(y)\,dy\bigg)^{\frac{1}{q}}\\
	& + \sup_{t \in (0,\infty)} V_2(t) \left\{ \sup_{\psi \in {\mathfrak M}^+(t,\infty)} \frac{\bigg( \int_t^{\infty}  \bigg( \int_t^z \sup_{x \le \tau} \bigg( \frac{u(\tau)}{B(\tau)} \bigg)^r  \, \int_t^{\tau} \psi (s) \, {\mathcal B}(t,s) \,ds \,dx \bigg)^{\frac{q}{r}} w(z)\,dz\bigg)^{\frac{r}{q}}}{\bigg( \int_t^{\infty} \psi(s)^{\frac{p}{r}}  \, {\mathcal B}(t,s) \,  ds \bigg)^{\frac{r}{p}}} \right\}^{\frac{1}{r}}.
\end{align*}

{\rm (i)} Let $p \le q$. By Theorem \ref{aux.thm.2}, we have for any $t \in (0,\infty)$ that
\begin{align*}
\sup_{\psi \in {\mathfrak M}^+(t,\infty)} \frac{\bigg( \int_t^{\infty}  \bigg( \int_t^z \sup_{x \le \tau} \bigg( \frac{u(\tau)}{B(\tau)} \bigg)^r  \, \int_t^{\tau} \psi (s) \, {\mathcal B}(t,s) \,ds \,dx \bigg)^{\frac{q}{r}} w(z)\,dz\bigg)^{\frac{r}{q}}}{\bigg( \int_t^{\infty} \psi(s)^{\frac{p}{r}}  \, {\mathcal B}(t,s) \,  ds \bigg)^{\frac{r}{p}}} & \\
& \hspace{-9cm} \approx \,   \sup_{x \in (t,\infty)} \bigg( \int_t^x {\mathcal B}(t,y) \, \bigg( \int_y^x \sup_{s \le \tau} \bigg( \frac{u(\tau)}{B(\tau)} \bigg)^{r} \,ds \bigg)^{\frac{p}{p-r}} \, dy \bigg)^{\frac{p-r}{p}} \, \bigg( \int_x^{\infty} w(y) \,dy \bigg)^{\frac{r}{q}} \\
& \hspace{-8.5cm} +  \, \sup_{x \in (t,\infty)} \bigg( \int_t^x {\mathcal B}(t,y)  \, dy \bigg)^{\frac{p-r}{p}} \, \bigg( \int_x^{\infty} \, \bigg( \int_x^y \bigg( \sup_{s \le \tau}\frac{u(\tau)}{B(\tau)}\bigg)^r \,ds \bigg)^{\frac{q}{r}} w(y) \,dy \bigg)^{\frac{r}{q}} \\
& \hspace{-8.5cm} +  \, \sup_{x \in (t,\infty)} \bigg( \int_{[x,\infty)} \, d \, \bigg( - \sup_{t \le \tau} \bigg(\frac{u(\tau)}{B(\tau)}\bigg)^{\frac{pr}{p-r}} \bigg( \int_t^{\tau} {\mathcal B}(t,y)\,dy \bigg) \bigg) \bigg)^{\frac{p-r}{p}} \, \bigg( \int_t^x (y - t)^{\frac{q}{r}} w(y) \,dy \bigg)^{\frac{r}{q}} \\
& \hspace{-8.5cm} +  \, \sup_{x \in (t,\infty)} \bigg( \int_{(t,x]} \,  (y - t)^{\frac{p}{p-r}} \, d \, \bigg( - \sup_{y \le \tau} \bigg(\frac{u(\tau)}{B(\tau)}\bigg)^{\frac{pr}{p-r}} \bigg( \int_t^{\tau} {\mathcal B}(t,y)  \, dy \bigg) \bigg) \bigg)^{\frac{p-r}{p}} \, \bigg( \int_x^{\infty} w(y) \,dy \bigg)^{\frac{r}{q}} \\
& \hspace{-8.5cm} +  \, \bigg( \int_t^{\infty} (y - t)^{\frac{q}{r}} w(y)\,dy \bigg)^{\frac{1}{q}} \lim_{x \rightarrow \infty} \bigg(\sup_{x \le \tau} \bigg( \frac{u(\tau)}{B(\tau)} \bigg)^r \bigg( \int_t^{\tau} {\mathcal B}(t,y)  \, dy \bigg)^{\frac{p-r}{p}}\bigg).
\end{align*}

Consequently,
\begin{align*}
A_2 \approx & \sup_{t \in (0,\infty)} V_2(t)  \sup_{t \le \tau} \bigg( \frac{u(\tau)}{B(\tau)} \bigg) \bigg(\int_t^{\tau} {\mathcal B}(t,s) \, ds \bigg)^{\frac{p-r}{pr}} \bigg(\int_0^{\infty} \bigg(\frac{yt}{y+t}\bigg)^{\frac{q}{r}}w(y)\,dy\bigg)^{\frac{1}{q}}\\
& + \, \sup_{t \in (0,\infty)} V_2(t)  \, \sup_{x \in (t,\infty)} \bigg( \int_t^x {\mathcal B}(t,y) \, \bigg( \int_y^x \sup_{s \le \tau} \bigg( \frac{u(\tau)}{B(\tau)} \bigg)^{r} \,ds \bigg)^{\frac{p}{p-r}} \, dy \bigg)^{\frac{p-r}{pr}} \, \bigg( \int_x^{\infty} w(y) \,dy \bigg)^{\frac{1}{q}} & \\
& + \, \sup_{t \in (0,\infty)} V_2(t)  \, \sup_{x \in (t,\infty)} \bigg( \int_t^x {\mathcal B}(t,y)  \, dy \bigg)^{\frac{p-r}{pr}} \, \bigg( \int_x^{\infty} \, \bigg( \int_x^y \bigg( \sup_{s \le \tau}\frac{u(\tau)}{B(\tau)}\bigg)^r \,ds \bigg)^{\frac{q}{r}} w(y) \,dy \bigg)^{\frac{1}{q}} \\
& + \, \sup_{t \in (0,\infty)} V_2(t)  \, \sup_{x \in (t,\infty)} \bigg( \int_{[x,\infty)} \, d \, \bigg( - \sup_{t \le \tau} \bigg(\frac{u(\tau)}{B(\tau)}\bigg)^{\frac{pr}{p-r}} \bigg( \int_t^{\tau} {\mathcal B}(t,y)\,dy \bigg) \bigg) \bigg)^{\frac{p-r}{pr}} \, \bigg( \int_t^x (y - t)^{\frac{q}{r}} w(y) \,dy \bigg)^{\frac{1}{q}} \\
& + \, \sup_{t \in (0,\infty)} V_2(t)  \, \sup_{x \in (t,\infty)} \bigg( \int_{(t,x]} \,  (y - t)^{\frac{p}{p-r}} \, d \, \bigg( - \sup_{y \le \tau} \bigg(\frac{u(\tau)}{B(\tau)}\bigg)^{\frac{pr}{p-r}} \bigg( \int_t^{\tau} {\mathcal B}(t,y)  \, dy \bigg) \bigg) \bigg)^{\frac{p-r}{pr}} \, \bigg( \int_x^{\infty} w(y) \,dy \bigg)^{\frac{1}{q}} \\
& + \, \sup_{t \in (0,\infty)} V_2(t)  \, \bigg( \int_t^{\infty} (y - t)^{\frac{q}{r}} w(y)\,dy \bigg)^{\frac{1}{q}} \lim_{x \rightarrow \infty} \bigg(\sup_{x \le \tau} \frac{u(\tau)}{B(\tau)} \bigg( \int_t^{\tau} {\mathcal B}(t,y)  \, dy \bigg)^{\frac{p-r}{pr}}\bigg).
\end{align*}

Combining, we arrive at
\begin{align*}
A \approx &  \, \sup_{t \in (0,\infty)} \bigg( \int_0^t \bigg( \int_s^t \bigg( \frac{B(y)}{y} \bigg)^{p'}\,dy \bigg)^{\frac{m'}{p'}}
v_2(s)\,ds \bigg)^{\frac{1}{m'}} \bigg( \sup_{t \le \tau} \frac{u(\tau)}{B(\tau)} \bigg) \bigg( \int_0^{\infty} \bigg(\frac{yt}{y+t}\bigg)^{\frac{q}{r}}w(y)\,dy \bigg)^{\frac{1}{q}} \\
& +  \, \sup_{t \in (0,\infty)} \bigg( \int_0^t \bigg( \int_s^t \bigg( \frac{B(y)}{y} \bigg)^{p'}\,dy \bigg)^{\frac{m'}{p'}}
v_2(s)\,ds \bigg)^{\frac{1}{m'}} \bigg(\int_t^{\infty} \bigg( \int_t^y \bigg( \sup_{x \le \tau}\bigg( \frac{u(\tau)}{B(\tau)} \bigg)^r  \bigg) \,dx \bigg)^{\frac{q}{r}}w(y)\,dy\bigg)^{\frac{1}{q}} \\
& +\, \sup_{t \in (0,\infty)} V_2(t)  \sup_{t \le \tau} \bigg( \frac{u(\tau)}{B(\tau)} \bigg) \bigg(\int_t^{\tau} {\mathcal B}(t,s) \, ds \bigg)^{\frac{p-r}{pr}} \bigg(\int_0^{\infty} \bigg(\frac{yt}{y+t}\bigg)^{\frac{q}{r}}w(y)\,dy\bigg)^{\frac{1}{q}}\\
& + \, \sup_{t \in (0,\infty)} V_2(t)  \, \sup_{x \in (t,\infty)} \bigg( \int_t^x {\mathcal B}(t,y) \, \bigg( \int_y^x \sup_{s \le \tau} \bigg( \frac{u(\tau)}{B(\tau)} \bigg)^{r} \,ds \bigg)^{\frac{p}{p-r}} \, dy \bigg)^{\frac{p-r}{pr}} \, \bigg( \int_x^{\infty} w(y) \,dy \bigg)^{\frac{1}{q}} & \\
& + \, \sup_{t \in (0,\infty)} V_2(t)  \, \sup_{x \in (t,\infty)} \bigg( \int_t^x {\mathcal B}(t,y)  \, dy \bigg)^{\frac{p-r}{pr}} \, \bigg( \int_x^{\infty} \, \bigg( \int_x^y \bigg( \sup_{s \le \tau}\frac{u(\tau)}{B(\tau)}\bigg)^r \,ds \bigg)^{\frac{q}{r}} w(y) \,dy \bigg)^{\frac{1}{q}} \\
& + \, \sup_{t \in (0,\infty)} V_2(t)  \, \sup_{x \in (t,\infty)} \bigg( \int_{[x,\infty)} \, d \, \bigg( - \sup_{t \le \tau} \bigg(\frac{u(\tau)}{B(\tau)}\bigg)^{\frac{pr}{p-r}} \bigg( \int_t^{\tau} {\mathcal B}(t,y)\,dy \bigg) \bigg) \bigg)^{\frac{p-r}{pr}} \, \bigg( \int_t^x (y - t)^{\frac{q}{r}} w(y) \,dy \bigg)^{\frac{1}{q}} \\
& + \, \sup_{t \in (0,\infty)} V_2(t)  \, \sup_{x \in (t,\infty)} \bigg( \int_{(t,x]} \,  (y - t)^{\frac{p}{p-r}} \, d \, \bigg( - \sup_{y \le \tau} \bigg(\frac{u(\tau)}{B(\tau)}\bigg)^{\frac{pr}{p-r}} \bigg( \int_t^{\tau} {\mathcal B}(t,y)  \, dy \bigg) \bigg) \bigg)^{\frac{p-r}{pr}} \, \bigg( \int_x^{\infty} w(y) \,dy \bigg)^{\frac{1}{q}} \\
& + \, \sup_{t \in (0,\infty)} V_2(t)  \, \bigg( \int_t^{\infty} (y - t)^{\frac{q}{r}} w(y)\,dy \bigg)^{\frac{1}{q}} \lim_{x \rightarrow \infty} \bigg(\sup_{x \le \tau} \frac{u(\tau)}{B(\tau)} \bigg( \int_t^{\tau} {\mathcal B}(t,y)  \, dy \bigg)^{\frac{p-r}{pr}}\bigg).
\end{align*}

{\rm (ii)} Let $q < p$. By Theorem \ref{aux.thm.2}, we get for any $t \in (0,\infty)$ that
\begin{align*}
\sup_{\psi \in {\mathfrak M}^+(t,\infty)} \frac{\bigg( \int_t^{\infty}  \bigg( \int_t^z \sup_{x \le \tau} \bigg( \frac{u(\tau)}{B(\tau)} \bigg)^r  \, \int_t^{\tau} \psi (s) \, {\mathcal B}(t,s) \,ds \,dx \bigg)^{\frac{q}{r}} w(z)\,dz\bigg)^{\frac{r}{q}}}{\bigg( \int_t^{\infty} \psi(s)^{\frac{p}{r}}  \, {\mathcal B}(t,s) \,  ds \bigg)^{\frac{r}{p}}} & \\
& \hspace{-9cm} \approx \,   \bigg(\int_t^{\infty} \bigg( \int_t^{\tau} {\mathcal B}(t,s) \,ds \bigg)^{\frac{p(q-r)}{r(p-q)}} {\mathcal B}(t,\tau) \, \bigg( \int_{\tau}^{\infty} \bigg( \int_{\tau}^z \sup_{s \le y} \bigg( \frac{u(y)}{B(y)} \bigg)^r \,ds \bigg)^{\frac{q}{r}} w(z) \,dz \bigg)^{\frac{p}{p-q}} \,d\tau\bigg)^{\frac{r(p-q)}{pq}} \\
& \hspace{-8.5cm} +  \, \bigg(\int_t^{\infty} \bigg( \int_t^{\tau} {\mathcal B}(t,s) \, \bigg( \int_s^{\tau} \sup_{z \le y} \bigg( \frac{u(y)}{B(y)} \bigg)^r \,dz \bigg)^{\frac{p}{p-r}} \,ds \bigg)^{\frac{q(p-r)}{r(p-q)}}  \bigg( \int_{\tau}^{\infty} w(x)\,dx\bigg)^{\frac{q}{p-q}} \, w(\tau)\,d\tau\bigg)^{\frac{r(p-q)}{pq}} \\
& \hspace{-8.5cm} +  \, \bigg(\int_t^{\infty} \bigg( \int_{[x,\infty)} d \bigg( - \bigg( \sup_{y \le \tau} \bigg( \frac{u(\tau)}{B(\tau)} \bigg)^{\frac{pr}{p-r}} \bigg( \int_t^{\tau} {\mathcal B}(t,s)\,ds \bigg) \bigg) \bigg) \bigg)^{\frac{q(p-r)}{r(p-q)}} \bigg( \int_t^x (z - t)^{\frac{q}{r}} w(z) \,dz \bigg)^{\frac{q}{p-q}} (x-t)^{\frac{q}{r}} w(x) \,dx\bigg)^{\frac{r(p-q)}{pq}} \\
& \hspace{-8.5cm} +  \, \bigg(\int_t^{\infty} \bigg( \int_{(t,x]} (y - t)^{\frac{p}{p-r}}d \bigg( - \bigg( \sup_{y \le \tau} \bigg( \frac{u(\tau)}{B(\tau)} \bigg)^{\frac{pr}{p-r}} \bigg( \int_t^{\tau} {\mathcal B}(t,s)\,ds \bigg) \bigg) \bigg) \bigg)^{\frac{q(p-r)}{r(p-q)}} \bigg( \int_x^{\infty} w(z) \,dz \bigg)^{\frac{q}{p-q}} w(x) \,dx\bigg)^{\frac{r(p-q)}{pq}} \\
& \hspace{-8.5cm} +  \, \bigg( \int_t^{\infty} (y - t)^{\frac{q}{r}} w(y)\,dy \bigg)^{\frac{r}{q}} \lim_{x \rightarrow \infty} \bigg(\sup_{x \le \tau} \bigg( \frac{u(\tau)}{B(\tau)} \bigg)^r \bigg( \int_t^{\tau} {\mathcal B}(t,s)  \, ds \bigg)^{\frac{p-r}{p}}\bigg).
\end{align*}
Thus
\begin{align*}
A_2 \approx & \sup_{t \in (0,\infty)} V_2(t)  \sup_{t \le \tau} \bigg( \frac{u(\tau)}{B(\tau)} \bigg) \bigg(\int_t^{\tau} {\mathcal B}(t,s) \, ds \bigg)^{\frac{p-r}{pr}} \bigg(\int_0^{\infty} \bigg(\frac{yt}{y+t}\bigg)^{\frac{q}{r}}w(y)\,dy\bigg)^{\frac{1}{q}}\\
& + \sup_{t \in (0,\infty)} V_2(t) \,\bigg(\int_t^{\infty} \bigg( \int_t^{\tau} {\mathcal B}(t,s) \,ds \bigg)^{\frac{p(q-r)}{r(p-q)}} {\mathcal B}(t,\tau) \, \bigg( \int_{\tau}^{\infty} \bigg( \int_{\tau}^z \sup_{s \le y} \bigg( \frac{u(y)}{B(y)} \bigg)^r \,ds \bigg)^{\frac{q}{r}} w(z) \,dz \bigg)^{\frac{p}{p-q}} \,d\tau\bigg)^{\frac{p-q}{pq}} \\
& + \sup_{t \in (0,\infty)} V_2(t)  \, \bigg(\int_t^{\infty} \bigg( \int_t^{\tau} {\mathcal B}(t,s) \, \bigg( \int_s^{\tau} \sup_{z \le y} \bigg( \frac{u(y)}{B(y)} \bigg)^r \,dz \bigg)^{\frac{p}{p-r}} \,ds \bigg)^{\frac{q(p-r)}{r(p-q)}}  \bigg( \int_{\tau}^{\infty} w(x)\,dx\bigg)^{\frac{q}{p-q}} \, w(\tau)\,d\tau\bigg)^{\frac{p-q}{pq}} \\
& + \sup_{t \in (0,\infty)} V_2(t) \, \bigg(\int_t^{\infty} \bigg( \int_{[x,\infty)} d \bigg( - \bigg( \sup_{y \le \tau} \bigg( \frac{u(\tau)}{B(\tau)} \bigg)^{\frac{pr}{p-r}} \bigg( \int_t^{\tau} {\mathcal B}(t,s)\,ds \bigg) \bigg) \bigg) \bigg)^{\frac{q(p-r)}{r(p-q)}} \bigg( \int_t^x (z - t)^{\frac{q}{r}} w(z) \,dz \bigg)^{\frac{q}{p-q}} (x-t)^{\frac{q}{r}} w(x) \,dx\bigg)^{\frac{p-q}{pq}} \\
& + \sup_{t \in (0,\infty)} V_2(t) \, \bigg(\int_t^{\infty} \bigg( \int_{(t,x]} (y - t)^{\frac{p}{p-r}}d \bigg( - \bigg( \sup_{y \le \tau} \bigg( \frac{u(\tau)}{B(\tau)} \bigg)^{\frac{pr}{p-r}} \bigg( \int_t^{\tau} {\mathcal B}(t,s)\,ds \bigg) \bigg) \bigg) \bigg)^{\frac{q(p-r)}{r(p-q)}} \bigg( \int_x^{\infty} w(z) \,dz \bigg)^{\frac{q}{p-q}} w(x) \,dx\bigg)^{\frac{p-q}{pq}} \\
& + \sup_{t \in (0,\infty)} V_2(t) \, \bigg( \int_t^{\infty} (y - t)^{\frac{q}{r}} w(y)\,dy \bigg)^{\frac{1}{q}} \lim_{x \rightarrow \infty} \bigg(\sup_{x \le \tau} \bigg( \frac{u(\tau)}{B(\tau)} \bigg)\bigg( \int_t^{\tau} {\mathcal B}(t,s)  \, ds \bigg)^{\frac{p-r}{pr}}\bigg).
\end{align*}

Combining, we arrive at
\begin{align*}
A \approx &  \, \sup_{t \in (0,\infty)} \bigg( \int_0^t \bigg( \int_s^t \bigg( \frac{B(y)}{y} \bigg)^{p'}\,dy \bigg)^{\frac{m'}{p'}}
v_2(s)\,ds \bigg)^{\frac{1}{m'}} \bigg( \sup_{t \le \tau} \frac{u(\tau)}{B(\tau)} \bigg) \bigg( \int_0^{\infty} \bigg(\frac{yt}{y+t}\bigg)^{\frac{q}{r}}w(y)\,dy \bigg)^{\frac{1}{q}} \\
& +  \, \sup_{t \in (0,\infty)} \bigg( \int_0^t \bigg( \int_s^t \bigg( \frac{B(y)}{y} \bigg)^{p'}\,dy \bigg)^{\frac{m'}{p'}}
v_2(s)\,ds \bigg)^{\frac{1}{m'}} \bigg(\int_t^{\infty} \bigg( \int_t^y \bigg( \sup_{x \le \tau}\bigg( \frac{u(\tau)}{B(\tau)} \bigg)^r  \bigg) \,dx \bigg)^{\frac{q}{r}}w(y)\,dy\bigg)^{\frac{1}{q}} \\
& + \sup_{t \in (0,\infty)} V_2(t)  \sup_{t \le \tau} \bigg( \frac{u(\tau)}{B(\tau)} \bigg) \bigg(\int_t^{\tau} {\mathcal B}(t,s) \, ds \bigg)^{\frac{p-r}{pr}} \bigg(\int_0^{\infty} \bigg(\frac{yt}{y+t}\bigg)^{\frac{q}{r}}w(y)\,dy\bigg)^{\frac{1}{q}}\\
& + \sup_{t \in (0,\infty)} V_2(t) \,\bigg(\int_t^{\infty} \bigg( \int_t^{\tau} {\mathcal B}(t,s) \,ds \bigg)^{\frac{p(q-r)}{r(p-q)}} {\mathcal B}(t,\tau) \, \bigg( \int_{\tau}^{\infty} \bigg( \int_{\tau}^z \sup_{s \le y} \bigg( \frac{u(y)}{B(y)} \bigg)^r \,ds \bigg)^{\frac{q}{r}} w(z) \,dz \bigg)^{\frac{p}{p-q}} \,d\tau\bigg)^{\frac{p-q}{pq}} \\
& + \sup_{t \in (0,\infty)} V_2(t)  \, \bigg(\int_t^{\infty} \bigg( \int_t^{\tau} {\mathcal B}(t,s) \, \bigg( \int_s^{\tau} \sup_{z \le y} \bigg( \frac{u(y)}{B(y)} \bigg)^r \,dz \bigg)^{\frac{p}{p-r}} \,ds \bigg)^{\frac{q(p-r)}{r(p-q)}}  \bigg( \int_{\tau}^{\infty} w(x)\,dx\bigg)^{\frac{q}{p-q}} \, w(\tau)\,d\tau\bigg)^{\frac{p-q}{pq}} \\
& + \sup_{t \in (0,\infty)} V_2(t) \, \bigg(\int_t^{\infty} \bigg( \int_{[x,\infty)} d \bigg( - \bigg( \sup_{y \le \tau} \bigg( \frac{u(\tau)}{B(\tau)} \bigg)^{\frac{pr}{p-r}} \bigg( \int_t^{\tau} {\mathcal B}(t,s)\,ds \bigg) \bigg) \bigg) \bigg)^{\frac{q(p-r)}{r(p-q)}} \bigg( \int_t^x (z - t)^{\frac{q}{r}} w(z) \,dz \bigg)^{\frac{q}{p-q}} (x-t)^{\frac{q}{r}} w(x) \,dx\bigg)^{\frac{p-q}{pq}} \\
& + \sup_{t \in (0,\infty)} V_2(t) \, \bigg(\int_t^{\infty} \bigg( \int_{(t,x]} (y - t)^{\frac{p}{p-r}}d \bigg( - \bigg( \sup_{y \le \tau} \bigg( \frac{u(\tau)}{B(\tau)} \bigg)^{\frac{pr}{p-r}} \bigg( \int_t^{\tau} {\mathcal B}(t,s)\,ds \bigg) \bigg) \bigg) \bigg)^{\frac{q(p-r)}{r(p-q)}} \bigg( \int_x^{\infty} w(z) \,dz \bigg)^{\frac{q}{p-q}} w(x) \,dx\bigg)^{\frac{p-q}{pq}} \\
& + \sup_{t \in (0,\infty)} V_2(t) \, \bigg( \int_t^{\infty} (y - t)^{\frac{q}{r}} w(y)\,dy \bigg)^{\frac{1}{q}} \lim_{x \rightarrow \infty} \bigg(\sup_{x \le \tau} \bigg( \frac{u(\tau)}{B(\tau)} \bigg)\bigg( \int_t^{\tau} {\mathcal B}(t,s)  \, ds \bigg)^{\frac{p-r}{pr}}\bigg).
\end{align*}

Now we estimate $B$. By Theorem \ref{gks}, (b), and integrating by parts, we have that
\begin{align*}
\sup_{\vp \in {\mathfrak M}^+} \frac{1}{\|\vp\|_{r',h^{1-r'},(0,\infty)}}
\bigg( \int_0^{\infty} \bigg( \int_t^{\infty} \bigg( \int_0^s \vp(x)\,dx  \bigg)^{p'}\,\frac{ds}{s^{p'}} \bigg)^{\frac{m'}{p'}} v_2(t)\,dt \bigg)^{\frac{1}{m'}} & \notag\\
& \hspace{-9cm} \approx \sup_{t \in (0,\infty)}  \bigg( \int_0^t v_2(s)\,ds \bigg)^{\frac{1}{m'}} \bigg( \int_t^{\infty} \bigg( \int_s^{\infty} \frac{dy}{y^{p'}} \,dy \bigg)^{\frac{r(p-1)}{p-r}} \bigg( \int_0^s h(x) u(x)^r \,dx \bigg)^{\frac{r}{p-r}} h(s) u(s)^r \,ds \bigg)^{\frac{p-r}{pr}}  \notag\\
& \hspace{-8.5cm} + \sup_{t \in (0,\infty)}  \bigg( \int_t^{\infty} \bigg( \int_s^{\infty} \frac{dy}{y^{p'}} \bigg)^{\frac{m'}{p'}} v_2(s)\,ds \bigg)^{\frac{1}{m'}}  \bigg( \int_0^t h(x) u(x)^r \,dx \bigg)^{\frac{1}{r}} \\
& \hspace{-9cm} \approx \sup_{t \in (0,\infty)}  \bigg( \int_0^t v_2(s)\,ds \bigg)^{\frac{1}{m'}} \bigg( \int_t^{\infty} s^{\frac{r}{r-p}} \bigg( \int_0^s h(x) u(x)^r \,dx \bigg)^{\frac{r}{p-r}} h(s) u(s)^r \,ds \bigg)^{\frac{p-r}{pr}}  \\
& \hspace{-8.5cm} + \sup_{t \in (0,\infty)}  \bigg( \int_t^{\infty} s^{ - \frac{m'}{p}} v_2(s)\,ds \bigg)^{\frac{1}{m'}}  \bigg( \int_0^t h(x) u(x)^r \,dx \bigg)^{\frac{1}{r}}.
\end{align*}
Since
$$ 
\bigg( \int_{0}^{t} hu^r \bigg)^\frac{p}{p-r} = \int_{0}^{t} d\,\bigg( \int_{0}^{s} hu^r \bigg)^\frac{p}{p-r} \approx \int_{0}^{t} \bigg( \int_{0}^{s} hu^r \bigg)^\frac{r}{p-r} h(s)u(s)^r \, ds,
$$
then we get that
\begin{align*}
\sup_{\vp \in {\mathfrak M}^+} \frac{1}{\|\vp\|_{r',h^{1-r'},(0,\infty)}}
\bigg( \int_0^{\infty} \bigg( \int_t^{\infty} \bigg( \frac{1}{s}  \int_0^s \vp(x) u(x) \,dx  \bigg)^{p'}\,ds \bigg)^{\frac{m'}{p'}} v_2(t)\,dt \bigg)^{\frac{1}{m'}} & \\
& \hspace{-9cm} \approx \sup_{t \in (0,\infty)}  \bigg( \int_0^t v_2(s)\,ds \bigg)^{\frac{1}{m'}} \bigg( \int_t^{\infty} s^{\frac{r}{r-p}} \bigg( \int_0^s h(x) u(x)^r \,dx \bigg)^{\frac{r}{p-r}} h(s) u(s)^r \,ds \bigg)^{\frac{p-r}{pr}}  \\
& \hspace{-8.5cm} + \sup_{t \in (0,\infty)}  \bigg( \int_t^{\infty} s^{ - \frac{m'}{p}} v_2(s)\,ds \bigg)^{\frac{1}{m'}} \bigg(\int_{0}^{t} \bigg( \int_{0}^{s} h(x)u(x)^r \,dx \bigg)^\frac{r}{p-r} h(s)u(s)^r \, ds\bigg)^\frac{p-r}{pr}.
\end{align*}
By Lemma \ref{gluing.lem.0}, we arrive at
\begin{align*}
\sup_{\vp \in {\mathfrak M}^+} \frac{1}{\|\vp\|_{r',h^{1-r'},(0,\infty)}}
\bigg( \int_0^{\infty} \bigg( \int_t^{\infty} \bigg( \frac{1}{s}  \int_0^s \vp(x) u(x) \,dx  \bigg)^{p'}\,ds \bigg)^{\frac{m'}{p'}} v_2(t)\,dt \bigg)^{\frac{1}{m'}} & \\
& \hspace{-9cm} \approx \sup_{t \in (0,\infty)} \left(\int_{0}^{\infty} \left(\frac{t^\frac{1}{p}}{t^\frac{1}{p} + s^\frac{1}{p}}\right)^{m'} v_2(s) \, ds\right)^\frac{1}{m'} \left(\int_{0}^{\infty}\left(\frac{s^\frac{1}{p}}{t^\frac{1}{p} + s^\frac{1}{p}}\right)^\frac{pr}{p-r} \big(s^\frac{1}{p}\big)^{-\frac{pr}{p-r}} \left( \int_{0}^{s} hu^r \right)^\frac{r}{p-r} h(s)u(s)^r \, ds\right)^\frac{p-r}{pr}\\
& \hspace{-9cm}  \approx \sup_{t \in (0,\infty)} \bigg(\int_{0}^{\infty} \bigg(\frac{t}{s+t}\bigg)^{\frac{m'}{p}} v_2(s) \, ds\bigg)^\frac{1}{m'} \bigg(\int_{0}^{\infty}\bigg(\frac{1}{s+t}\bigg)^\frac{r}{p-r} \bigg( \int_{0}^{s} hu^r \bigg)^\frac{r}{p-r} h(s)u(s)^r \, ds\bigg)^\frac{p-r}{pr}.
\end{align*}
Applying Theorem \ref{thm.IBP.0}, it can be seen that if for some $t > 0$ 
$$
\int_{0}^{\infty}\bigg(\frac{1}{s+t}\bigg)^\frac{r}{p-r} \bigg( \int_{0}^{s} hu^r \bigg)^\frac{r}{p-r} h(s)u(s)^r \, ds < \infty, 
$$
then
$$
\int_{0}^{\infty} \bigg( \int_{0}^{s} hu^r \bigg)^\frac{p}{p-r}  d\,\bigg(-\bigg(\frac{1}{s+t}\bigg)^\frac{r}{p-r}\bigg) \le \frac{p}{p-r} \int_{0}^{\infty}\bigg(\frac{1}{s+t}\bigg)^\frac{r}{p-r} \bigg( \int_{0}^{s} hu^r \bigg)^\frac{r}{p-r} h(s)u(s)^r \, ds.
$$
Similarly, if
$$
\int_{0}^{\infty} \bigg( \int_{0}^{s} hu^r \bigg)^\frac{p}{p-r}  d\,\bigg(-\bigg(\frac{1}{s+t}\bigg)^\frac{r}{p-r}\bigg) < \infty
$$
for some $t > 0$, then
$$
\int_{0}^{\infty}\bigg(\frac{1}{s+t}\bigg)^\frac{r}{p-r} \bigg( \int_{0}^{s} hu^r \bigg)^\frac{r}{p-r} h(s)u(s)^r \, ds \le \frac{p-r}{p}\int_{0}^{\infty} \bigg( \int_{0}^{s} hu^r \bigg)^\frac{p}{p-r}  d\,\bigg(-\bigg(\frac{1}{s+t}\bigg)^\frac{r}{p-r}\bigg).
$$
So, we can write
$$
\int_{0}^{\infty}\bigg(\frac{1}{s+t}\bigg)^\frac{r}{p-r} \bigg( \int_{0}^{s} hu^r \bigg)^\frac{r}{p-r} h(s)u(s)^r \, ds 
\approx \int_{0}^{\infty} \bigg( \int_{0}^{s} hu^r \bigg)^\frac{p}{p-r}  d\,\bigg(-\bigg(\frac{1}{s+t}\bigg)^\frac{r}{p-r}\bigg).
$$
We obtain that 
\begin{align*}
\int_{0}^{\infty}\bigg(\frac{1}{s+t}\bigg)^\frac{r}{p-r} \bigg( \int_{0}^{s} hu^r \bigg)^\frac{r}{p-r} h(s)u(s)^r \, ds 
&\approx -\int_{0}^{\infty} \bigg( \int_{0}^{s} hu^r \bigg)^\frac{p}{p-r} \,\bigg(\frac{1}{s+t}\bigg)^\frac{2r-p}{p-r}\bigg(-\frac{1}{(s+t)^2}\bigg)\, ds \\
& = \int_{0}^{\infty} \bigg( \int_{0}^{s} hu^r \bigg)^\frac{p}{p-r} \,\bigg(\frac{1}{s+t}\bigg)^\frac{p}{p-r}\, ds.
\end{align*}
Thus, we have that
\begin{align*}
\sup_{\vp \in {\mathfrak M}^+} \frac{1}{\|\vp\|_{r',h^{1-r'},(0,\infty)}}
\bigg( \int_0^{\infty} \bigg( \int_t^{\infty} \bigg( \frac{1}{s}  \int_0^s \vp(x) u(x) \,dx  \bigg)^{p'}\,ds \bigg)^{\frac{m'}{p'}} v_2(t)\,dt \bigg)^{\frac{1}{m'}}&\\
&\hspace{-7cm} \approx \sup_{t \in (0,\infty)} V_3(t) \, \bigg(\int_{0}^{\infty} \bigg( \int_{0}^{s} hu^r \bigg)^\frac{p}{p-r} \,\bigg(\frac{1}{s+t}\bigg)^\frac{p}{p-r}\, ds\bigg)^\frac{p-r}{pr}.
\end{align*}
Consequenty, we get that
\begin{align*}
B & \approx \sup_{g \in \mp^+} \frac{ 1 }
{\|g\|_{\frac{q}{q-r},w^{\frac{r}{r-q}}}^{\frac{1}{r}}} 
\sup_{h:\, \int_0^x h \le \int_0^x \big( \int_{\tau}^{\infty} g\big)\,d\tau}   
\sup_{t \in (0,\infty)}  V_3(t) \,  \bigg( \int_0^\infty \bigg(\int_0^s h u^r \bigg)^\frac{p}{p-r}\bigg( \frac{1}{s + t}\bigg)^{\frac{p}{p-r}}\,ds \bigg)^{\frac{p-r}{pr}}. 
\end{align*}
Interchanging suprema, by duality (since $ p>r$), we arrive at
\begin{align*}
B & \approx \sup_{g \in \mp^+} \frac{ 1 } {\|g\|_{\frac{q}{q-r},w^{\frac{r}{r-q}}}^{\frac{1}{r}}}    
\sup_{t \in (0,\infty)}  V_3(t) \,
\sup_{h:\, \int_0^x h \le \int_0^x \big( \int_{\tau}^{\infty} g\big)\,d\tau}\bigg( \int_0^\infty \bigg(\int_0^s h u^r \bigg)^\frac{p}{p-r} \bigg( \frac{1}{s + t}\bigg)^{\frac{p}{p-r}}\,ds \bigg)^{\frac{p-r}{pr}} 
\\
& = \sup_{g \in \mp^+} \frac{ 1 } {\|g\|_{\frac{q}{q-r},w^{\frac{r}{r-q}}}^{\frac{1}{r}}}     
\sup_{t \in (0,\infty)}  V_3(t) \,
\left\{ \sup_{\psi \in {\mathfrak M}^+(0,\infty)} \frac{\sup_{h:\, \int_0^x h \le \int_0^x \big( \int_{\tau}^{\infty} g\big)\,d\tau}\int_0^{\infty}  \bigg( \int_0^s h u^r \bigg) \, \psi (s) \,  ds }{\bigg( \int_0^{\infty} \psi(s)^{\frac{p}{r}}  \, (s+t)^{\frac{p}{r}} \,  ds \bigg)^{\frac{r}{p}}} \right\}^{\frac{1}{r}}. 
\end{align*}
Applying Fubini Theorem, by Theorem \ref{transfermon}, we have that
\begin{align*}
B & \approx \sup_{g \in \mp^+} \frac{ 1 } {\|g\|_{\frac{q}{q-r},w^{\frac{r}{r-q}}}^{\frac{1}{r}}}     
\sup_{t \in (0,\infty)}  V_3(t)\,
\left\{ \sup_{\psi \in {\mathfrak M}^+(0,\infty)} \frac{\sup_{h:\, \int_0^x h \le \int_0^x \big( \int_{\tau}^{\infty} g\big)\,d\tau}\int_0^{\infty} h(x) u(x)^r \bigg( \int_x^\infty \psi \bigg) \,dx }{\bigg( \int_0^{\infty} \psi(s)^{\frac{p}{r}}  \, (s+t)^{\frac{p}{r}} \,  ds \bigg)^{\frac{r}{p}}} \right\}^{\frac{1}{r}}\\
& =\sup_{g \in \mp^+} \frac{ 1 } {\|g\|_{\frac{q}{q-r},w^{\frac{r}{r-q}}}^{\frac{1}{r}}}     
\sup_{t \in (0,\infty)}  V_3(t) \,
\left\{ \sup_{\psi \in {\mathfrak M}^+(0,\infty)} \frac{\int_{0}^{\infty}  \bigg( \int_x^{\infty} g\bigg)  \sup_{x \le \tau} u(\tau)^r \, \bigg( \int_{\tau}^{\infty} \psi \bigg)\, dx}{\bigg( \int_0^{\infty} \psi(s)^{\frac{p}{r}}  \, (s+t)^{\frac{p}{r}} \,  ds \bigg)^{\frac{r}{p}}} \right\}^{\frac{1}{r}}.
\end{align*}
Again, by Fubini Theorem, and interchanging suprema, we have that
\begin{align*}
B & \approx \sup_{t \in (0,\infty)}  V_3(t) \,
\left\{ \sup_{\psi \in {\mathfrak M}^+(0,\infty)} \frac{1}{\bigg( \int_0^{\infty} \psi(s)^{\frac{p}{r}}  \, (s+t)^{\frac{p}{r}} \,  ds \bigg)^{\frac{r}{p}}}
\sup_{g \in \mp^+} \frac{ \int_0^{\infty} g(y) \int_0^{y} \sup_{x \le \tau} u(\tau)^r \, \bigg( \int_{\tau}^{\infty} \psi \bigg)\, dx \, dy} {\|g\|_{\frac{q}{q-r},w^{\frac{r}{r-q}}}} \right\}^{\frac{1}{r}}.
\end{align*}
Applying duality principle yields the following estimate:
\begin{align*}
B & \approx \sup_{t \in (0,\infty)}  V_3(t) \,
\left\{ \sup_{\psi \in {\mathfrak M}^+(0,\infty)} \frac{\bigg(\int_0^{\infty} \bigg(\int_0^{y} \sup_{x \le \tau} u(\tau)^r \, \bigg( \int_{\tau}^{\infty} \psi \bigg)\, dx\bigg)^{\frac{q}{r}} w(y)\, dy \bigg)^{\frac{r}{q}}}{\bigg( \int_0^{\infty} \psi(s)^{\frac{p}{r}}  \, (s+t)^{\frac{p}{r}} \,  ds \bigg)^{\frac{r}{p}}}
\right\}^{\frac{1}{r}}.
\end{align*}
Here we apply Theorem \ref{aux.thm.3}.

{\rm (i)} Let $p \le q$. Then we have
\begin{align*}
	B \approx & \sup_{t \in (0,\infty)}  V_3(t) \,
	 \sup_{s \in (0,\infty)} \bigg( \int_0^s (y+t)^{\frac{p(3r -2p)}{(2p-r)(p-r)}} \bigg( \int_y^s \bigg( \sup_{x \le \tau} {\mathcal K}(\tau,t)^r \bigg)\,dx \bigg)^{\frac{p}{p-r}}\,dy\bigg)^{\frac{p-r}{pr}} \, \bigg( \int_s^{\infty} w(z)\,dz \bigg)^{\frac{1}{q}} 
	 \\
	& + \sup_{t \in (0,\infty)}  V_3(t) \, \sup_{s \in (0,\infty)}\bigg(\int_0^s (y+t)^{\frac{p(3r-2p)}{(2p-r)(p-r)}}\,dy\bigg)^\frac{p-r}{pr} \bigg(\int_s^{\infty} \bigg(\int_s^y\bigg(\sup_{x \le \tau} {\mathcal K}(\tau,t)^r \bigg)\,dx\bigg)^\frac{q}{r}\, w(y)dy\bigg)^\frac{1}{q}
	\\
	& + \sup_{t \in (0,\infty)}  V_3(t) \, \sup_{x \in (0,\infty)} \bigg(\int_{[x,\infty)}\,d\bigg(-\sup_{s\le \tau} {\mathcal K}(\tau,t)^\frac{pr}{p-r} \bigg(\int_0^{\tau}(y+t)^{\frac{p(3r-2p)}{(2p-r)(p-r)}}\,dy\bigg)\bigg)\bigg)^\frac{p-r}{pr} \bigg(\int_0^{x}y^\frac{q}{r}w(y)\,dy\bigg)^\frac{1}{q}
	\\
	& + \sup_{t \in (0,\infty)}  V_3(t) \, \sup_{x \in (0,\infty)} \bigg(\int_{(0,x]} s^\frac{p}{p-r}\,d\bigg(-\sup_{s\le \tau} {\mathcal K}(\tau,t)^\frac{pr}{p-r} \bigg(\int_0^{\tau}(y+t)^{\frac{p(3r-2p)}{(2p-r)(p-r)}}\,dy\bigg)\bigg)\bigg)^\frac{p-r}{pr} \bigg(\int_{x}^{\infty}w(y)\,dy\bigg)^\frac{1}{q}
	\\
	& + \bigg(\int_{0}^{\infty}y^\frac{q}{r}w(y)\,dy\bigg)^\frac{1}{q} \,\sup_{t \in (0,\infty)}  V_3(t) \, \lim_{s \rightarrow \infty}\bigg(\sup_{s\le \tau} {\mathcal K}(\tau,t) \bigg(\int_0^{\tau}(y+t)^{\frac{p(3r-2p)}{(2p-r)(p-r)}}\,dy\bigg)^\frac{p-r}{pr}\bigg)
	\\
	& + \sup_{t \in (0,\infty)}  V_3(t) \, t^\frac{r}{p (2p - r)}  \, \bigg(\int_{0}^{\infty}\bigg(\int_{0}^{x}\bigg(\sup_{s\le \tau} {\mathcal K}(\tau,t)^r \bigg)\,ds\bigg)^\frac{q}{r}w(x)\,dx\bigg)^\frac{1}{q}.
\end{align*}

{\rm (ii)} Let $q < p$. Then we have
\begin{align*}
	B \approx & \sup_{t \in (0,\infty)}  V_3(t) \, \bigg(\int_0^{\infty} \bigg( \int_0^x (\tau+t)^{\frac{p(3r - 2p)}{(2p-r)(p-r)}} \,d\tau \bigg)^{\frac{p(q-r)}{r(p-q)}} \, (x+t)^{\frac{p(3r - 2p)}{(2p-r)(p-r)}} \, \bigg( \int_x^{\infty} \bigg( \int_x^z \bigg( \sup_{y \le \tau} {\mathcal K}(\tau,t)^r \bigg) \,dy \bigg)^{\frac{q}{r}} \, w(z)\,dz \bigg)^{\frac{p}{p-q}} \,dx \bigg)^{\frac{p-q}{pq}} \\
	& + \sup_{t \in (0,\infty)}  V_3(t) \, \bigg( \int_0^{\infty} \bigg( \int_0^{y} (x+t)^{\frac{p(3r - 2p)}{(2p - r)(p-r)}} \bigg( \int_x^y \bigg( \sup_{s \le \tau} {\mathcal K}(\tau,t)^r \bigg)\,ds \bigg)^{\frac{p}{p-r}}\,dx \bigg)^{\frac{q(p-r)}{r(p-q)}} \bigg(\int_y^{\infty} w(z)\,dz\bigg)^{\frac{q}{p-q}} w(y)\,dy \bigg)^{\frac{p-q}{pq}} \\
	& + \sup_{t \in (0,\infty)}  V_3(t) \, \bigg( \int_0^{\infty} \bigg( \int_{[x,\infty)} d\,\bigg( - \bigg(\sup_{y \le \tau} {\mathcal K}(\tau,t)^{\frac{pr}{p-r}} \bigg(\int_0^{\tau} (z+t)^{\frac{p(3r-2p)}{(2p - r)(p-r)}}\,dz\bigg)\bigg)\bigg)\bigg)^{\frac{q(p-r)}{r(p-q)}} \bigg(\int_0^x z^{\frac{q}{r}} w(z) \,dz\bigg)^{\frac{q}{p-q}} x^{\frac{q}{r}} w(x)\,dx \bigg)^{\frac{p-q}{pq}} \\
	& + \sup_{t \in (0,\infty)}  V_3(t) \, \bigg( \int_0^{\infty} \bigg( \int_{(0,x]} y^{\frac{p}{p-r}} d \bigg( - \bigg(\sup_{y \le \tau} {\mathcal K}(\tau,t)^{\frac{pr}{p-r}} \bigg( \int_0^{\tau} (z+t)^{\frac{p(3r-2p)}{(2p-r)(p-r)}}\,dz \bigg)\bigg)\bigg)\bigg)^{\frac{q(p-r)}{r(p-q)}} \bigg(\int_x^{\infty} w\bigg)^{\frac{q}{p-q}} w(x)\,dx  \bigg)^{\frac{p-q}{pq}}\\
	& + \bigg(\int_{0}^{\infty}y^\frac{q}{r}w(y)\,dy\bigg)^\frac{1}{q} \,\sup_{t \in (0,\infty)}  V_3(t) \, \lim_{s \rightarrow \infty}\bigg(\sup_{s\le \tau} {\mathcal K}(\tau,t) \bigg(\int_0^{\tau}(y+t)^{\frac{p(3r-2p)}{(2p-r)(p-r)}}\,dy\bigg)^\frac{p-r}{pr}\bigg)
	\\
	& + \sup_{t \in (0,\infty)}  V_3(t) \, t^\frac{r}{p (2p - r)}  \, \bigg(\int_{0}^{\infty}\bigg(\int_{0}^{x}\bigg(\sup_{s\le \tau} {\mathcal K}(\tau,t)^r \bigg)\,ds\bigg)^\frac{q}{r}w(x)\,dx\bigg)^\frac{1}{q}.
\end{align*}

The proof is completed.
\end{proof}

\section{Boundedness of $M_{\phi,\Lambda^{\alpha}(b)}$ from $\GG(p_1,m_1,v)$ into $\GG(p_2,m_2,w)$}\label{BofMF}

In this section we give the proof of the reduction theorem for the boundedness of $M_{\phi,\Lambda^{\alpha}(b)}$ from $\GG(p_1,m_1,v_1)$ into $\GG(p_2,m_2,v_2)$ and calculate the best constant in inequality \eqref{opnorm of M}.

\

\noindent{\bf Proof of  Theorem \ref{main.reduc.thm}:}

Assume that the inequality 	
$$
\bigg( \int_0^{\infty} \bigg( \int_0^x \big[ \big(M_{\phi,\Lambda^{\alpha}(b)}f\big)^* (t)\big]^{p_2}\,dt\bigg)^{\frac{m_2}{p_2}} w(x)\,dx\bigg)^{\frac{1}{m_2}} \le C \bigg( \int_0^{\infty} \bigg( \int_0^x [f^* (\tau)]^{p_1}\,d\tau \bigg)^{\frac{m_1}{p_1}} v(x)\,dx \bigg)^{\frac{1}{m_1}}
$$
holds for all $f \in \mp (\rn)$. 

Denote by $\mf^{\rad,\dn}(\rn)$ the set of all measurable, non-negative,
radially decreasing functions on $\rn$, that is,
$$
\mf^{\rad,\dn}(\rn) : = \{f \in \mf(\rn):\, f(x) = h(|x|),\,x \in \rn
~\mbox{with}~ h \in \mp^+ ((0,\infty);\dn)\},
$$
where the notation $\mp^+ ((0,\infty);\dn)$ is used to denote the
subset of those functions from $\mf^+ (0,\infty)$ which are
non-increasing on $(0,\infty)$.

Recall that the inequality
$$
\big(M_{\phi,\Lambda^{\alpha}(b)}g\big)^* (t) \ge C \, \sup_{t \le \tau} \phi (\tau)^{-1} \bigg( \int_0^{\tau} [g^* (y)]^{\alpha} b(y)\,dy\bigg)^{\frac{1}{\alpha}}
$$	
holds for all $g \in \mf^{\rad,\dn}(\rn)$ with constant $C > 0$ independent of $g$ and $t$ (cf. \cite[Lemma 3.12]{musbil}). 

Thus the inequality
$$
\bigg( \int_0^{\infty} \bigg( \int_0^x \bigg[ \sup_{t \le \tau} \phi (\tau)^{-1} \bigg( \int_0^{\tau} [g^*(y)]^{\alpha} b(y)\,dy\bigg)^{\frac{1}{\alpha}}\bigg]^{p_2}\,dt\bigg)^{\frac{m_2}{p_2}} w(x)\,dx\bigg)^{\frac{1}{m_2}} \le C \bigg( \int_0^{\infty} \bigg( \int_0^x [g^* (\tau)]^{p_1}\,d\tau \bigg)^{\frac{m_1}{p_1}} v(x)\,dx \bigg)^{\frac{1}{m_1}}
$$
holds for all $g \in \mf^{\rad,\dn}(\rn)$, which evidently can be rewritten as follows
$$
\bigg( \int_0^{\infty} \bigg( \int_0^x \big[ (T_{B/\phi^{\alpha},b} g^*) (t) \big]^{\frac{p_2}{\alpha}}\,dt\bigg)^{\frac{m_2}{p_2}} w(x)\,dx\bigg)^{\frac{1}{m_2}} \le C \bigg( \int_0^{\infty} \bigg( \int_0^x [g^* (\tau)]^{\frac{p_1}{\alpha}}\,d\tau \bigg)^{\frac{m_1}{p_1}} v(x)\,dx \bigg)^{\frac{1}{m_1}}, \qquad g \in \mf^{\rad,\dn}(\rn).
$$
Since for any $h \in \mp (\rn)$ there exists $g \in \mf^{\rad,\dn}(\rn)$ such that $g^* = h^*$, then the ineqaulity
$$
\bigg( \int_0^{\infty} \bigg( \int_0^x \big[ (T_{B/\phi^{\alpha},b} h^*) (t) \big]^{\frac{p_2}{\alpha}}\,dt\bigg)^{\frac{m_2}{p_2}} w(x)\,dx\bigg)^{\frac{1}{m_2}} \le C \bigg( \int_0^{\infty} \bigg( \int_0^x [h^* (\tau)]^{\frac{p_1}{\alpha}}\,d\tau \bigg)^{\frac{m_1}{p_1}} v(x)\,dx \bigg)^{\frac{1}{m_1}}
$$
holds for all $h \in \mp (\rn)$, as well.

Now assume that the inequality
$$
\bigg( \int_0^{\infty} \bigg( \int_0^x \big[ (T_{B/\phi^{\alpha},b} h^*) (t) \big]^{\frac{p_2}{\alpha}}\,dt\bigg)^{\frac{m_2}{p_2}} w(x)\,dx\bigg)^{\frac{1}{m_2}} \le C \bigg( \int_0^{\infty} \bigg( \int_0^x [h^* (\tau)]^{\frac{p_1}{\alpha}}\,d\tau \bigg)^{\frac{m_1}{p_1}} v(x)\,dx \bigg)^{\frac{1}{m_1}}
$$
holds for all $h \in \mp (\rn)$.

Obviously, the last inequality is equivalent to the inequality
$$
\bigg( \int_0^{\infty} \bigg( \int_0^x \bigg[ \sup_{t \le \tau} \phi (\tau)^{-1} \bigg( \int_0^{\tau} [f^* (y)]^{\alpha} b(y)\,dy\bigg)^{\frac{1}{\alpha}}\bigg]^{p_2}\,dt\bigg)^{\frac{m_2}{p_2}} w(x)\,dx\bigg)^{\frac{1}{m_2}} \le C \bigg( \int_0^{\infty} \bigg( \int_0^x [f^* (\tau)]^{p_1}\,d\tau \bigg)^{\frac{m_1}{p_1}} v(x)\,dx \bigg)^{\frac{1}{m_1}}
$$
for all $f \in \mp (\rn)$. 

Recall that the inequality
$$
(M_{\phi,\Lambda^{\alpha}(b)}f)^* (t) \le C \sup_{t \le \tau} \phi (\tau)^{-1} \bigg( \int_0^{\tau} [f^* (y)]^{\alpha}b(y)\,dy\bigg)^{\frac{1}{\alpha}}
$$
holds for all $f \in \mp (\rn)$ (cf. \cite[Corollary 3.6]{musbil}).

Consequently, the inequality
$$
\bigg( \int_0^{\infty} \bigg( \int_0^x \big[ \big(M_{\phi,\Lambda^{\alpha}(b)}f\big)^* (t)\big]^{p_2}\,dt\bigg)^{\frac{m_2}{p_2}} w(x)\,dx\bigg)^{\frac{1}{m_2}} \le C \bigg( \int_0^{\infty} \bigg( \int_0^x [f^* (\tau)]^{p_1}\,d\tau \bigg)^{\frac{m_1}{p_1}} v(x)\,dx \bigg)^{\frac{1}{m_1}}
$$
holds for all $f \in \mp (\rn)$, as well.

The proof is completed.

\qed



Denote by 
\begin{equation}\label{v0tilde}
\tilde{v}_0 (t) : = t^{\frac{m_1}{p_1} - 1}\int_0^t v(s)s^{\frac{m_1}{p_1}}\,ds \int_t^{\infty} v(s)\,ds, \qquad t \in (0,\infty),
\end{equation}
\begin{equation}\label{defof_u2}
\tilde{v}_1 (t) : = \int_0^t v(s)s^{\frac{m_1}{p_1}}\,ds + t^{\frac{m_1}{p_1}} \int_t^{\infty} v(s)\,ds, \qquad t \in (0,\infty),
\end{equation}
and
\begin{equation}\label{defof_v22}
\tilde{v}_2(t) : = \frac{t^{\frac{m_1(p_1-\alpha)}{p_1(m_1-\alpha)}}\tilde{v}_0(t)}{\tilde{v}_1(t)^{\frac{2m_1-\alpha}{m_1-\alpha}}}, \qquad t \in (0,\infty).
\end{equation} 

\begin{theorem}
	Let $0 < \alpha < m_1 < p_1 \le p_2 < m_2 < \infty$, $\alpha \le r < \infty$ and $b \in \W (0,\infty) \cap \mp^+ ((0,\infty);\dn)$ be such that the function $B(t)$ satisfies  $0 < B(t) < \infty$ for every $t \in (0,\infty)$, $B \in \Delta_2$, $B(\infty) = \infty$ and $B(t) / t^{\alpha / r}$ is quasi-increasing. Moreover, let $\phi \in \W (0,\infty) \cap C(0,\infty)$ be such that $\phi \in Q_{r}(0,\infty)$ is a quasi-increasing function. Assume that $v \in \W_{m_1,p_1}(0,\infty)$ and $w \in \W(0,\infty)$. Suppose that
	$$
	0 < \int_0^t \bigg( \int_s^t \bigg( \frac{B(y)}{y} \bigg)^{\frac{p_1}{p_1-\alpha}}\,dy \bigg)^{\frac{m_1(p_1-\alpha)}{p_1(m_1-\alpha)}} \tilde{v}_2(s)\,ds < \infty, \qquad t \in (0,\infty),
	$$
	where $\tilde{v}_2$ is defined by \eqref{defof_v22}. Then
	\begin{align*}
	\|M_{\phi,\Lambda^{\alpha}(b)}\|_{\GG(p_1,m_1,v) \rightarrow \GG(p_2,m_2,w)} & \\
	& \hspace{-4cm} \approx  \sup_{t \in (0,\infty)} \bigg( \int_0^t \bigg( \int_s^t \bigg( \frac{B(y)}{y} \bigg)^{\frac{p_1}{p_1-\alpha}}\,dy \bigg)^{\frac{m_1(p_1-\alpha)}{p_1(m_1-\alpha)}}
	\tilde{v}_2(s)\,ds \bigg)^{\frac{m_1-\alpha}{m_1\alpha}} \bigg( \sup_{t \le \tau} \frac{1}{\phi(\tau)} \bigg) \left( \int_0^{\infty} \bigg(\frac{yt}{y+t}\bigg)^{\frac{m_2}{p_2}}w(y)\,dy \right)^{\frac{1}{m_2}} \\
	& \hspace{-3.5cm} +  \, \sup_{t \in (0,\infty)} \bigg( \int_0^t \bigg( \int_s^t \bigg( \frac{B(y)}{y} \bigg)^{\frac{p_1}{p_1-\alpha}}\,dy \bigg)^{\frac{m_1(p_1-\alpha)}{p_1(m_1-\alpha)}}
	\tilde{v}_2(s)\,ds \bigg)^{\frac{m_1-\alpha}{m_1\alpha}} \left( \int_t^{\infty} \bigg( \int_t^y \bigg( \sup_{x \le \tau}\bigg( \frac{1}{\phi(\tau)^{p_2}} \bigg)  \bigg) \,dx \bigg)^{\frac{m_2}{p_2}}w(y)\,dy \right)^{\frac{1}{m_2}} \\
	& \hspace{-3.5cm} + \, \sup_{t \in (0,\infty)}  \bigg( \int_0^{\infty} \bigg( \frac{1}{s + t}\bigg)^{\frac{m_1\alpha}{p_1(m_1-\alpha)}}\tilde{v}_2(s)\,ds \bigg)^{\frac{m_1-\alpha}{m_1\alpha}} \left( \int_0^t \bigg( \int_0^y  \bigg( \sup_{x \le \tau \le t} \bigg(\frac{B(\tau)}{\phi(\tau)^{\alpha}}\bigg)^{\frac{p_2}{\alpha}}  \bigg) \,dx \bigg)^{\frac{m_2}{p_2}} w(y)\,dy \right)^{\frac{1}{m_2}} \\
	& \hspace{-3.5cm} + \,\sup_{t \in (0,\infty)}  \bigg( \int_0^{\infty} \bigg( \frac{1}{s + t}\bigg)^{\frac{m_1\alpha}{p_1(m_1-\alpha)}} \tilde{v}_2(s)\,ds \bigg)^{\frac{m_1-\alpha}{m_1\alpha}} \bigg( \int_0^y  \bigg( \sup_{x \le \tau \le t} \bigg(\frac{B(\tau)}{\phi(\tau)^{\alpha}}\bigg)^{\frac{p_2}{\alpha}}   \bigg) \,dx \bigg)^{\frac{1}{p_2}} \left( \int_t^{\infty}  w(y)\,dy \right)^{\frac{1}{m_2}}. 
	\end{align*}
\end{theorem}

\begin{proof}
	The statement immediately follows by Theorems \ref{main.reduc.thm} and \ref{thm.prev}.	
\end{proof}

\begin{theorem}
	Let $0 < \alpha < m_1 \le p_2 < \min\{p_1,m_2\} < \infty$, $\alpha \le r < \infty$ and $b \in \W (0,\infty) \cap \mp^+ ((0,\infty);\dn)$ be such that the function $B(t)$ satisfies  $0 < B(t) < \infty$ for every $t \in (0,\infty)$, $B \in \Delta_2$, $B(\infty) = \infty$ and $B(t) / t^{\alpha / r}$ is quasi-increasing. Moreover, let $\phi \in \W (0,\infty) \cap C(0,\infty)$ be such that $\phi \in Q_{r}(0,\infty)$ is a quasi-increasing function. Assume that $v \in \W_{m_1,p_1}(0,\infty)$ and $w \in \W(0,\infty)$. 
	Suppose that
	$$
	0 < \int_0^t \bigg( \int_s^t \bigg(\frac{B(y)}{y}\bigg)^{\frac{p_1}{p_1-\alpha}}\,dy \bigg)^{\frac{m_1(p_1-\alpha)}{p_1(m_1-\alpha)}} \tilde{v}_2(s)\,ds < \infty, \qquad t \in (0,\infty),
	$$
	\begin{gather*}
	\int_0^t s^{\frac{m_1(p_1-\alpha)}{p_1(m_1-\alpha)}}\tilde{v_0}(s)\tilde{v}_1(s)^{-\frac{2m_1-\alpha}{m_1-\alpha}}\,ds < \infty, \quad \int_t^{\infty} s^{\frac{m_1(p_1-2\alpha)}{p_1(m_1-\alpha)}}\tilde{v_0}(s)\tilde{v}_1(s)^{-\frac{2m_1-\alpha}{m_1-\alpha}} \,ds < \infty ,\qquad t \in (0,\infty), \\
	\intertext{and} \int_0^1 s^{\frac{m_1(p_1-2\alpha)}{p_1(m_1-\alpha)}}\tilde{v_0}(s)\tilde{v}_1(s)^{-\frac{2m_1-\alpha}{m_1-\alpha}} \,ds = \int_1^{\infty} s^{\frac{m_1(p_1-\alpha)}{p_1(m_1-\alpha)}}\tilde{v_0}(s)\tilde{v}_1(s)^{-\frac{2m_1-\alpha}{m_1-\alpha}}\,ds  =  \infty,
	\end{gather*}
	where $\tilde{v}_0$, $\tilde{v}_1$ and $\tilde{v}_2$ are defined by \eqref{v0tilde}, \eqref{defof_u2} and \eqref{defof_v22}, respectively. 
	Denote by
	$$
	\tilde{V}_2(t) : = \bigg( \int_0^t s^{\frac{m_1(p_1-\alpha)}{p_1(m_1-\alpha)}}\tilde{v_0}(s)\tilde{v}_1(s)^{-\frac{2m_1-\alpha}{m_1-\alpha}}\,ds \bigg)^{\frac{m_1-\alpha}{m_1\alpha}}, \quad \tilde{V}_3(t) = \bigg( \int_0^{\infty} \bigg( \frac{t}{y + t}\bigg)^{\frac{m_1\alpha}{p_1(m_1-\alpha)}}\tilde{v}_2(y)\,dy \bigg)^{\frac{m_1-\alpha}{m_1\alpha}}, \quad 0 < t < \infty,
	$$ 
	$$
	\tilde{\mathcal {B}}(t,s) = \bigg( \int_t^s \bigg( \frac{B(y)}{y} \bigg)^{\frac{p_1}{p_1-\alpha}}\,dy \bigg)^{\frac{p_1(p_2-\alpha)}{(p_1-p_2)\alpha}}  \,  \bigg( \frac{B(s)}{s} \bigg)^{\frac{p_1}{p_1-\alpha}}, \quad 0 < t < s < \infty,
	$$	
	and
	$$
	\tilde{\mathcal K}(\tau,t) = \frac{B(\tau)^{\frac{1}{\alpha}}}{\phi(\tau)} (\tau+t)^{-\frac{2}{2p_1-p_2}}, \quad 0 < \tau,\,t < \infty.
	$$
	
	{\rm (i)} If $p_1 \le m_2$, then
	\begin{align*}
	\|M_{\phi,\Lambda^{\alpha}(b)}\|_{\GG(p_1,m_1,v) \rightarrow \GG(p_2,m_2,w)} & \\
	& \hspace{-4.8cm} \approx \, \sup_{t \in (0,\infty)} \bigg( \int_0^t \bigg( \int_s^t \bigg( \frac{B(y)}{y} \bigg)^{^{\frac{p_1}{p_1-\alpha}}}\,dy \bigg)^{\frac{m_1(p_1-\alpha)}{p_1(m_1-\alpha)}}
	\tilde{v}_2(s)\,ds \bigg)^{\frac{m_1-\alpha}{m_1\alpha}} \bigg( \sup_{t \le \tau} \frac{1}{\phi(\tau)} \bigg) \bigg(\int_0^{\infty} \bigg(\frac{yt}{y+t}\bigg)^{\frac{m_2}{p_2}}w(y)\,dy\bigg)^{\frac{1}{m_2}} \\
	& \hspace{-4.5cm} +  \, \sup_{t \in (0,\infty)} \bigg( \int_0^t \bigg( \int_s^t \bigg( \frac{B(y)}{y} \bigg)^{\frac{p_1}{p_1-\alpha}}\,dy \bigg)^{\frac{m_1(p_1-\alpha)}{p_1(m_1-\alpha)}}
	\tilde{v}_2(s)\,ds \bigg)^{\frac{m_1-\alpha}{m_1\alpha}} \bigg(\int_t^{\infty} \bigg( \int_t^y \bigg( \sup_{x \le \tau} \frac{1}{\phi(\tau)} \bigg)^{p_2} \,dx \bigg)^{\frac{m_2}{p_2}}w(y)\,dy \bigg)^{\frac{1}{m_2}} \\
	& \hspace{-4.5cm} + \sup_{t \in (0,\infty)} \tilde{V}_2(t)\, \sup_{t \le \tau} \frac{1}{\phi(\tau)} \bigg(\int_t^{\tau} \tilde{\mathcal B}(t,s) \, ds \bigg)^{\frac{p_1-p_2}{p_1p_2}}  \bigg(\int_0^{\infty} \bigg(\frac{yt}{y+t}\bigg)^{\frac{m_2}{p_2}}w(y)\,dy\bigg)^{\frac{1}{m_2}} \\
	& \hspace{-4.5cm} + \, \sup_{t \in (0,\infty)} \tilde{V}_2(t)  \, \sup_{x \in (t,\infty)} \bigg( \int_t^x \tilde{\mathcal B}(t,y) \, \bigg( \int_y^x \bigg( \sup_{s \le \tau} \frac{1}{\phi(\tau)} \bigg)^{p_2} \,ds \bigg)^{\frac{p_1}{p_1-p_2}} \, dy \bigg)^{\frac{p_1-p_2}{p_1p_2}} \, \bigg( \int_x^{\infty} w(y) \,dy \bigg)^{\frac{1}{m_2}} \\
	& \hspace{-4.5cm} + \, \sup_{t \in (0,\infty)} \tilde{V}_2(t)  \, \sup_{x \in (t,\infty)} \bigg( \int_t^x \tilde{\mathcal B}(t,y)  \, dy \bigg)^{\frac{p_1-p_2}{p_1p_2}} \, \bigg( \int_x^{\infty} \, \bigg( \int_x^y  \bigg( \sup_{s \le \tau} \frac{1}{\phi(\tau)}\bigg)^{p_2} \,ds \bigg)^{\frac{m_2}{p_2}} w(y) \,dy \bigg)^{\frac{1}{m_2}} \\
	& \hspace{-4.5cm} + \, \sup_{t \in (0,\infty)} \tilde{V}_2(t)  \, \sup_{x \in (t,\infty)} \bigg( \int_{[x,\infty)} \, d \, \bigg( - \sup_{t \le \tau} \bigg(\frac{1}{\phi(\tau)}\bigg)^{\frac{p_1p_2}{p_1-p_2}} \bigg( \int_t^{\tau} \tilde{\mathcal B}(t,y)\,dy \bigg) \bigg) \bigg)^{\frac{p_1-p_2}{p_1p_2}} \, \bigg( \int_t^x (y - t)^{\frac{m_2}{p_2}} w(y) \,dy \bigg)^{\frac{1}{m_2}} \\
	& \hspace{-4.5cm} + \, \sup_{t \in (0,\infty)} \tilde{V}_2(t)  \, \sup_{x \in (t,\infty)} \bigg( \int_{(t,x]} \,  (y - t)^{\frac{p_1}{p_1-p_2}} \, d \, \bigg( - \sup_{y \le \tau} \bigg(\frac{1}{\phi(\tau)}\bigg)^{\frac{p_1p_2}{p_1-p_2}} \bigg( \int_t^{\tau} \tilde{\mathcal B}(t,y)  \, dy \bigg) \bigg) \bigg)^{\frac{p_1-p_2}{p_1p_2}} \, \bigg( \int_x^{\infty} w(y) \,dy \bigg)^{\frac{1}{m_2}} \\
	& \hspace{-4.5cm} + \, \sup_{t \in (0,\infty)} \tilde{V}_2(t)  \, \bigg( \int_t^{\infty} (y - t)^{\frac{m_2}{p_2}} w(y)\,dy \bigg)^{\frac{1}{m_2}} \lim_{x \rightarrow \infty} \bigg(\sup_{x \le \tau} \frac{1}{\phi(\tau)} \bigg( \int_t^{\tau} \tilde{\mathcal B}(t,y)  \, dy \bigg)^{\frac{p_1-p_2}{p_1p_2}}\bigg) \\
	& \hspace{-4.5cm} + \, \sup_{t \in (0,\infty)}  \tilde{V}_3(t) \,
	\sup_{s \in (0,\infty)} \bigg( \int_0^s (y+t)^{\frac{p_1(3p_2 -2p_1)}{(2p_1-p_2)(p_1-p_2)}} \bigg( \int_y^s \sup_{x \le \tau} \tilde{\mathcal K}(\tau,t)^{p_2}\,dx \bigg)^{\frac{p_1}{p_1-p_2}}\,dy\bigg)^{\frac{p_1-p_2}{p_1p_2}} \, \bigg( \int_s^{\infty} w(z)\,dz \bigg)^{\frac{1}{m_2}} \\
	& \hspace{-4.5cm} + \, \sup_{t \in (0,\infty)}  \tilde{V}_3(t) \, \sup_{s \in (0,\infty)}\bigg(\int_0^s (y+t)^{\frac{p_1(3p_2 -2p_1)}{(2p_1-p_2)(p_1-p_2)}}\,dy\bigg)^{\frac{p_1-p_2}{p_1p_2}} \bigg(\int_s^{\infty} \bigg(\int_s^y \sup_{x \le \tau} \tilde{\mathcal K}(\tau,t)^{p_2}\,dx\bigg)^\frac{m_2}{p_2}\, w(y)dy\bigg)^\frac{1}{m_2}
	\\
	& \hspace{-4.5cm} + \, \sup_{t \in (0,\infty)}  \tilde{V}_3(t) \, \sup_{x \in (0,\infty)} \bigg(\int_{[x,\infty)}\,d\bigg(-\sup_{s\le \tau} \tilde{\mathcal K}(\tau,t)^\frac{p_1p_2}{p_1-p_2}  \bigg(\int_0^{\tau}(y+t)^{\frac{p_1(3p_2 -2p_1)}{(2p_1-p_2)(p_1-p_2)}}\,dy\bigg)\bigg)\bigg)^{\frac{p_1-p_2}{p_1p_2}} \bigg(\int_0^{x}y^\frac{m_2}{p_2}w(y)\,dy\bigg)^\frac{1}{m_2}	
	\\
	& \hspace{-4.5cm} + \, \sup_{t \in (0,\infty)}  \tilde{V}_3(t) \, \sup_{x \in (0,\infty)} \bigg(\int_{(0,x]} s^\frac{p_1}{p_1-p_2}\,d\bigg(-\sup_{s\le \tau} \tilde{\mathcal K}(\tau,t)^\frac{p_1p_2}{p_1-p_2} \bigg(\int_0^{\tau}(y+t)^{\frac{p_1(3p_2 -2p_1)}{(2p_1-p_2)(p_1-p_2)}}\,dy\bigg)\bigg)\bigg)^{\frac{p_1-p_2}{p_1p_2}} \bigg(\int_{x}^{\infty}w(y)\,dy\bigg)^\frac{1}{m_2}
	\\
	& \hspace{-4.5cm} + \, \bigg(\int_{0}^{\infty}y^\frac{m_2}{p_2}w(y)\,dy\bigg)^\frac{1}{m_2} \,\sup_{t \in (0,\infty)}  \tilde{V}_3(t) \, \lim_{s \rightarrow \infty}\bigg(\sup_{s\le \tau} \tilde{\mathcal K}(\tau,t) \bigg(\int_0^{\tau}(y+t)^{\frac{p_1(3p_2 -2p_1)}{(2p_1-p_2)(p_1-p_2)}}\,dy\bigg)^{\frac{p_1-p_2}{p_1p_2}}\bigg)	
	\\
	& \hspace{-4.5cm} + \, \sup_{t \in (0,\infty)}  \tilde{V}_3(t) \, t^\frac{p_2}{p_1 (2p_1 - p_2)}  \, \bigg(\int_{0}^{\infty}\bigg(\int_{0}^{x} \sup_{s\le \tau} \tilde{\mathcal K}(\tau,t)^{p_2}\,ds\bigg)^\frac{m_2}{p_2}w(x)\,dx\bigg)^\frac{1}{m_2}.
	\end{align*}
	
	{\rm (ii)} If $m_2 < p_1$, then	
	\begin{align*}
	\|M_{\phi,\Lambda^{\alpha}(b)}\|_{\GG(p_1,m_1,v) \rightarrow \GG(p_2,m_2,w)} & \\ 
	& \hspace{-4.8cm} \approx \sup_{t \in (0,\infty)} \bigg( \int_0^t \bigg( \int_s^t \bigg( \frac{B(y)}{y} \bigg)^{^{\frac{p_1}{p_1-\alpha}}} dy \bigg)^{\frac{m_1(p_1-\alpha)}{p_1(m_1-\alpha)}}
	\tilde{v}_2(s) ds \bigg)^{\frac{m_1-\alpha}{m_1\alpha}} \bigg( \sup_{t \le \tau} \frac{1}{\phi(\tau)} \bigg) \bigg(\int_0^{\infty} \bigg(\frac{yt}{y+t}\bigg)^{\frac{m_2}{p_2}}w(y) dy\bigg)^{\frac{1}{m_2}} \\
	& \hspace{-4.5cm} + \sup_{t \in (0,\infty)} \bigg( \int_0^t \bigg( \int_s^t \bigg( \frac{B(y)}{y} \bigg)^{\frac{p_1}{p_1-\alpha}} dy \bigg)^{\frac{m_1(p_1-\alpha)}{p_1(m_1-\alpha)}}
	\tilde{v}_2(s) ds \bigg)^{\frac{m_1-\alpha}{m_1\alpha}} \bigg(\int_t^{\infty} \bigg( \int_t^y \bigg( \sup_{x \le \tau} \frac{1}{\phi(\tau)}  \bigg)^{p_2} dx \bigg)^{\frac{m_2}{p_2}}w(y) dy \bigg)^{\frac{1}{m_2}} \\
	& \hspace{-4.5cm} + \sup_{t \in (0,\infty)} \tilde{V}_2(t) \sup_{t \le \tau} \bigg( \frac{1}{\phi(\tau)} \bigg) \bigg(\int_t^{\tau} \tilde{\mathcal B}(t,s) ds \bigg)^{\frac{p_1-p_2}{p_1p_2}}  \bigg(\int_0^{\infty} \bigg(\frac{yt}{y+t}\bigg)^{\frac{m_2}{p_2}}w(y) dy\bigg)^{\frac{1}{m_2}} \\
	& \hspace{-4.5cm} + \sup_{t \in (0,\infty)} \tilde{V}_2(t) \bigg( \int_t^{\infty} \bigg( \int_t^{\tau} \tilde{\mathcal B}(t,s) ds \bigg)^{\frac{p_1(m_2-p_2)}{p_2(p_1-m_2)}} \tilde{\mathcal B}(t,\tau) \bigg( \int_{\tau}^{\infty} \bigg( \int_{\tau}^z \bigg( \sup_{x \le y} \frac{1}{\phi(y)} \bigg)^{p_2} dx \bigg)^{\frac{m_2}{p_2}} w(z) dz \bigg)^{\frac{p_1}{p_1-m_2}} d\tau  \bigg)^{\frac{p_1-m_2}{p_1m_2}} 
	\\
	& \hspace{-4.5cm} + \sup_{t \in (0,\infty)} \tilde{V}_2(t)  \bigg( \int_t^{\infty} \bigg( \int_t^{\tau} \tilde{\mathcal B}(t,s) \bigg( \int_s^{\tau} \bigg( \sup_{z \le y} \frac{1}{\phi(y)} \bigg)^{p_2} dz \bigg)^{\frac{p_1}{p_1-p_2}} ds \bigg)^{\frac{m_2(p_1-p_2)}{p_2(p_1-m_2)}}  \bigg( \int_{\tau}^{\infty} w(x) dx\bigg)^{\frac{m_2}{p_1-m_2}} w(\tau) d\tau \bigg)^{\frac{p_1-m_2}{p_1m_2}} 
	\\
	& \hspace{-4.5cm} + \sup_{t \in (0,\infty)} \tilde{V}_2(t) \bigg( \int_t^{\infty} \bigg( \int_{[x,\infty)} d \bigg( - \bigg( \sup_{y \le \tau} \bigg( \frac{1}{\phi(\tau)} \bigg)^{\frac{p_1p_2}{(p_1-p_2)}} \bigg( \int_t^{\tau} \tilde{\mathcal B}(t,s) ds \bigg) \bigg) \bigg) \bigg)^{\frac{m_2(p_1-p_2)}{p_2(p_1-m_2)}} \times \\
	& \hspace{-3cm} \times \bigg( \int_t^x (z - t)^{\frac{m_2}{p_2}} w(z) dz \bigg)^{\frac{m_2}{p_1-m_2}} x^{\frac{m_2}{p_2}} w(x) dx \bigg)^{\frac{p_1-m_2}{p_1m_2}}  
	\\
	& \hspace{-4.5cm} + \sup_{t \in (0,\infty)} \tilde{V}_2(t) \bigg( \int_t^{\infty} \bigg( \int_{(t,x]} (y - t)^{\frac{p_1}{p_1-p_2}}d \bigg( - \bigg( \sup_{y \le \tau} \bigg( \frac{1}{\phi(\tau)} \bigg)^{\frac{p_1p_2}{(p_1-p_2)}} \bigg( \int_t^{\tau} \tilde{\mathcal B}(t,s)\,ds \bigg) \bigg) \bigg) \bigg)^{\frac{m_2(p_1-p_2)}{p_2(p_1-m_2)}} \times \\
	& \hspace{-3cm} \times \bigg( \int_x^{\infty} w(z) dz \bigg)^{\frac{m_2}{p_1-m_2}} w(x) dx \bigg)^{\frac{p_1-m_2}{p_1m_2}} 
	\\
	& \hspace{-4.5cm} + \sup_{t \in (0,\infty)} \tilde{V}_2(t) \bigg( \int_t^{\infty} (y - t)^{\frac{m_2}{p_2}} w(y)dy \bigg)^{\frac{1}{m_2}} \lim_{x \rightarrow \infty} \bigg(\sup_{x \le \tau} \frac{1}{\phi(\tau)} \bigg( \int_t^{\tau} \tilde{\mathcal B}(t,s) ds \bigg)^{\frac{p_1-p_2}{p_1p_2}}\bigg) 
	\\
	& \hspace{-4.5cm} + \sup_{t \in (0,\infty)} \tilde{V}_3(t) \bigg(\int_0^{\infty} \!\bigg( \int_0^x \! (\tau+t)^{\frac{p_1(3p_2 -2p_1)}{(2p_1-p_2)(p_1-p_2)}}d\tau \bigg)^{\frac{p_1(m_2-p_2)}{p_2(p_1-m_2)}} (x+t)^{\frac{p_1(3p_2 -2p_1)}{(2p_1-p_2)(p_1-p_2)}}  \times \\
	& \hspace{-3cm} \times \bigg( \int_x^{\infty} \bigg( \int_x^z \bigg( \sup_{y \le \tau} \tilde{\mathcal K}(\tau,t)^{p_2} \bigg) dy \bigg)^{\frac{m_2}{p_2}} \! w(z)dz \bigg)^{\frac{p_1}{p_1-m_2}} \! dx \bigg)^{\frac{p_1-m_2}{p_1m_2}}  
	\\
	& \hspace{-4.5cm} + \sup_{t \in (0,\infty)} \tilde{V}_3(t)  \bigg( \int_0^{\infty} \bigg( \int_0^{y} (x+t)^{\frac{p_1(3p_2 -2p_1)}{(2p_1-p_2)(p_1-p_2)}} \bigg( \int_x^y \bigg( \sup_{s \le \tau} \tilde{\mathcal K}(\tau,t)^{p_2} \bigg)ds \bigg)^{\frac{p_1}{p_1-p_2}}dx \bigg)^{\frac{m_2(p_1-p_2)}{p_2(p_1-m_2)}} \times \\
	& \hspace{-3cm} \times \bigg(\int_y^{\infty} w(z)dz\bigg)^{\frac{m_2}{p_1-m_2}} w(y)dy \bigg)^{\frac{p_1-m_2}{p_1m_2}}  
	\\
	& \hspace{-4.5cm} + \sup_{t \in (0,\infty)} \tilde{V}_3(t) \bigg( \int_0^{\infty} \! \bigg( \int_{[x,\infty)} \! d\bigg( - \bigg(\sup_{y \le \tau} \tilde{\mathcal K}(\tau,t)^\frac{p_1p_2}{p_1-p_2} \bigg(\int_0^{\tau} (z+t)^{\frac{p_1(3p_2 -2p_1)}{(2p_1-p_2)(p_1-p_2)}} dz\bigg)\bigg)\bigg)\bigg)^{\frac{m_2(p_1-p_2)}{p_2(p_1-m_2)}} \times \\
	& \hspace{-3cm} \times \bigg(\int_0^x z^{\frac{m_2}{p_2}} w(z) dz\bigg)^{\frac{m_2}{p_1-m_2}} x^{\frac{m_2}{p_2}} w(x)dx \bigg)^{\frac{p_1-m_2}{p_1m_2}}  
	\\
	& \hspace{-4.5cm} + \sup_{t \in (0,\infty)} \tilde{V}_3(t) \bigg( \int_0^{\infty} \bigg( \int_{(0,x]} y^{\frac{p_1}{p_1-p_2}} d \bigg( - \bigg(\sup_{y \le \tau} \tilde{\mathcal K}(\tau,t)^\frac{p_1p_2}{p_1-p_2} \bigg( \int_0^{\tau} (z+t)^{\frac{p_1(3p_2 -2p_1)}{(2p_1-p_2)(p_1-p_2)}}dz \bigg)\bigg)\bigg)\bigg)^{\frac{m_2(p_1-p_2)}{p_2(p_1-m_2)}} \times \\
	& \hspace{-3cm} \times \bigg(\int_x^{\infty} w(z)dz\bigg)^{\frac{m_2}{p_1-m_2}} w(x)dx  \bigg)^{\frac{p_1-m_2}{p_1m_2}}  
	\\
	& \hspace{-4.5cm} + \bigg(\int_{0}^{\infty}y^\frac{m_2}{p_2}w(y)dy\bigg)^\frac{1}{m_2} \sup_{t \in (0,\infty)}  \tilde{V}_3(t) \lim_{s \rightarrow \infty}\bigg(\sup_{s\le \tau} \tilde{\mathcal K}(\tau,t) \bigg(\int_0^{\tau}(y+t)^{\frac{p_1(3p_2 -2p_1)}{(2p_1-p_2)(p_1-p_2)}} dy\bigg)^\frac{p_1-p_2}{p_1p_2}\bigg)
	\\
	& \hspace{-4.5cm} + \sup_{t \in (0,\infty)}  \tilde{V}_3(t) t^\frac{p_2}{p_1 (2p_1 - p_2)} \bigg(\int_{0}^{\infty}\bigg(\int_{0}^{x}\bigg(\sup_{s\le \tau} \tilde{\mathcal K}(\tau,t)^{p_2}\bigg) ds\bigg)^\frac{m_2}{p_2}w(x)dx\bigg)^\frac{1}{m_2}.
	\end{align*}
	
\end{theorem}	

\begin{proof}
	The statement immediately follows by Theorems \ref{main.reduc.thm} and \ref{them.prev.2}.	
\end{proof}


\end{document}